\documentclass[12pt,oneside,english]{amsart}
\usepackage[T1]{fontenc}
\usepackage[utf8]{luainputenc}
\usepackage{geometry}
\usepackage{babel}
\usepackage{amsthm}
\usepackage{amssymb}
\usepackage{esint}
\usepackage[numbers]{natbib}
\usepackage[unicode=true,pdfusetitle,
 bookmarks=true,bookmarksnumbered=false,bookmarksopen=false,
 breaklinks=false,pdfborder={0 0 1},backref=false,colorlinks=false]
 {hyperref}

\makeatletter
\numberwithin{equation}{section}
\numberwithin{figure}{section}
\theoremstyle{plain}
\newtheorem{thm}{\protect\theoremname}[section]
  \theoremstyle{remark}
  \newtheorem{rem}[thm]{\protect\remarkname}
  \theoremstyle{plain}
  \newtheorem{lem}[thm]{\protect\lemmaname}
  \theoremstyle{plain}
  \newtheorem{prop}[thm]{\protect\propositionname}

\makeatother

  \providecommand{\lemmaname}{Lemma}
  \providecommand{\propositionname}{Proposition}
  \providecommand{\remarkname}{Remark}
\providecommand{\theoremname}{Theorem}

\begin{document}

\title{Quantization of time-like energy for wave maps into spheres}

\author{Roland Grinis}

\address{Mathematical Institute, University of Oxford, Andrew Wiles Building,
Radcliffe Observatory Quarter, Woodstock Road, Oxford, OX2 6GG, U.K.}

\email{roland.grinis@maths.ox.ac.uk}

\date{9 May 2016}


\begin{abstract}
In this article we consider large energy wave maps in dimension 2+1,
as in the resolution of the threshold conjecture by Sterbenz and Tataru
\citep{TataruSterbenzWave,TataruSterbenzWaveReg}, but more specifically
into the unit Euclidean sphere $\mathbb{S}^{n-1}\subset\mathbb{R}^{n}$
with $n\geq2$, and study further the dynamics of the sequence of
wave maps that are obtained in \citep{TataruSterbenzWaveReg} at the
final rescaling for a first, finite or infinite, time singularity.
We prove that, on a suitably chosen sequence of time slices at this
scaling, there is a decomposition of the map, up to an error with
asymptotically vanishing energy, into a decoupled sum of rescaled
solitons concentrating in the interior of the light cone and a term
having asymptotically vanishing energy dispersion norm, concentrating
on the null boundary and converging to a constant locally in the interior
of the cone, in the energy space. 

Similar and stronger results have been recently obtained in the equivariant
setting by several authors \citep{CKLS1,CKLS2,Catiara,CatiaraCorrigendum,Kenig},
where better control on the dispersive term concentrating on the null
boundary of the cone is provided and in some cases the asymptotic
decomposition is shown to hold for all time. Here however, we do not
impose any symmetry condition on the map itself and our strategy follows
the one from bubbling analysis of harmonic maps into spheres in the
supercritical regime due to Lin and Rivière \citep{Lin,LiR}, which
we make work here in the hyperbolic context of \citep{TataruSterbenzWaveReg}. 
\end{abstract}
\maketitle
\tableofcontents{}

\section{Introduction}

\subsection{Wave maps into spheres.}

We discuss here some facts, important for our argument, regarding
smooth wave maps with target the Euclidean sphere. For a broad introduction
to the subject we shall refer the reader to the monograph of Shatah
and Struwe \citep{ShatahStruweGeomWave}.

Wave maps are smooth maps $\phi:I\times\mathbb{R}^{2}\rightarrow\mathbb{R}^{n}$,
defined on some time interval $I\subset\mathbb{R}$, taking values
in the sphere $\mathbb{S}^{n-1}\subset\mathbb{R}^{n}$, which concretely
means: 
\begin{equation}
\phi^{\dagger}\phi=1,\,\,\,\phi^{\dagger}\nabla_{t,x}\phi=0,\label{eq:GeometricId}
\end{equation}
with the evolution $\phi[t]:=(\phi(t),\partial_{t}\phi(t))\in T(\mathbb{S}^{n-1})$,
taking values in the tangent bundle and belonging to the space $C_{t}^{0}(I\,;\dot{H}_{x}^{1})\cap C_{t}^{1}(I\,;L_{x}^{2})$,
governed by the equation:
\begin{equation}
\square\phi=-\phi\partial_{\alpha}\phi^{\dagger}\partial^{\alpha}\phi,\label{eq:WMeq}
\end{equation}
where the D'Alembertian is given by $\square:=\partial_{\alpha}\partial^{\alpha}=-\partial_{t}^{2}+\Delta_{x}$.
Note our convention here is that we are summing over repeating indices,
where $\alpha$ is running from $0$ to $2$, with $\partial_{0}=\partial_{t}$
and $\partial^{0}=-\partial_{t}$ as we will be always raising the
indices with respect to the Minkowski metric $\mu=-dt^{\otimes2}+dx_{1}^{\otimes2}+dx_{2}^{\otimes2}$
on $\mathbb{R}^{2+1}$ unless clearly stated otherwise. We recall
that equation (\ref{eq:WMeq}) is invariant with respect to the scaling:
\[
\phi(t,x)\longmapsto\phi(\lambda t,\lambda x),
\]
for any $\lambda>0$, and also any space-time translation.

Let us mention a few important conservation laws associated to the
above evolution. Firstly, recall that the energy of a wave map at
time $t_{0}\in I$, scale invariant in dimension 2+1, is given by:
\[
\mathcal{E}[\phi](t_{0}):=\frac{1}{2}\int_{\mathbb{R}^{2}}\left|\partial_{t}\phi(t_{0})\right|^{2}+\left|\nabla_{x}\phi(t_{0})\right|^{2}dx=\frac{1}{2}\left\Vert \nabla_{t,x}\phi(t_{0})\right\Vert _{L_{x}^{2}}^{2},
\]
and a conservation of energy law holds:
\begin{equation}
\mathcal{E}[\phi](t_{0})=\mathcal{E}[\phi](t_{1}),\label{eq:EnergyConservation}
\end{equation}
for any $t_{0},t_{1}\in I$. Secondly, as the target is the Euclidean
sphere $\mathbb{S}^{n-1}$, equation (\ref{eq:WMeq}) is equivalent
to the conservation law:
\begin{equation}
\partial^{\alpha}(\phi\partial_{\alpha}\phi^{\dagger}-\partial_{\alpha}\phi\phi^{\dagger})=0,\label{eq:ConservationLaw}
\end{equation}
which is a consequence of Noether's theorem and the symmetries of
the sphere (and similarly for other homogeneous Riemannian manifolds
but we shall focus on the sphere here for simplicity), recalling that
wave maps are formally critical points of the Lagrangian:
\begin{equation}
\mathcal{L}(\phi):=\int_{\mathbb{R}^{2+1}}\partial^{\alpha}\phi^{\dagger}\partial_{\alpha}\phi dtdx,\label{eq:Lagrangian}
\end{equation}
of which (\ref{eq:WMeq}) is the Euler-Lagrange equation. The use
of (\ref{eq:ConservationLaw}) means however, that some of our arguments
do not directly generalize to the case when one has an arbitrary closed
Riemannian manifold as a target.

Another consequence of the variational point of view and Noether's
theorem, is that smooth wave maps enjoy the stress energy tensor:
\begin{equation}
T_{\alpha\beta}[\phi]:=\partial_{\alpha}\phi^{\dagger}\partial_{\beta}\phi-\frac{1}{2}\mu_{\alpha\beta}\partial^{\gamma}\phi^{\dagger}\partial_{\gamma}\phi,\label{eq:StressEnergyTensor}
\end{equation}
being divergence free:
\begin{equation}
\partial^{\alpha}T_{\alpha\beta}[\phi]=0,\label{eq:StressTensorDiv0}
\end{equation}
and the energy conservation law (\ref{eq:EnergyConservation}) is
in fact obtained by contracting $T[\phi]$ with $\partial_{t}$ and
using (\ref{eq:StressTensorDiv0}) with Stokes' theorem in $[t_{0},t_{1}]\times\mathbb{R}^{2}$.
As we shall see later, many other monotonicity and Morawetz type estimates,
very important in the blow-up analysis of large energy wave maps,
are obtained in this way.

Finally, closing our presentation of wave maps, we remark that the
Lagrangian $\mathcal{L}$ is Lorentz invariant which implies that,
after composition with Lorentz transformations, the map still solves
equation (\ref{eq:WMeq}) and in particular the conservation law (\ref{eq:ConservationLaw})
also stays true.

\subsection{Statement of the main result. }

Before presenting our main result, let us set up some notation. As
usual, for two positive quantities $A$ and $B$ we will be writing
$A\lesssim B$ if $A\leq C\cdot B$ for some implicit constant $C>0$
whose dependence should be clarified when necessary. We also write
$A\sim B$ whenever the additional estimate $B\lesssim A$ holds.
Similarly, for the $O$-notation, we set $A=O(B)$ with $A$ not necessarily
positive this time, if $\left|A\right|\leq C\cdot B$. 

Regarding the asymptotic notation, arising in various statements of
the soliton decomposition below, we write $o_{X}(A)$, as $\nu\rightarrow+\infty$
in the background with $X$ some Banach space (typically a Sobolev
space), for a sequence of elements $f_{\nu}\in X$ with $\left\Vert f_{\nu}\right\Vert _{X}\leq c_{\nu}\cdot A$
where $c_{\nu}\downarrow0$. In the same spirit, we will write $A_{\nu}\ll B_{\nu}$
whenever $A_{\nu}/B_{\nu}\rightarrow0$ holds.

By $B_{r_{0}}(x_{0})\subset\mathbb{R}^{2}$, we will be always referring
to a spatial open ball of radius $r_{0}>0$ and center $x_{0}\in\mathbb{R}^{2}$,
whereas in space-time our basic domains should be light cones. We
denote the forward light cone by:
\[
C:=\left\{ (t,x)\,:\,0\leq t,\, r\leq t\right\} ,\,\,\, r:=\left|x\right|,
\]
and the restriction to some time interval $I$, as well as time sections,
by:
\[
C_{I}:=C\cap(I\times\mathbb{R}^{2}),\,\,\, S_{t_{0}}:=C\cap(\left\{ t_{0}\right\} \times\mathbb{R}^{2}),
\]
respectively, with $\partial C_{I}:=\left\{ (t,x)\,:\, t\in I,\, r=t\right\} $
standing for the \emph{lateral }boundary, to which we usually refer
as the \emph{null boundary}. Given some $\delta>0$, it will be convenient
also to set $C^{\delta}:=(\delta,0)+C$, with the convention that
$C^{0}$ stays for $\cup_{\delta>0}C^{\delta}$, the open interior
of $C$. Accordingly, we have $C_{I}^{\delta}:=C_{I}\cap C^{\delta}$,
$S_{t_{0}}^{\delta}:=S_{t_{0}}\cap C^{\delta}$ and if $\delta>0$,
$\partial C_{I}^{\delta}$ for the lateral boundary of $C_{I}^{\delta}$.

We recall now the set-up from \citep{TataruSterbenzWaveReg} (which
of course holds for any closed Riemannian manifold as target, but
we restrict ourselves to the case of $\mathbb{S}^{n-1}$ for the sake
of consistency). By the finite speed of propagation, translation and
scaling invariance properties, we shall restrict ourselves to the
forward light cone $C$ on which it is convenient to study at the
same time both scenarios: the finite time blow-up at the tip of the
cone, as well as the problem of scattering as $t\rightarrow+\infty$.
Hence, we can assume that we are given a wave map $\phi$ on $C$,
smooth up to but not necessarily including the origin $(0,0)$, and
satisfying the energy bound:
\begin{equation}
\mathcal{E}_{S_{t_{0}}}[\phi]:=\frac{1}{2}\left\Vert \nabla_{t,x}\phi\right\Vert _{L^{2}(S_{t_{0}})}^{2}\leq\mathcal{E},\,\,\,\forall t_{0}\in[0,\infty),\label{eq:GlobalEnergyE}
\end{equation}
where $\mathcal{E}$ is an arbitrarily large but fixed for the rest
of the paper bound on which most of our constants will depend. Let
us introduce here the notation for the energy of the wave map $\phi$
over some domain $U\subset\mathbb{R}^{2+1}$ at the time slice $\{t=t_{0}\}$
setting: 
\[
\mathcal{E}_{U}[\phi](t_{0}):=\frac{1}{2}\int_{U\cap\{t=t_{0}\}}\left|\nabla_{t,x}\phi(t_{0})\right|^{2}dx=\frac{1}{2}\left\Vert \nabla_{t,x}\phi(t_{0})\right\Vert _{L^{2}(U\cap\{t=t_{0}\})}^{2},
\]
or simply $\mathcal{E}_{U}[\phi]$ when there is no ambiguity, as
for example with $\mathcal{E}_{S_{t_{0}}}[\phi]$ above. For the latter
quantity, we recall the important monotonicity property:
\[
\mathcal{E}_{S_{t_{0}}}[\phi]\leq\mathcal{E}_{S_{t_{1}}}[\phi]\,\,\,\mathrm{for}\,\,\, t_{0}\leq t_{1},
\]
which is obtained, as the conservation of energy law (\ref{eq:EnergyConservation}),
contracting the stress energy tensor $T[\phi]$ with $\partial_{t}$
and using (\ref{eq:StressTensorDiv0}) with Stokes' theorem, this
time however applied in $C_{[t_{0},t_{1}]}$, giving:
\begin{equation}
\mathcal{E}_{S_{t_{1}}}[\phi]=\mathcal{F}_{[t_{0},t_{1}]}[\phi]+\mathcal{E}_{S_{t_{0}}}[\phi],\,\,\,\mathcal{F}_{[t_{0},t_{1}]}[\phi]:=\int_{\partial C_{[t_{0},t_{1}]}}\left(\frac{1}{4}\left|L\phi\right|^{2}+\frac{1}{2}\left|r^{-1}\partial_{\theta}\phi\right|^{2}\right)dA,\label{eq:Energy-FluxIdentity}
\end{equation}
where $\mathcal{F}_{[t_{0},t_{1}]}[\phi]$ is called the \emph{flux}
of the wave map from $t_{1}$ to $t_{0}$, and $L$ is part of the
null frame:
\[
L:=\partial_{t}+\partial_{r},\,\,\,\underline{L}:=\partial_{t}-\partial_{r}.
\]

The monotonicity property and the global bound (\ref{eq:GlobalEnergyE})
enable us to define the limits:
\[
\mathcal{E}_{0}:=\lim_{t\downarrow0}\mathcal{E}_{S_{t}}[\phi],\,\,\,\mathcal{E}_{\infty}:=\lim_{t\uparrow\infty}\mathcal{E}_{S_{t}}[\phi],
\]
and imply that $\mathcal{F}_{[t_{0},t_{1}]}[\phi]\downarrow0$ as
$t_{0}$, $t_{1}$ both tend to zero or infinity. The latter can be
used, together with the angular part of $\mathcal{F}_{[t_{0},t_{1}]}[\phi]$
from (\ref{eq:Energy-FluxIdentity}), to construct, given any $\varepsilon>0$,
an extension of $\phi$ outside the cone $C$ on $(0,t_{0}]$ for
$t_{0}=t_{0}(\varepsilon)$ small enough, and on $[t_{\infty},\infty)$
for $t_{\infty}=t_{\infty}(\varepsilon)$ large enough, solving the
wave maps equation (which is possible by finite speed of propagation,
hence we shall slightly abuse notation denoting those extensions by
$\phi$) such that:
\[
\mathcal{E}[\phi](t)-\mathcal{E}_{S_{t}}[\phi]\leq\varepsilon\mathcal{E},\,\,\,\forall t\in(0,t_{0}]\cup[t_{\infty},\infty),
\]
see sections 6.1 and 6.2 in \citep{TataruSterbenzWaveReg}. By the
small energy theorem of Tao \citep{TaoWaveII}, if $\mathcal{E}[\phi](t_{0})$
can be chosen small enough, then $\mathcal{E}_{0}=0$ and $\phi$
can be extended to a smooth wave map for all time (this guarantees
also that the above extensions are smooth everywhere except possibly
$(0,0)$, even if $\mathcal{E}_{S_{t}}[\phi]$ is large, provided
$\varepsilon>0$ was chosen small enough initially). 

Moreover, via a continuity-iteration-renormalization argument, $\phi$
is proved in \citep{TaoWaveII} to belong to a space $S\subset C_{t}^{0}(I\,;\dot{H}_{x}^{1})\cap C_{t}^{1}(I\,;L_{x}^{2})$,
implying control in all the Strichartz spaces amongst others, in which
well-posedness for the Cauchy problem (\ref{eq:WMeq}) can be established.
We discuss this more precisely with further references later in Section
\ref{sub:Regularity-theory-for}. Here, we should mention that, following
the terminology of Sterbenz and Tataru \citep{TataruSterbenzWaveReg},
we will say that \emph{scattering }holds if:
\[
\phi\in S,
\]
noting that, strictly speaking, this means that $\phi$ behaves like
a linear wave as $t\rightarrow\pm\infty$ after applying the microlocal
gauge (if small energy, see \citep{TaoWaveII}) or the diffusion gauge
(necessary if large energy, see \citep{TataruSterbenzWave}). We refer
the reader to the structure theorem of Sterbenz and Tataru in \citep{TataruSterbenzWave},
Proposition 3.9 there, for further information. Let us take the opportunity
here to remark that, if the target manifold is a hyperbolic Riemann
surface, then scattering in the classical sense was established by
Krieger and Schlag \citep{KriegerSchlag} for wave maps in the Coulomb
gauge. For the hyperbolic spaces, this was achieved by Tao \citep{TaoWaveMapProgram}
using the caloric gauge. Therefore, if $\mathcal{E}[\phi](t_{\infty})$
could be chosen small enough for some extension we consider the scattering
problem for $\phi$ as $t\rightarrow+\infty$ resolved.

Once energy gets large, blow-up can occur and the first examples of
finite time singularity for equivariant wave maps into $\mathbb{S}^{2}$
were constructed by Krieger, Schlag and Tataru \citep{KriegerSchlagTataru},
as well as Rodnianski and Sterbenz \citep{RodSterbenz} and also Raphaël
and Rodnianski \citep{RaphaelRod}, where, as for the harmonic map
heat flow, the mechanism behind the singular behavior was concentration
of a non-trivial harmonic map. More generally, the wave map $\phi$
could have concentrated at the origin at least one \emph{soliton}:
these are defined to be finite energy smooth maps $\omega:\mathbb{R}^{2+1}\rightarrow\mathbb{S}^{n-1}$
solving the wave maps equation (\ref{eq:WMeq}) and satisfying:
\[
X\omega=0,
\]
for some constant time-like vector field $X$ on $\mathbb{R}^{2+1}$.
In particular, precomposing $\omega$ with a Lorentz transformation
$\Psi$ that takes $\partial_{t}$ to $X$, we obtain a finite energy
harmonic map from $\mathbb{R}^{2}$ steady in the time direction which,
upon extending over spatial infinity using the removable singularity
theorem of Sacks and Uhlenbeck \citep{SacksUhlenbeck}, gives a harmonic
two-sphere $\omega\circ\Psi:\mathbb{R}\times\mathbb{S}^{2}\rightarrow\mathbb{S}^{n-1}$
familiar from the bubbling analysis of harmonic maps and heat flows.
Let us note here that this last point of view enables us to set $\omega(\infty):=\lim_{|x|\rightarrow\infty}\omega(t,x)$,
which is well-defined and independent of time $t$ chosen.

The \emph{threshold conjecture}, resolved by Sterbenz and Tataru \citep{TataruSterbenzWave,TataruSterbenzWaveReg}
(for closed Riemannian manifolds), Krieger and Schlag \citep{KriegerSchlag}
(for hyperbolic surfaces) and Tao \citep{TaoWaveMapProgram} (for
hyperbolic spaces of any dimension), predicts that concentration of
solitons is the essential mechanism behind blow-up. That is if $\mathcal{E}_{0}$,
$\mathcal{E}_{\infty}$ are less than the energy threshold below which
every harmonic two-sphere is constant, then one has regularity at
$t=0$ and scattering as $t\rightarrow+\infty$. 

One of the central difficulties in establishing this conjecture, in
the general non-symmetric situation, was that relying only on standard
Morawetz type estimates obtained from the stress energy tensor, it
was not possible to get a non-trivial amount of energy concentrating
within the light cone required to produce a non-constant soliton.
As far as the program of Sterbenz and Tataru is concerned, the breakthrough
was made in \citep{TataruSterbenzWave}, where they obtain that, on
top of concentrating energy, the map must concentrate a non-trivial
amount $\epsilon(\mathcal{E})>0$ of the BMO type \emph{energy dispersion
norm}. That is if:
\begin{equation}
\sup_{k}\left\Vert P_{k}\phi\right\Vert _{L_{t,x}^{\infty}((0,t_{0}]\cup[t_{\infty},\infty))}<\epsilon(\mathcal{E}),\label{eq:EnergyDispSmall}
\end{equation}
where $P_{k}$ stands for the Littlewood-Paley projection, then:
\[
\phi\in S((0,t_{0}]\cup[t_{\infty},\infty)),
\]
and the map extends smoothly to a neighborhood of $t=0$ (we shall
state a slightly more precise version of this theorem in Section \ref{sub:Regularity-theory-for}).
This is a large data result and is proved in \citep{TataruSterbenzWave}
via an induction on energy argument.

Let us note here, as an aside, that the program of Krieger and Schlag
\citep{KriegerSchlag}, as well as the one of Tao \citep{TaoWaveMapProgram},
proceeded via a different induction on energy argument and without
any smallness assumption as (\ref{eq:EnergyDispSmall}). As there
are no non-constant solitons for the targets considered there, one
obtains global regularity and scattering for arbitrarily large data
in those cases. We point out on the other hand, that the concentration-compactness
techniques used in \citep{KriegerSchlag} can also lead to a fruitful
study of the formation of solitons, as was demonstrated so far for
equivariant wave maps in \citep{CKLS1,CKLS2,Catiara,CatiaraCorrigendum,Kenig}.
In the present work however, we shall adopt a more direct approach
staying closer to \citep{TataruSterbenzWave,TataruSterbenzWaveReg},
see Section \ref{sub:Strategy} for a detailed summary of our strategy.

In Section \ref{sub:Dispersive-property}, we will briefly discuss
results from \citep{TataruSterbenzWaveReg} that convert concentration
of energy dispersion into concentration of a non-trivial amount of
time-like energy, as this is how, arguing by contradiction, we get
the energy dispersion norm of the term concentrating on the null boundary
asymptotically vanishing. On the other hand, the fact that arguments
in \citep{TataruSterbenzWaveReg} give that only \emph{some} energy
is prevented from escaping into the null boundary at a finite time
singularity, is a serious obstacle to controlling null concentration
further. In fact, techniques dealing with this phenomenon would have
to strengthen \citep{TataruSterbenzWave,TataruSterbenzWaveReg} considerably
in this situation, if not giving a wholly alternative proof to the
threshold conjecture (which we shall not attempt in this paper).
\begin{thm}
\emph{\label{thm:ThresholdConjecture}(Sterbenz and Tataru \citep{TataruSterbenzWave,TataruSterbenzWaveReg}).}
Suppose that the wave map $\phi$ is singular at $(0,0)$, respectively
$\phi\notin S[t_{\infty},\infty)$ for any extension as discussed
above, then there exists a sequence $\lambda_{\nu}^{0}\downarrow0$,
respectively $\lambda_{\nu}^{\infty}\uparrow\infty$, the so-called
\emph{final rescaling}, such that setting:
\[
\phi_{\nu}(\cdot):=\phi(\lambda_{\nu}^{0}\cdot),\,\,\,\mathit{\mathit{respectively}}\,\,\,\phi(\lambda_{\nu}^{\infty}\cdot),
\]
we can find a sequence of concentration points $(t_{\nu},x_{\nu})\in C_{[1,O(1)]}^{\frac{1}{2}}$
and scales $r_{\nu}\downarrow0$, for which:
\[
\phi_{\nu}(t_{\nu}+r_{\nu}t,x_{\nu}+r_{\nu}x)\longrightarrow\omega(t,x)\,\,\,\mathrm{\mathit{in}}\,\,\,(H_{t,x}^{1})_{loc}\left([-\frac{1}{2},\frac{1}{2}]\times\mathbb{R}^{2}\right),
\]
for some non-constant soliton $\omega$.
\end{thm}
We shall describe in detail the final rescaling $\phi_{\nu}$ at the
beginning of Section \ref{sec:Bubbling-analysis}, see Lemma \ref{lem:FinalRescaling}.
In our main theorem, we study this sequence further, carrying out
a blow-up analysis for it and establishing an analogue of the energy
identity from the bubbling analysis of harmonic maps and heat flows
(and many other geometric variational problems), see for example the
works \citep{DingTian,WindingTopping,LauRiv} and the references therein
for the critical regime, and for a supercritical situation the papers
of Lin and Rivière \citep{Lin,LiR}, which are of closer flavor to
the arguments presented in this paper. 
\begin{thm}
\label{thm:Main}Upon passing to a subsequence for the wave maps $\left\{ \phi_{\nu}\right\} _{\nu\in\mathbb{N}}$
obtained in Theorem \ref{thm:ThresholdConjecture}, or abstractly
those satisfying the conclusions of Lemma \ref{lem:FinalRescaling},
we have:

$\bullet$ \emph{Blow-up analysis for asymptotically self-similar
sequences of wave maps:} there exists a non-trivial finite collection
of time-like geodesics $\varrho_{1},\ldots,\varrho_{I}$, emanating
from the origin in Minkowski space $\mathbb{R}^{2+1}$, along which
the maps concentrate some threshold $\epsilon_{s}>0$ of energy:
\[
\liminf_{\nu\rightarrow\infty}\mathcal{E}_{B_{r}(\varrho_{i}(t))}[\phi_{\nu}]>\epsilon_{s}\,\,\,\forall t\in[1,2],\,\,\,\forall r>0,\,\,\, i=1,\ldots,I,
\]
where we are writing $\varrho_{i}(t):=\varrho_{i}\cap S_{t}$, and
the maps converge locally to a constant away from $\varrho_{i}$ in
the interior of the light cone: 
\[
\phi_{\nu}\longrightarrow\mathrm{const.}\,\,\,\mathrm{\mathit{on}}\,\,\, C_{[1,2]}^{0}\setminus\cup_{i}\varrho_{i},
\]
locally in $C_{t}^{0}(H_{x}^{1})\cap C_{t}^{1}(L_{x}^{2})$.

$\bullet$ \emph{Dispersive property for null-concentration:} the
parts of the maps $\phi_{\nu}$ that get concentrated on the null
boundary $\partial C$ have asymptotically vanishing energy dispersion
norm, that is fixing the constant:
\[
\delta_{0}:=\frac{1}{10}\mathrm{dist}(\cup_{i}\varrho_{i},\partial C_{[1,2]}),
\]
the maps $\phi_{\nu}$ on $C_{[t_{0}-\delta_{0},t_{0}+\delta_{0}]}\setminus C_{[t_{0}-\delta_{0},t_{0}+\delta_{0}]}^{2\delta_{0}}$
admit extensions $\varpi_{t_{0},\nu}$ to $[t_{0}-\delta_{0},t_{0}+\delta_{0}]\times\mathbb{R}^{2}$,
for each $t_{0}\in[1+\delta_{0},2-\delta_{0}]$, solving the wave
maps equation on this short, but independent of $\nu$, time interval
and satisfying: 
\[
\nabla_{t,x}\varpi_{t_{0},\nu}\longrightarrow0\,\,\,\mathrm{\mathit{in}}\,\,\, C_{t}^{0}(L_{x}^{2})_{loc}\left(([t_{0}-\delta_{0},t_{0}+\delta_{0}]\times\mathbb{R}^{2})\setminus\partial C_{[t_{0}-\delta_{0},t_{0}+\delta_{0}]}\right),
\]
\[
\mathit{\mathit{and}}\,\,\,\limsup_{\nu\rightarrow\infty}\sup_{k}\left(2^{-k}\left\Vert P_{k}\nabla_{t,x}\varpi_{t_{0},\nu}\right\Vert _{L_{t,x}^{\infty}[t_{0}-\delta_{0},t_{0}+\delta_{0}]}\right)=0;
\]

$\bullet$ \emph{Asymptotic decomposition:} we can find a sequence
of time slices:
\[
\{t_{\nu}\}_{\nu\in\mathbb{N}}\subset[1+\delta_{0},2-\delta_{0}],
\]
on which there exists a non-trivial collection of $J=J(\{t_{\nu}\}_{\nu\in\mathbb{N}})\lesssim_{\mathcal{E}}1$
sequences of points $a_{\nu}^{j}\in\mathbb{R}^{2}$, $|a_{\nu}^{j}|<t_{\nu}-5\delta_{0}$,
with associated scales $\lambda_{\nu}^{j}\downarrow0$ for $j=1,\ldots,J$,
satisfying:
\[
\frac{\lambda_{\nu}^{i}}{\lambda_{\nu}^{j}}+\frac{\lambda_{\nu}^{j}}{\lambda_{\nu}^{i}}+\frac{|a_{\nu}^{i}-a_{\nu}^{j}|^{2}}{\lambda_{\nu}^{i}\lambda_{\nu}^{j}}\longrightarrow\infty
\]
as $\nu\rightarrow+\infty$ for distinct $i\neq j$, such that:
\[
\phi_{\nu}(t,x)=\sum_{j=1}^{J}\left(\omega_{j}\left(\frac{t-t_{\nu}}{\lambda_{\nu}^{j}},\frac{x-a_{\nu}^{j}}{\lambda_{\nu}^{j}}\right)-\omega_{j}(\infty)\right)+\varpi_{t_{\nu},\nu}(t,x)+o_{\dot{H}_{x}^{1}\times L_{x}^{2}}(1)\,\,\,\mathit{\mathit{on}}\,\,\, S_{t_{\nu}},
\]
where $\omega_{j}:\mathbb{R}^{2+1}\rightarrow\mathbb{S}^{n-1}$ are
solitons for which:
\begin{equation}
\phi_{\nu}(t_{\nu}+\lambda_{\nu}^{j}t,a_{\nu}^{j}+\lambda_{\nu}^{j}x)\longrightarrow\omega_{j}(t,x)\,\,\,\mathit{\mathit{on}}\,\,\,\mathbb{R}^{2+1}\setminus\cup_{q}\varrho_{q}^{j},\label{eq:ConvergenceToSolitons}
\end{equation}
locally in $C_{t}^{0}(H_{x}^{1})\cap C_{t}^{1}(L_{x}^{2})$, for a
finite collection, $q=1,\ldots,q(\omega_{j},\mathcal{E})$, of parallel
time-like geodesics $\varrho_{q}^{j}$.\end{thm}
\begin{rem}
\label{EquivariantStaffRem}In other words, we have\emph{ energy quantization}
in the interior of the light cone for wave maps into spheres. This
is a little first step towards understanding the \emph{soliton resolution
conjecture} for the (2+1)-dimensional wave maps equation with target
$\mathbb{S}^{n-1}$. It states that in addition, such a decomposition
should be unique holding for all time and that $\varpi_{t_{0},\nu}$
should have asymptotically vanishing energy in the case of finite
time blow-up (we note that this is guaranteed in the equivariant case
by the well-known exterior energy estimate, see \citep{ShatahStruweGeomWave}),
or correspond to the scattering part of the wave map in the case of
global existence. Some further estimates, following directly from
the work of Sterbenz and Tataru \citep{TataruSterbenzWave,TataruSterbenzWaveReg},
regarding the terms $\varpi_{t_{0},\nu}$ can be found in Remark \ref{RemSBoundOnW}
and Section \ref{sub:Dispersive-property} (for example, (\ref{eq:NewWeightedControl})
there gives decay for the angular and the null $L=\partial_{t}+\partial_{r}$
energy). We note in the end though that our techniques do not lead
to any further information.

We mention here that the soliton resolution conjecture has recently
been shown to hold for the 1-equivariant wave maps into $\mathbb{S}^{2}\subset\mathbb{R}^{3}$
with initial data having topological degree one and energy strictly
less than 3 times $4\pi$ (note that $4\pi$ is the energy threshold)
by Côte, Kenig, Lawrie and Schlag at finite time singularity in \citep{CKLS1},
and in \citep{CKLS2} for the case of global existence (more general
surfaces of revolution are also considered). Note that in this situation,
one knows a priori the uniqueness of the possible configurations of
solitons that can be concentrated (in fact there is only one of them
and it is the unique equivariant degree one harmonic map). The conjecture
is also established for the examples constructed by Krieger, Schlag
and Tataru \citep{KriegerSchlagTataru}, as well as Raphaël and Rodnianski
\citep{RaphaelRod}.

Without this restriction on the initial data, the soliton resolution
along a sequence of times was obtained in the 1-equivariant setting
by Côte \citep{Catiara,CatiaraCorrigendum} building upon \citep{CKLS1,CKLS2},
and more generally for the $\ell$-equivariant case for any integer
$\ell\geq1$ by Jia and Kenig \citep{Kenig} relying on a method different
from \citep{CKLS1,CKLS2,Catiara} (in both works, the finite time
singularity and the global existence case have been considered). We
refer the reader to \citep{Kenig} for more references and an overview
with some history of the various beautiful techniques used to tackle
the soliton resolution conjecture in the radial/equivariant cases
for a variety of non-linear wave equations initiated by Duyckaerts,
Kenig and Merle, see for example \citep{DuyckaertsKenigMerle}. We
also note that those techniques have been very recently applied to
prove the sequential soliton resolution conjecture without any symmetry
assumptions for some focusing semi-linear wave equations by Duyckaerts,
Jia, Kenig and Merle \citep{Jia,KenigMerleSolRes,KenigMerleScatProfile}.
The strategy of the present paper will have a very different flavor
though. An outline can be found in Section \ref{sub:Strategy}. 
\end{rem}
Let us say that the techniques we use to establish the above theorem
leave completely open the question of uniqueness of the set of solitons.
In fact, as suggested by an example of Topping \citep{WindingTopping}
for the harmonic map heat flow, this, and therefore the soliton resolution
conjecture, could fail for certain targets (in view of the work of
Simon \citep{SimonAsymptotics} however, such pathologies are believed
to be excluded when working with real analytic targets like $\mathbb{S}^{n-1}$).
Therefore, there is a notoriously difficult and long way from Theorem
\ref{thm:Main} to the full soliton resolution conjecture as one should
expect the former to hold for any closed Riemannian manifold as a
target and the only place where we use the fact that our target is
a sphere is when relying on the conservation law (\ref{eq:ConservationLaw})
in the proof of the compensation estimates in Section \ref{sub:CompensationEstimate}.
Establishing the analogue of those estimates for general targets is
an important open question even in the elliptic theory, see the work
of Rivière \citep{RiviereConj} for a further discussion.

\subsection{\label{sub:Strategy}Discussion of the strategy.}

We should close the introduction by outlining the proof of Theorem
\ref{thm:Main} which is contained in Section \ref{sec:Bubbling-analysis}. 

The first point of Theorem \ref{thm:Main} is obtained in Section
\ref{sub:Concentration-compactness}. For the sequence of wave maps
$\left\{ \phi_{\nu}\right\} _{\nu\in\mathbb{N}}$ at the final rescaling,
Sterbenz and Tataru \citep{TataruSterbenzWaveReg} obtain a decay
estimate along the scaling vector field $\partial_{\rho}=\frac{1}{(t^{2}-r^{2})^{1/2}}(t\partial_{t}+r\partial_{r})$:
\[
\int\int_{C_{[\varsigma_{\nu},\varsigma_{\nu}^{-1}]}^{\epsilon_{\nu}^{\frac{1}{2}}}}\frac{1}{(t^{2}-r^{2})^{\frac{1}{2}}}\left|\partial_{\rho}\phi_{\nu}\right|^{2}dxdt\longrightarrow0,
\]
for some sequences $\varsigma_{\nu}\downarrow0$, $\epsilon_{\nu}^{\frac{1}{2}}\ll\varsigma_{\nu}$,
see Lemma \ref{lem:FinalRescaling}. If one uses a local version of
the latter, by contracting the stress energy tensor (\ref{eq:StressEnergyTensor})
with $\varphi\partial_{\rho}$, for some compactly supported cut-off
$\varphi$ on the unit hyperbolic plane $\mathbb{H}^{2}$, it is possible
to spread a given energy control on some ball $B_{r_{0}}(x_{0})\Subset S_{1}^{0}$,
at the time slice $t=1$ say, along the flow of the vector field $\partial_{\rho}$
for any finite amount of time; in other words the wave maps $\phi_{\nu}$
would have small energy, uniformly in $\nu$, on the whole of:
\[
\left\{ \lambda z\,:\,\lambda\in[1,2],\, z\in\left\{ t=1\right\} \times B_{r_{0}}(x_{0})\right\} ,
\]
provided they did so initially at $t=1$. This is a simple analogue
of the fact, from the blow-up analysis of supercritical harmonic maps,
that one must have the tangent Radon measures monotone under scaling
(see the work of Lin \citep{Lin}, and Lemma \ref{lem:MonotonicityLemma}
here). 

This way, relying as well on concentration-compactness at $t=1$ and
the small energy compactness result under control of a time-like direction
due to Sterbenz and Tataru \citep{TataruSterbenzWaveReg}, see Lemma
\ref{prop:Simple-compactness-result.} here, we are able to obtain
a subsequence for $\left\{ \phi_{\nu}\right\} _{\nu\in\mathbb{N}}$
which converges on $C_{[1,2]}^{0}$, away from a finite set of time-like
rays passing through the origin, to a regular self-similar wave map
$\phi$. By homogeneity and the singularity removable theorem of Sacks
and Uhlenbeck \citep{SacksUhlenbeck}, the map $\phi$ extends to
a smooth wave map on the whole of the open forward light cone $C^{0}$
(the details of this argument are contained in Lemma \ref{lem:ConcentrationCompactness}).
We note that similar arguments give also the convergence to solitons
statement (\ref{eq:ConvergenceToSolitons}) claimed in Theorem \ref{thm:Main}
(see Lemma \ref{lem:CoveringLemma} for this point). We recall, however,
that self-similar wave maps of finite energy must be constant. This
is a well-known result, the proof of which can be found in \citep{TataruSterbenzWaveReg}
(see also Proposition \ref{prop:SelfSimilar} here for a precise statement). 

On the other hand, another crucial property of the wave maps at the
final scaling of Sterbenz and Tataru \citep{TataruSterbenzWaveReg},
is that a non-trivial amount of energy is uniformly held at a fixed
distance away from the null boundary. Hence, our configuration of
time-like rays, along which the wave maps concentrate, must be non-trivial.
At this stage of the proof, this yields the first point of Theorem
\ref{thm:Main}.

Because only some time-like energy is obtained in \citep{TataruSterbenzWaveReg}
(and this should have been so almost surely, if one considers the
non-scattering problem for example), the second point of Theorem \ref{thm:Main},
treated in Section \ref{sub:Dispersive-property}, tries to address
the issue of null concentration. By cutting the parts of the map concentrating
at the time-like geodesics, we are able to solve the wave maps equation
for a uniform amount of time, even though the energy of the initial
data is a priori large (thanks to the finite speed of propagation
property and the fact the configuration of time-like rays was fixed
initially). Running the arguments of Sterbenz and Tataru \citep{TataruSterbenzWaveReg}
backwards, yields then the claimed control for the energy dispersion
norm (see Lemma \ref{lem:NoTimeLikeEnergyImplyDispersive}).

The construction of the asymptotic decomposition and the proof of
the energy quantization, the third point of Theorem \ref{thm:Main},
is contained in Section \ref{sub:Asymptotic-decomposition.}. Upon
choosing a suitable sequence of time slices $\{t_{\nu}^{(1)}\}_{\nu\in\mathbb{N}}\subset(1,2)$
and scales $\delta_{\nu}\downarrow0$, we study the wave maps:
\[
\phi_{i,\nu}(\cdot):=\phi_{\nu}(t_{\nu}^{(1)}+\delta_{\nu}\cdot,\varrho_{i}(t_{\nu}^{(1)})+\delta_{\nu}\cdot)\,\,\,\mathrm{on}\,\,\,[-1,1]\times B_{1},
\]
for each geodesic $\varrho_{i}$, from the first point of Theorem
\ref{thm:Main}. The maps $\phi_{i,\nu}$ converge to the constant
$c_{\phi}$ corresponding to the self-similar wave map $\phi$ mentioned
previously, locally in $L_{t}^{\infty}(H_{x}^{1}\times L_{x}^{2})$
away from $\varrho_{i}$, and in fact strongly in $L_{t}^{\infty}(L_{x}^{2})$.
The time slices $\{t_{\nu}^{(1)}\}_{\nu\in\mathbb{N}}$ have been
chosen such that: 
\[
X_{i}\phi_{i,\nu}\longrightarrow0\,\,\,\mathrm{in}\,\,\, L_{t,x}^{2},
\]
for the constant time-like vector field $X_{i}$ pointing in the direction
of the ray $\varrho_{i}$. The concentration scales $\{\delta_{\nu}\}_{\nu\in\mathbb{N}}$
have been chosen decaying slowly enough, to avoid losing energy in
the process:
\[
\lim_{\nu\rightarrow\infty}\sup_{t\in[1,2]}\mathcal{E}_{S_{t}^{\delta_{\nu}}\setminus\cup_{i}B_{\delta_{\nu}}(\varrho_{i}(t))}[\phi_{\nu}]=0.
\]

From there, we appeal to the compensation type estimates from Section
\ref{sub:CompensationEstimate} (the only place where we use the fact
that our target is the sphere $\mathbb{S}^{n-1}$), decomposing the
gradient as:
\[
\nabla_{t,x}\phi_{i,\nu}=\Theta_{i,\nu}+\Xi_{i,\nu},
\]
\[
\mathrm{with}\,\,\,\Theta_{i,\nu}\longrightarrow0\,\,\,\mathrm{in}\,\,\, L_{t,x}^{2}\,\,\,\mathrm{and}\,\,\,\sum_{k\in\mathbb{Z}}\left\Vert P_{k}\Xi_{i,\nu}\right\Vert _{L_{t}^{1}(L_{x}^{2})}\lesssim1,
\]
which is obtained in Proposition \ref{prop:Compensation-estimate.}.
To construct $\Theta_{i,\nu}$, we rely essentially on the time-like
decay above, and for $\Xi_{i,\nu}$ the div-curl type structure of
the non-linearity:
\[
\Omega_{\alpha}^{i,\nu}\partial^{\alpha}\phi_{i,\nu},\,\,\,\mathrm{where}\,\,\,\Omega_{\alpha}^{i,\nu}:=\phi_{i,\nu}\partial_{\alpha}\phi_{i,\nu}^{\dagger}-\partial_{\alpha}\phi_{i,\nu}\phi_{i,\nu}^{\dagger},
\]
coming from the conservation law (\ref{eq:ConservationLaw}). Furthermore,
we obtain a decomposition for the higher order time-like derivatives
of $\phi_{i,\nu}$:
\[
X_{i}^{2}\phi_{i,\nu}=\Gamma_{i,\nu}+\Pi_{i,\nu},
\]
where the first term is a linear combination of:
\begin{equation}
\sum_{k\in\mathbb{Z}}P_{k}\nabla_{x}[\Omega_{x}^{i,\nu}(P_{>k+10}\phi_{i,\nu})],\,\,\,\sum_{k\in\mathbb{Z}}P_{k}[\Omega_{x}^{i,\nu}(P_{\leq k+10}\nabla_{x}\phi_{i,\nu})],\,\,\,\mathrm{and}\,\,\,\Omega_{t,x}^{i,\nu}\nabla_{t,x}\phi_{i,\nu},\label{eq:NonLinearBulkOfLemma}
\end{equation}
that we note being local in time and quadratic in the gradient, and
the second one satisfies a favorable decay estimate:

\[
\sum_{k\in\mathbb{Z}}2^{-2k}\left\Vert P_{k}\Pi_{i,\nu}\right\Vert _{L_{t,x}^{2}[-1,1]}^{2}\longrightarrow0.
\]
This is obtained in Lemma \ref{lem:HigherOrderTimeLike} of Section
\ref{sub:CompensationEstimate}, relying crucially on the conservation
law (\ref{eq:ConservationLaw}) again, and plays an important role
in the proof of the Besov decay estimate for wave maps on neck domains
of Lemma \ref{lem:Besov-control.} in Section \ref{sub:Asymptotic-decomposition.},
to which we come in few moments here.

We proceed then by constructing the soliton decomposition for the
wave maps $\phi_{i,\nu}$, up to terms called \emph{necks} in the
literature on harmonic maps, which are given by $\phi_{i,\nu}$ restricted
to a finite collection of conformally degenerating annuli:
\[
[-\frac{r_{i,\nu}^{k}}{2},\frac{r_{i,\nu}^{k}}{2}]\times\left(B_{R_{i,\nu}^{k}}(x_{i,\nu}^{k})\setminus B_{r_{i,\nu}^{k}}(x_{i,\nu}^{k})\right)\subset[-1,1]\times B_{1}\,\,\,\mathit{\mathrm{with}}\,\,\, r_{i,\nu}^{k}\ll R_{i,\nu}^{k}
\]
and $k=1,\ldots,K_{i}(\mathcal{E})$, satisfying the local energy
decay estimate:
\begin{equation}
\sup_{2^{-\ell}r_{i,\nu}^{k}\leq r\leq2^{\ell}R_{i,\nu}^{k}}\sup_{t\in[-\frac{r}{2},\frac{r}{2}]}\mathcal{E}_{B_{2r}(x_{i,\nu}^{k})\setminus B_{r}(x_{i,\nu}^{k})}[\phi_{i,\nu}](t)\longrightarrow0,\label{eq:LocalEnergyIntro}
\end{equation}
for any positive integer $\ell\in\mathbb{N}$. This is the content
of Lemma \ref{lem:CoveringLemma}, and represents essentially a standard
argument of concentration-compactness. The whole of Theorem \ref{thm:Main}
is then reduced to showing that those necks have asymptotically vanishing
energy.

In doing so, upon picking up suitable time slices $\{t_{\nu}^{(2)}\}_{\nu\in\mathbb{N}}\subset(-\frac{1}{2},\frac{1}{2})$
before applying Lemma \ref{lem:CoveringLemma}, and taking the fastest
concentrating scale $\lambda_{\mathrm{min},\nu}:=\min_{i}\{\lambda_{\nu}^{i}\}$,
we consider the maps:
\[
\phi_{\nu,x_{i,\nu}^{k}}(t,x):=\phi_{i,\nu}(t_{\nu}^{(2)}+\lambda_{\mathrm{min},\nu}t,x_{i,\nu}^{k}+\lambda_{\mathrm{min},\nu}x)\,\,\,\mathrm{on}\,\,\,[-1,1]\times\mathbb{R}^{2},
\]
together with: 
\begin{align*}
\Theta_{\nu,x_{i,\nu}^{k}}(t,x):= & \lambda_{\mathrm{min},\nu}\Theta_{i,\nu}(t_{\nu}^{(2)}+\lambda_{\mathrm{min},\nu}t,x_{i,\nu}^{k}+\lambda_{\mathrm{min},\nu}x),\\
\Xi_{\nu,x_{i,\nu}^{k}}(t,x):= & \lambda_{\mathrm{min},\nu}\Xi_{i,\nu}(t_{\nu}^{(2)}+\lambda_{\mathrm{min},\nu}t,x_{i,\nu}^{k}+\lambda_{\mathrm{min},\nu}x),\\
\Pi_{\nu,x_{i,\nu}^{k}}(t,x):= & \lambda_{\mathrm{min},\nu}^{2}\Pi_{i,\nu}(t_{\nu}^{(2)}+\lambda_{\mathrm{min},\nu}t,x_{i,\nu}^{k}+\lambda_{\mathrm{min},\nu}x),
\end{align*}
and $\{t_{\nu}^{(2)}\}_{\nu\in\mathbb{N}}$ was chosen in such a way
that:
\[
\left\Vert \Theta_{\nu,x_{i,\nu}^{k}}(0)\right\Vert _{L_{x}^{2}}+\left\Vert X_{i}\phi_{\nu,x_{i,\nu}^{k}}(0)\right\Vert _{L_{x}^{2}}+\sum_{k\in\mathbb{Z}}2^{-2k}\left\Vert P_{k}\Pi_{\nu,x_{i,\nu}^{k}}(0)\right\Vert _{L_{x}^{2}}^{2}\longrightarrow0.
\]

We use then the second and third items of the decay statement above,
to write for the gradient of $\phi_{\nu,x_{i,\nu}^{k}}$ on the neck
domain:
\[
\nabla_{t,x}\phi_{\nu,x_{i,\nu}^{k}}=\Upsilon_{\nu,x_{i,\nu}^{k}}\,\,\,\mathrm{on}\,\,\,[-1,1]\times(B_{\lambda_{\mathrm{min},\nu}^{-1}R_{i,\nu}^{k}}\setminus B_{\lambda_{\mathrm{min},\nu}^{-1}r_{i,\nu}^{k}}),
\]
with the RHS supported on $[-1,1]\times(B_{2\lambda_{\mathrm{min},\nu}^{-1}R_{i,\nu}^{k}}\setminus B_{2^{-1}\lambda_{\mathrm{min},\nu}^{-1}r_{i,\nu}^{k}})$
and satisfying:
\[
\left\Vert \Upsilon_{\nu,x_{i,\nu}^{k}}\right\Vert _{L_{t}^{\infty}(L_{x}^{2})[-1,1]}\lesssim1,\,\,\,\sup_{k\in\mathbb{Z}}\left\Vert P_{k}\Upsilon_{\nu,x_{i,\nu}^{k}}(0)\right\Vert _{L_{x}^{2}}\longrightarrow0.
\]
This is proved in Lemma \ref{lem:Besov-control.} using the decay
for $X_{i}\phi_{\nu,x_{i,\nu}^{k}}$, localizing to the neck region
the already obtained favorable estimate for $\Pi_{\nu,x_{i,\nu}^{k}}$,
and relying on the local energy control (\ref{eq:LocalEnergyIntro})
to get a weak $\dot{B}_{\infty}^{-1,2}$ decay estimate for the non-linear
terms at high frequency, which are quadratic in the gradient of the
map $\phi_{\nu,x_{i,\nu}^{k}}$ such as (\ref{eq:NonLinearBulkOfLemma})
left over from Lemma \ref{lem:HigherOrderTimeLike}.

Finally, we are brought to the following control for the energy of
$\phi_{\nu,x_{i,\nu}^{k}}$ on the neck domain at time $t=0$:
\begin{align*}
 & \left\Vert \nabla_{t,x}\phi_{\nu,x_{i,\nu}^{k}}(0)\right\Vert _{L_{x}^{2}(B_{\lambda_{\mathrm{min},\nu}^{-1}R_{i,\nu}^{k}}\setminus B_{\lambda_{\mathrm{min},\nu}^{-1}r_{i,\nu}^{k}})}^{2}\\
 & \lesssim(\sup_{k\in\mathbb{Z}}\left\Vert P_{k}\Upsilon_{\nu,x_{i,\nu}^{k}}(0)\right\Vert _{L_{x}^{2}})\sum_{k\in\mathbb{Z}}\left\Vert P_{k}\Xi_{\nu,x_{i,\nu}^{k}}(0)\right\Vert _{L_{x}^{2}}\\
 & +\left\Vert \Upsilon_{\nu,x_{i,\nu}^{k}}(0)\right\Vert _{L_{x}^{2}}\left\Vert \Theta_{\nu,x_{i,\nu}^{k}}(0)\right\Vert _{L_{x}^{2}}+o(1),
\end{align*}
and this gives the desired energy collapsing result.

\section{Technical results\label{sec:Notation-and-technical staff}}

In this section we gather some of the technical results, mainly restricted
to the regularity theory of wave maps, that we will be using in Section
\ref{sec:Bubbling-analysis} to establish Theorem \ref{thm:Main}.
The crucial compensation estimate is proved in Section \ref{sub:CompensationEstimate}.

\subsection{\label{sub:Some-harmonic-analysis.}Some harmonic analysis.}

We will be mainly relying on the spatial Fourier transform. For $\phi(t,x)\in\mathcal{S}(\mathbb{R}^{2})$,
a Schwartz function on $\mathbb{R}^{2}$ at some fixed time $t$,
we define: 
\[
\hat{\phi}(t,\xi):=\int_{\mathbb{R}^{2}}e^{-2\pi ix\cdot\xi}\phi(t,x)dx,
\]
together with the inverse transform given by:
\[
\check{\varphi}(t,x)=\int_{\mathbb{R}^{2}}e^{2\pi ix\cdot\xi}\varphi(t,\xi)d\xi,
\]
for a Schwartz function $\varphi(t,\xi)$ on the frequency space.
The space-time Fourier transform:
\[
\mathcal{F}\psi(\tau,\xi)=\int_{\mathbb{R}^{2}}\int_{\mathbb{R}}e^{-2\pi i(t\tau+x\cdot\xi)}\psi(t,x)dtdx,\,\,\,\psi\in\mathcal{S}(\mathbb{R}\times\mathbb{R}^{2}),
\]
with inverse denoted by $\mathcal{F}^{-1}$, should however appear
in Section \ref{sub:CompensationEstimate} while treating high modulations. 

The use of Littlewood-Paley theory will be quite beneficial to our
analysis and general references for it are the monographs of Taylor
\citep{TaylorTools} and Grafakos \citep{Grafakos}. We shall rely
on the discrete version here only: the Littlewood-Paley projection
$P_{\leq k}$, with $k\in\mathbb{Z}$, is defined to be a Fourier
multiplier with symbol $m_{\leq k}(\xi):=m_{\leq0}(2^{-k}\left|\xi\right|)$,
i.e. via the convolution:
\begin{equation}
P_{\leq k}\phi(t,x):=2^{2k}\int_{\mathbb{R}^{2}}\check{m}_{\leq0}\left(2^{k}(x-y)\right)\phi(t,y)dy,\label{eq:ConvolutionLPprojRepresentation}
\end{equation}
for some radial non-negative function $m_{\leq0}(\left|\xi\right|)$
in frequency space, identically $1$ on $\left|\xi\right|\leq1$ and
$0$ for $\left|\xi\right|\geq2$. 

We also set $P_{k}$ to be a multiplier with symbol $m_{k}(\xi):=m_{0}(2^{-k}\left|\xi\right|)$,
where $m_{0}(\left|\xi\right|):=m_{\leq0}(\left|\xi\right|)-m_{\leq0}(2\left|\xi\right|)$,
and the operators $P_{<k}$, $P_{k_{1}\leq\cdot\leq k_{2}}$, $P_{\geq k}$,
etc. are then defined in the usual way. Note that LP-projections make
sense for functions defined only at some given time $t$, or restricted
to any time interval, and more generally commute with time cut-offs.
Furthermore they are \emph{disposable} \emph{multipliers}, i.e. have
the distributional convolution kernels of bounded mass, even when
considered on the whole of space-time which in practice means that
they are bounded on any translation invariant Banach space of functions
on $\mathbb{R}\times\mathbb{R}^{2}$ and therefore can be discarded
from the estimates as one wishes.

Two elementary but important facts about LP-projections that we would
like to mention here are the \emph{finite band property} that states:
\begin{equation}
\left\Vert \nabla_{x}P_{\leq k}\phi\right\Vert _{L_{x}^{p}}\lesssim2^{k}\left\Vert P_{\leq k}\phi\right\Vert _{L_{x}^{p}},\label{eq:FiniteBandLess}
\end{equation}
and further:
\begin{equation}
\left\Vert \nabla_{x}P_{k}\phi\right\Vert _{L_{x}^{p}}\sim2^{k}\left\Vert P_{k}\phi\right\Vert _{L_{x}^{p}},\label{eq:FiniteBandEqual}
\end{equation}
for any $1\leq p\leq\infty$, as well as \emph{Bernstein's inequality}:
\begin{equation}
\left\Vert P_{k}\phi\right\Vert _{L_{x}^{p}}\lesssim2^{2k\left(\frac{1}{q}-\frac{1}{p}\right)}\left\Vert P_{k}\phi\right\Vert _{L_{x}^{q}},\label{eq:Bernstein}
\end{equation}
for any $1\leq q\leq p\leq\infty$. The latter is especially useful
converting integrability into regularity at low frequencies.

We can decompose any Schwartz function using LP-projections, and as
we typically consider maps taking values in the sphere, we will be
considering affinely (i.e. upon adding a constant) Schwartz functions,
obtaining: 
\begin{equation}
\phi=P_{\leq0}\phi+\sum_{k>0}P_{k}\phi=\mathrm{const.}+\sum_{k\in\mathbb{Z}}P_{k}\phi\,\,\,\mathrm{in}\,\,\,\mathcal{S}(\mathbb{R}^{2}).\label{eq:LPDecomposition}
\end{equation}
While working with the gradient $\nabla_{t,x}\phi$, this will make
no difference of course. By duality, the above decompositions hold
also for tempered distributions and are used to define various Besov
and Triebel-Lizorkin spaces, see \citep{Grafakos}. Let us present
here some examples important for our argument.

In this paper, we will be mainly working with the Besov spaces $B_{q}^{s,p}(\mathbb{R}^{2})$,
for $s\in\mathbb{R}$ and $1\leq p,q\leq\infty$, together with the
homogeneous versions $\dot{B}_{q}^{s,p}(\mathbb{R}^{2})$, defined
as completions with respect to the norms:

\[
\left\Vert \phi\right\Vert _{B_{q}^{s,p}}^{q}:=\left\Vert P_{\leq0}\phi\right\Vert _{L_{x}^{p}}^{q}+\sum_{k>0}2^{qsk}\left\Vert P_{k}\phi\right\Vert _{L_{x}^{p}}^{q},\,\,\,\left\Vert \phi\right\Vert _{\dot{B}_{q}^{s,p}}^{q}:=\sum_{k\in\mathbb{Z}}2^{qsk}\left\Vert P_{k}\phi\right\Vert _{L_{x}^{p}}^{q},
\]
and taking the $\ell^{\infty}$ norm if $q=\infty$ instead, of subspaces
of $\mathcal{S}(\mathbb{R}^{2})$ for which those norms are finite.
We remark that the case $p,q=2$ corresponds to the familiar Sobolev
spaces $H_{x}^{s}$, and their homogeneous versions $\dot{H}_{x}^{s}$
respectively. 

We introduce also the local Hardy space $\mathcal{H}_{loc}^{1}(\mathbb{R}^{2})$
with its homogeneous counterpart $\mathcal{H}^{1}(\mathbb{R}^{2})$,
as Triebel-Lizorkin spaces $F_{2}^{0,1}(\mathbb{R}^{2})=\mathcal{H}_{loc}^{1}(\mathbb{R}^{2})$
and $\dot{F}_{2}^{0,1}(\mathbb{R}^{2})=\mathcal{H}^{1}(\mathbb{R}^{2})$
(this characterization is obtained in \citep{Grafakos}), both subspaces
of $L_{x}^{1}$, defined as the completion of Schwartz functions with
respect to the norms:
\[
\left\Vert \phi\right\Vert _{F_{2}^{0,1}}:=\left\Vert P_{\leq0}\phi\right\Vert _{L_{x}^{1}}+\parallel(\sum_{k\geq1}|P_{k}\phi|^{2})^{1/2}\parallel_{L_{x}^{1}},\,\,\,\left\Vert \phi\right\Vert _{\dot{F}_{2}^{0,1}}:=\parallel(\sum_{k\in\mathbb{Z}}|P_{k}\phi|^{2})^{1/2}\parallel_{L_{x}^{1}},
\]
and which admit the local and homogeneous BMO spaces as a duals, $(\mathcal{H}_{loc}^{1})'=bmo$
and $(\mathcal{H}^{1})'=\mathrm{BMO}$ respectively. Although the
latter does not admit a Littlewood-Paley type characterization, the
former does via the Triebel-Lizorkin space $F_{2}^{0,\infty}=bmo$,
which is defined to be the Banach space of all tempered distributions
$\varphi\in\mathcal{S}'(\mathbb{R}^{2})$ having the following norm
finite:
\[
\left\Vert \varphi\right\Vert _{F_{2}^{0,\infty}}:=\inf_{\left\{ \varphi_{k}\right\} \subset L^{\infty}}\{\left\Vert P_{\leq0}\varphi_{0}\right\Vert _{L^{\infty}}+\parallel(\sum_{k\geq1}|P_{k}\varphi_{k}|^{2})^{1/2}\parallel_{L^{\infty}}\,:\,\varphi=P_{\leq0}\varphi_{0}+\sum_{k\geq1}P_{k}\varphi_{k}\},
\]
the series above required to hold in $\mathcal{S}'$, see the monograph
of Taylor \citep{TaylorTools} for further information. Hardy spaces
are especially useful in estimating paraproducts (see below), and
let us mention here, with this in mind, that $\mathcal{H}^{1}$ embeds
into a Besov space with lower regularity but better summability:
\begin{equation}
\dot{F}_{2}^{0,1}(\mathbb{R}^{2})\subset\dot{B}_{1}^{-1,2}(\mathbb{R}^{2}).\label{eq:HardyBesovEmbedding}
\end{equation}
This fact, that we will enjoy exploiting in the proof of Proposition
\ref{prop:Compensation-estimate.} later, is taken from Lemma 7.19
of Krieger and Schlag \citep{KriegerSchlag} (page 250). For a related
result in the Lorentz space setting see the monograph of Hélein \citep{Helein}
(Theorem 3.3.10 and also the references mentioned there).

Littlewood-Paley decompositions are also very useful in studying non-linear
expressions, and one central example is the product $\theta\vartheta$
of two Schwartz functions $\theta$ and $\vartheta\in\mathcal{S}$.
Applying the decomposition (\ref{eq:LPDecomposition}), we can write:
\[
P_{k}\left(\theta\vartheta\right)=P_{k}\sum_{k_{1},k_{2}}(P_{k_{1}}\theta)(P_{k_{2}}\vartheta),
\]
but recalling that the Fourier transform of a product is a convolution
leads to the so-called \emph{Littlewood-Paley trichotomy} decomposition
(also called paraproduct decomposition), which simplifies the above
double sum into:
\begin{align*}
P_{k}\left(\theta\vartheta\right)= & \, P_{k}[\sum_{k_{1},k_{2}\geq k-6:\left|k_{1}-k_{2}\right|\leq O(1)}(P_{k_{1}}\theta)(P_{k_{2}}\vartheta)\\
 & +(P_{\leq k-7}\theta)(P_{k-3\leq\cdot\leq k+3}\vartheta)\\
 & +(P_{k-3\leq\cdot\leq k+3}\theta)(P_{\leq k-7}\vartheta)],
\end{align*}

$\bullet$\emph{the high-high interactions}: both $\theta$ and $\vartheta$
have Fourier support well above the scale $\left|\xi\right|\sim2^{k}$,
but the only way the sum of two annuli at larger scales $\left|\xi\right|\sim2^{k_{1}},2^{k_{2}}$
with $k_{1},k_{2}\geq k+6$ can intersect the small annulus at $\left|\xi\right|\sim2^{k}$,
is if they are approximately at the same scale, we should have $\left|k_{1}-k_{2}\right|\leq3$.

$\bullet$\emph{the low-high interactions}: if $\theta$ has Fourier
support in the ball of radius $2^{k-6}$, it will contribute to the
frequency scale $\left|\xi\right|\sim2^{k}$ if it is multiplied by
$\vartheta$ frequency localized to the annuli $\left|\xi\right|\sim2^{k_{2}}$
with $k-3\leq k_{2}\leq k+3$. The rougher components of $\vartheta$
bring up the low frequency parts of $\theta$. The sum in $k$ of
the low-high interactions is sometimes called a \emph{paraproduct}
in the literature. By symmetry, we have the same picture with the
roles of $\theta$ and $\vartheta$ interchanged: these are \emph{the
high-low interactions}.

We are then left only with the contribution of $\theta_{k_{1}}\vartheta_{k_{2}}$
where both terms are frequency localized at $2^{k_{1}},2^{k_{2}}\sim2^{k}$,
these are \emph{the} \emph{low-low interactions }and in our case it
will be often convenient to incorporate them in the high-high interactions\emph{.}

Finally, let us set up here the notation for some space-time function
spaces and related tools that we use. We define the Sobolev spaces
$H_{t,x}^{s}=H_{t,x}^{s}(\mathbb{R}\times\mathbb{R}^{2})$, for $s\in\mathbb{R}$,
by using the space-time Fourier transform and taking the completion
of $\mathcal{S}(\mathbb{R}\times\mathbb{R}^{2})$ with respect to
the norm: 
\[
\left\Vert \psi\right\Vert _{H_{t,x}^{s}}:=\left\Vert (1+\tau^{2}+\left|\xi\right|^{2})^{\frac{s}{2}}\mathcal{F}\psi(\tau,\xi)\right\Vert _{L_{t,x}^{2}}.
\]
We define the modulation projections $Q_{\leq j}$ and $Q_{j}$ for
$j\in\mathbb{Z}$ to be the Fourier multipliers with symbols:
\[
m_{0}(|\frac{\left|\tau\right|-\left|\xi\right|}{2^{j}}|)\,\,\,\mathrm{and}\,\,\, m(|\frac{\left|\tau\right|-\left|\xi\right|}{2^{j}}|),
\]
respectively (and similarly for $Q_{<j}$, $Q_{j_{1}\leq\cdot\leq j_{2}}$
and $Q_{\geq j}$). We note that those are not disposable so that
one needs to be careful when discarding them off from the estimates
in general, but as their symbols are bounded and smooth, they are
directly seen to be bounded on $L_{t,x}^{2}$ by Plancherel. Otherwise,
we have the following lemma due to Tao (Lemmata 3 and 4 in \citep{TaoWaveII}).
\begin{lem}
\label{lem:Dispo}The operators $P_{k}Q_{j}$, $P_{k}Q_{\leq j}$,
$P_{\leq k}Q_{\leq j}$ and $P_{\leq k}Q_{j}$ are disposable for
any pair of integers $j$ and $k$ with $j\geq k+O(1)$. Moreover,
for any $1\leq p\leq\infty$ and $j,j_{1},j_{2}\in\mathbb{Z}$, the
operators $Q_{\leq j}$, $Q_{j_{1}\leq\cdot\leq j_{2}}$ and $Q_{j}$
are bounded on the spaces $L_{t}^{p}(L_{x}^{2})$.
\end{lem}
Using the modulation projections $Q_{j}$, we define following Tao
\citep{TaoWaveII} the homogeneous $\dot{X}_{k}^{s,b,q}$ spaces associated
to the cone $\left\{ \left|\tau\right|=\left|\xi\right|\right\} $
at the spatial frequency scale $k$, for any fixed integer $k\in\mathbb{Z}$
and some given real $b\in\mathbb{R}$, to be the completion of the
space of Schwartz functions $\psi$ on $\mathbb{R}\times\mathbb{R}^{2}$
with respect to the norm:
\[
\left\Vert \psi\right\Vert _{\dot{X}_{k}^{s,b,q}}:=2^{sk}\left[\sum_{j}2^{qbj}\left\Vert Q_{j}P_{k}\psi\right\Vert _{L_{t,x}^{2}}^{q}\right]^{\frac{1}{q}},
\]
provided the latter is finite for $\psi$, and adopting the usual
convention if $q$ is infinite. For $q=1$ we obtain an atomic space.
As our methods here have more of an elliptic rather than dispersive
character in the end, we shall not use those spaces directly (other
than stating the estimates from regularity theory). However, the distinction
between the \emph{high modulations} regime $P_{k}Q_{>k+10}$, and
the one of \emph{frequency space-like} $P_{k}Q_{\leq k+10}$, is absolutely
crucial for our analysis.

To close this section, let us recall here the convention that function
spaces over domains are defined via minimal extensions. For example,
we shall write $X(I)$, where $X$ is a function space over $\mathbb{R}\times\mathbb{R}^{2}$
and $I$ some time interval, for the Banach space of functions $f$
in $I\times\mathbb{R}^{2}$ admitting an extension $f'$ to the whole
of $\mathbb{R}\times\mathbb{R}^{2}$ and set:
\[
\left\Vert f\right\Vert _{X(I)}:=\inf\left\{ \left\Vert f'\right\Vert _{X}\,:\, f'\in X,\,\,\, f'=f\,\,\,\mathrm{on}\,\,\, I\times\mathbb{R}^{2}\right\} .
\]

\subsection{\label{sub:Regularity-theory-for}Regularity theory for wave maps.}

We shall not give here the full definition of the space $DS$, and
its undifferentiated version $S$, used in the iteration arguments
of the proofs of well-posedness for the wave maps equation, referring
to \citep{TaoWaveII} section 10 or \citep{TataruSterbenzWave} section
5.2, but we will briefly summarize here some characteristic properties. 

At a given frequency scale $k\in\mathbb{Z}$, the space $DS$ is defined
as an intersection of several different spaces and for us it will
be enough to note that we have the control:
\begin{equation}
\left\Vert P_{k}\psi\right\Vert _{L_{t}^{\infty}(L_{x}^{2})}+\left\Vert P_{k}\psi\right\Vert _{\dot{X}_{k}^{0,\frac{1}{2},\infty}}+\sup_{(q,r):\frac{1}{q}+\frac{1}{2r}\leq\frac{1}{4}}2^{(\frac{1}{q}+\frac{2}{r}-1)k}\left\Vert P_{k}\psi\right\Vert _{L_{t}^{q}(L_{x}^{r})}\leq\left\Vert P_{k}\psi\right\Vert _{DS},\label{eq:DefOfDS}
\end{equation}
for any Schwartz function $\psi$ on $\mathbb{R}\times\mathbb{R}^{2}$
(under frequency localization, for the space $S$ we have $P_{k}\phi\in S$
if $\nabla_{t,x}P_{k}\phi\in DS$ for a Schwartz $\phi$). The first
component is the natural energy component on which we should mainly
rely in this work. The second one is the dispersive component to be
used only indirectly here but being important in gaining extra regularity
for the part of the wave map that has Fourier support away from the
light cone. The latter observation is exploited by Sterbenz and Tataru
\citep{TataruSterbenzWaveReg} in their compactness result that we
discuss below. The third component represents the standard Strichartz
spaces. We note that we do obtain the null concentration terms $\varpi_{t_{0},\nu}$
lying in this space, see Remark \ref{RemSBoundOnW}.

We note that, for the regularity theory, the $Q_{0}$-null structure
in the non-linearity of equation (\ref{eq:WMeq}) is crucial and the
components mentioned above are not enough by themselves to exploit
it so that one needs to introduce further suitable null frame Strichartz
spaces. However, as this structure will not play any direct role in
our arguments we should not elaborate more on this point here. Let
us simply remark in the end that $DS$ contains the atomic Fourier
restriction space:
\begin{equation}
\left\Vert P_{k}\psi\right\Vert _{DS}\lesssim\left\Vert P_{k}\psi\right\Vert _{\dot{X}_{k}^{0,\frac{1}{2},1}},\label{eq:XsbControl}
\end{equation}
referring to Lemma 8 in Tao's paper \citep{TaoWaveII} for the proof
of this fact, ideas from which we should actually use later in the
proof of Lemma \ref{lem:HigherOrderTimeLike}.

By default in \citep{TataruSterbenzWave}, the authors define then
the spaces $DS$ and $S$ as completions of Schwartz functions in
$\mathbb{R}\times\mathbb{R}^{2}$ with respect to the norms obtained
by $\ell^{2}$-summing the control on the LP-projections and adding
the $L^{\infty}$ norm for $S$: 
\begin{equation}
\left\Vert \psi\right\Vert _{DS}^{2}:=\sum_{k\in\mathbb{Z}}\left\Vert P_{k}\psi\right\Vert _{DS}^{2},\,\,\,\left\Vert \phi\right\Vert _{S}^{2}:=\left\Vert \phi\right\Vert _{L_{t,x}^{\infty}}^{2}+\sum_{k\in\mathbb{Z}}\left\Vert \nabla_{t,x}P_{k}\phi\right\Vert _{DS}^{2}.\label{eq:L2SummationDS}
\end{equation}
In practice however, it is sometimes convenient to replace the $\ell^{2}$
summation in (\ref{eq:L2SummationDS}) with a control with respect
to a \emph{frequency envelope}. Following Sterbenz and Tataru \citep{TataruSterbenzWave},
we call a sequence $c:=\left\{ c_{k}\right\} _{k\in\mathbb{Z}}\in\ell^{2}$
of positive numbers $c_{k}>0$ a $(\sigma_{0},\sigma_{1})$-admissible
frequency envelope if $0<\sigma_{0}<\sigma_{1}$ and for any $k_{0}<k_{1}$
we have:
\[
2^{-\sigma_{0}(k_{1}-k_{0})}c_{k_{1}}\leq c_{k_{0}}\leq2^{\sigma_{1}(k_{1}-k_{0})}c_{k_{1}}.
\]
Given some smooth initial data $\phi[0]=(\phi(0),\partial_{t}\phi(0))$
we can naturally attach to it an admissible frequency envelope by
setting:
\begin{equation}
c_{k}^{2}=\sum_{k_{0}<k}2^{-2\sigma_{1}(k-k_{0})}\left\Vert P_{k_{0}}\nabla_{t,x}\phi(0)\right\Vert _{L_{x}^{2}}^{2}+\sum_{k_{1}\geq k}2^{-2\sigma_{0}(k_{1}-k)}\left\Vert P_{k_{1}}\nabla_{t,x}\phi(0)\right\Vert _{L_{x}^{2}}^{2},\label{eq:FrequencyEnvelope}
\end{equation}
for which we note that:
\begin{equation}
\left(\sum_{k\in\mathbb{Z}}2^{2\sigma}c_{k}^{2}\right)^{\frac{1}{2}}\sim\left\Vert \nabla_{t,x}\phi(0)\right\Vert _{\dot{H}_{x}^{\sigma}},\,\,\,-\sigma_{0}<\sigma<\sigma_{1},\label{eq:SobolevFrequencyControl}
\end{equation}
so that given any function $\psi$ on $\mathbb{R}^{2}$, $\left\Vert P_{k}\psi\right\Vert _{L_{x}^{2}}\lesssim c_{k}$
implies:
\[
\left\Vert \psi\right\Vert _{\dot{H}_{x}^{\sigma}}\lesssim\left\Vert \nabla_{t,x}\phi(0)\right\Vert _{\dot{H}_{x}^{\sigma}},\,\,\,-\sigma_{0}<\sigma<\sigma_{1},
\]
which is very useful in controlling the regularity of an evolution
like the wave map.

Well-posedness theory for the wave maps equation with small energy
initial data is due to Tao \citep{TaoWaveII} and Tataru \citep{TataruWave},
and also Krieger \citep{KriegerSmallEnergy} who considered the hyperbolic
plane as target. We will be using here a local version that we state
below appearing as Theorem 1.3 in \citep{TataruWave}. Of course,
all of the results stated in this section are true for general closed
Riemannian manifolds as target, but we present them in the case of
spheres for the sake of consistency.
\begin{thm}
\label{thm:Small-data-waves}\emph{(Tao \citep{TaoWaveII}, Tataru
\citep{TataruWave}).} There exists a constant $\epsilon_{0}:=\epsilon_{0}(\mathbb{S}^{n-1})>0$
such that: 

$\bullet$ \emph{Regularity:} given some smooth initial data $\phi[0]\in T(\mathbb{S}^{n-1})$
at time $t=0$ constant outside a compact domain with energy:
\[
\mathcal{E}[\phi](0)<\epsilon_{0},
\]
there exists a unique smooth wave map $\phi$ defined on the whole
of Minkowski space $\mathbb{R}^{2+1}$ such that:
\begin{equation}
\left\Vert P_{k}\phi\right\Vert _{S}\lesssim c_{k},\label{eq:SmallDataRegularity}
\end{equation}
taking the frequency envelope $c$ from (\ref{eq:FrequencyEnvelope})
for $\phi[0]$ and where $\sigma_{0}=\sigma_{0}(\mathbb{S}^{n-1})$
is some fixed small positive constant but $\sigma_{1}$ can be chosen
arbitrarily large;

$\bullet$ \emph{Continuous dependence on initial data and rough solutions:
}given a sequence of smooth tuples $\phi_{\nu}[0]\in T(\mathbb{S}^{n-1})$
of initial data equal to a fixed constant outside some fixed compact
domain, with energy:
\[
\mathcal{E}[\phi_{\nu}](0)<\epsilon_{0},
\]
and converging strongly in $H_{x}^{1}\times L_{x}^{2}$ to some $\phi[0]$,
there exist smooth wave maps $\phi_{\nu}$ with the properties as
stated in the first point above and a map:
\[
\phi\in S,
\]
solving weakly the wave maps equation (\ref{eq:WMeq}), to which $\phi_{\nu}$
converge in $C_{t}^{0}(H_{x}^{1})\cap C_{t}^{1}(L_{x}^{2})$ on bounded
time intervals, and further: 
\[
\nabla_{t,x}\phi_{\nu}\rightarrow\nabla_{t,x}\phi\,\,\,\mathit{in}\,\,\, DS(\mathbb{R}^{2+1}).
\]

\end{thm}
We state now a compactness result due to Sterbenz and Tataru \citep{TataruSterbenzWaveReg}
for a sequence of small energy wave maps which become constant in
the direction of some smooth time-like vector field. The absence of
such a result in the general small energy case is precisely what makes
the study of wave maps near the null boundary of the light cone a
very challenging affair, requiring global non-linear techniques going
beyond the present article. We mention that the arguments in \citep{TataruSterbenzWaveReg}
rely on the elliptic flavor given to the situation by the assumption
that the sequence is asymptotically constant along a time-like vector
field, the use of the Fourier restriction component of $DS$ to gain
compactness and regularity for the limiting map, as well as the small
energy weak stability theory developed by Tataru \citep{TataruWave}
(which we have presented in the second point of Theorem \ref{thm:Small-data-waves}
here).
\begin{lem}
\label{prop:Simple-compactness-result.}\emph{(Sterbenz and Tataru
\citep{TataruSterbenzWaveReg}).} Consider a sequence of smooth wave
maps $\phi_{\nu}$ in $[-3,3]\times B_{3}$ with small energy: 
\begin{equation}
\sup_{t\in[-3,3]}\mathcal{E}_{B_{3}}[\phi_{\nu}](t)\leq\epsilon_{s},\label{eq:EnergyBoundCompactnessLemma}
\end{equation}
where $\epsilon_{s}>0$ depends only on $\epsilon_{0}$ from Theorem
\ref{thm:Small-data-waves}, and such that:
\begin{equation}
\left\Vert X\phi_{\nu}\right\Vert _{L_{t,x}^{2}([-3,3]\times B_{3})}\longrightarrow0,\label{eq:AsymptoticallyTimeLike}
\end{equation}
for some smooth time-like vector field $X$. Then there exists a wave
map:
\begin{equation}
\phi\in H_{t,x}^{\frac{3}{2}-\epsilon}([-1,1]\times B_{1}),\label{eq:LimitRegularity}
\end{equation}
for any $0<\epsilon<\frac{1}{2}$, satisfying:
\[
X\phi=0\,\,\,\mathit{on}\,\,\,[-1,1]\times B_{1},
\]
to which the maps $\phi_{\nu}$ converge in $C_{t}^{0}(H_{x}^{1})\cap C_{t}^{1}(L_{x}^{2})$
after passing to a subsequence, and further:
\begin{equation}
\nabla_{t,x}\phi_{\nu}\longrightarrow\nabla_{t,x}\phi\,\,\,\mathit{in}\,\,\, DS(\left\{ t\in[-1,1],\, r\leq2-\left|t\right|\right\} ).\label{eq:CompactnessConvergence}
\end{equation}
\end{lem}
\begin{rem}
\label{Rem:UniformityInTime}The proof of this lemma can be found
in Proposition 5.1 of \citep{TataruSterbenzWaveReg} and we remark
that convergence in $H_{t,x}^{1}(U)$ for any domain $U\Subset(-3,3)\times B_{3}$
only is claimed there. But the stronger statement (\ref{eq:CompactnessConvergence}),
to be understood in terms of minimal extensions, can be obtained as
follows. Let us fix $U=[-\frac{5}{2},\frac{5}{2}]\times B_{5/2}$,
then upon passing to a further subsequence we would have: 
\begin{equation}
\left\Vert \phi_{\nu}(t)-\phi(t)\right\Vert _{L_{x}^{2}(B_{\frac{5}{2}})}^{2}+\left\Vert \nabla_{t,x}\phi_{\nu}(t)-\nabla_{t,x}\phi(t)\right\Vert _{L_{x}^{2}(B_{\frac{5}{2}})}^{2}\longrightarrow0\,\,\,\mathrm{for}\,\,\,\mathrm{a.e.}\,\,\, t\,,\label{eq:STRONG_CONVERGENCE}
\end{equation}
therefore $\phi_{\nu}$ converge strongly to $\phi$ in $(H_{x}^{1}\times L_{x}^{2})(B_{5/2})$
for almost every $t$ that we can fix as close to $0$ as we wish.
Hence, assuming that $\epsilon_{s}$ was chosen small enough initially,
by the pigeonhole principle we have for $\sigma\in(2,\frac{5}{2})$:
\[
\int_{\partial B_{\sigma}}\left|\nabla_{t,x}\phi(t)\right|^{2}d\theta\lesssim\epsilon_{s},
\]
away from a set of measure $\frac{1}{10}$ say. Fixing such a $\sigma$,
we would have $\phi(t,\partial B_{\sigma})$ contained in a single
chart of $\mathbb{S}^{n-1}$ of diameter $O(\sqrt{\epsilon_{s}})$
around a point $c\in\mathbb{S}^{n-1}$. Moreover, upon passing to
a further subsequence, by the strong convergence (\ref{eq:STRONG_CONVERGENCE})
we can choose $\sigma\in(2,\frac{5}{2})$ such that $\phi_{\nu}(t)|_{\partial B_{\sigma}}\rightarrow\phi(t)|_{\partial B_{\sigma}}$
in the Hölder space $C^{\alpha}(\partial B_{\sigma})$ with $\alpha\in(0,\frac{1}{2})$,
using Morrey's inequality. Hence, we would have $\phi_{\nu}(t,\partial B_{\sigma})$
contained in the chart around $c\in\mathbb{S}^{n-1}$of diameter $O(\sqrt{\epsilon_{s}})$
as well, for all $\nu\in\mathbb{N}$ large enough. Therefore, we can
construct extensions $\phi'_{\nu}[t]\in T(\mathbb{S}^{n-1})$ of $\phi_{\nu}[t]|_{B_{\sigma}}$,
smooth as the latter are, with the energy bound:
\[
\mathcal{E}[\phi'_{\nu}](t)\lesssim\epsilon_{s},
\]
by smoothly interpolating between $\phi_{\nu}[t]|_{\partial B_{\sigma}}$
and $(c,0)\in T(\mathbb{S}^{n-1})$ on $B_{3}\setminus B_{\sigma}$.
By construction, we obtain $\phi'_{\nu}[t]$ strongly convergent in
$H_{x}^{1}\times L_{x}^{2}$ to some map $\phi'[t]$ agreeing with
$\phi[t]$ on $B_{2}$. In the end, setting the constant $\epsilon_{s}>0$
small enough and the time $t$ close enough to $0$, the convergence
statements are justified by the continuous dependence on the initial
data part of Theorem \ref{thm:Small-data-waves} and the finite speed
of propagation property. 

In particular, the assumption (\ref{eq:AsymptoticallyTimeLike}) gets
upgraded to:
\[
X\phi_{\nu}\longrightarrow0\,\,\,\mathrm{in}\,\,\, C_{t}^{0}(L_{x}^{2})\left([-1,1]\times B_{1}\right),
\]
and going further, the regularity theory of Theorem \ref{thm:Small-data-waves}
tells us that in fact we have:
\[
\phi\in C_{t}^{0}([-1,1]\,;H_{x}^{\frac{3}{2}-\epsilon}(B_{1}))\cap C_{t}^{1}([-1,1]\,;H_{x}^{\frac{1}{2}-\epsilon}(B_{1})),
\]
for any $0<\epsilon<\frac{1}{2}$  improving upon (\ref{eq:LimitRegularity}),
although it is unfortunately impossible to obtain convergence in such
a stronger space without further assumptions, especially regarding
the decay (\ref{eq:AsymptoticallyTimeLike}). 
\end{rem}
Let us close this section by mentioning the result of Sterbenz and
Tataru \citep{TataruSterbenzWave}, see both Theorem 1.3 and Proposition
3.9 there, which relaxes the assumption of small energy in the work
of Tao \citep{TaoWaveII} and Tataru \citep{TataruWave} to small
\emph{energy dispersion}. This represents a crucial technical ingredient
in the proof by Sterbenz and Tataru \citep{TataruSterbenzWaveReg}
of the threshold conjecture. Let us consider an open interval $I=(t_{0},t_{1})$,
which can be unbounded. 
\begin{thm}
\label{thm:EnergyDispersedWM}\emph{(Sterbenz and Tataru \citep{TataruSterbenzWave}).}
Given an energy bound $\mathcal{E}>0$, there exist constants $0<\epsilon(\mathbb{S}^{n-1},\mathcal{E})\ll1$
and $1\ll F(\mathbb{S}^{n-1},\mathcal{E})$ such that for any smooth
wave map $\phi$ on $(t_{0},t_{1})$ with energy bounded by $\mathcal{E}$
and $\nabla_{t,x}\phi$ spatially Schwartz, if we have:
\[
\sup_{k}\left\Vert P_{k}\phi\right\Vert _{L_{t,x}^{\infty}(t_{0},t_{1})}\leq\epsilon(\mathbb{S}^{n-1},\mathcal{E}),
\]
then
\[
\left\Vert \phi\right\Vert _{S(t_{0},t_{1})}\leq F(\mathbb{S}^{n-1},\mathcal{E}).
\]
Moreover, considering an admissible frequency envelope $c$ attached
to some $\phi[t]$ for $t_{0}<t<t_{1}$, as in (\ref{eq:FrequencyEnvelope})
and $\sigma_{0}$ as in Theorem \ref{thm:Small-data-waves}, we obtain:
\[
\left\Vert P_{k}\phi\right\Vert _{S(t_{0},t_{1})}\lesssim c_{k},
\]
and the map $\phi$ extends to a smooth wave map on a neighborhood
of the time interval $(t_{0},t_{1})$.\end{thm}
\begin{rem}
\label{RemSBoundOnW}In this paper, the above theorem will be used
indirectly only, but we can apply it immediately to the wave maps
$\varpi_{t_{0},\nu}$ from Theorem \ref{thm:Main} concentrating on
the null boundary $\partial C$, to obtain the bound:
\[
\left\Vert \varpi_{t_{0},\nu}\right\Vert _{S[t_{0}-\delta_{0},t_{0}+\delta_{0}]}\lesssim1,
\]
for any $t_{0}\in[1+\delta_{0},2-\delta_{0}]$.
\end{rem}

\subsection{\label{sub:CompensationEstimate}Compensation type estimates.}

We prove here two compensation estimates for wave maps into spheres
with a good bound in the direction of some constant time-like vector
field, relying on the conservation law (\ref{eq:ConservationLaw})
to treat high-high frequency interactions (this phenomena goes back
essentially to Wente). These estimates will play a key role in the
proof of no loss of energy in formation of solitons, and as in the
case of higher dimensional harmonic maps considered by Lin and Rivière
\citep{LiR}, this is the only place where we use the fact that our
target manifold is $\mathbb{S}^{n-1}$.
\begin{prop}
\label{prop:Compensation-estimate.} Let $\phi:[-1,1]\times\mathbb{R}^{2}\rightarrow\mathbb{S}^{n-1}$
be a smooth wave map equal to a constant $c$ outside a compact domain
in space, with energy bounded by some positive $\mathcal{E}>0$:
\begin{equation}
\left\Vert \nabla_{t,x}\phi\right\Vert _{L_{t}^{\infty}(L_{x}^{2})[-1,1]}^{2}\leq\mathcal{E},\label{eq:EnergyBounAssumtion}
\end{equation}
and $X$ a constant time-like vector field, that we may take to be:
\begin{equation}
X=\mathrm{cosh}(\zeta)\partial_{t}+\mathrm{sinh}(\zeta)\partial_{x_{1}},\label{eq:ExpressionForX-1}
\end{equation}
for some rapidity constant $\zeta\geq0$. Denote by $\chi=\chi(t)\in C_{0}^{\infty}(-1,1)$
a smooth time cut-off function, then there exists a decomposition
holding in $\mathcal{S}(\mathbb{R}\times\mathbb{R}^{2})$:
\begin{equation}
\chi\nabla_{t,x}\phi=\Theta_{X}+\Xi_{X},\label{eq:CompensatedDecomposition}
\end{equation}
satisfying:
\begin{equation}
\left\Vert \Theta_{X}\right\Vert _{L_{t,x}^{2}}\lesssim\left\Vert X\phi\right\Vert _{L_{t,x}^{2}[-1,1]}+\left\Vert \phi-c\right\Vert _{L_{t}^{\infty}(L_{x}^{2})[-1,1]}\label{eq:TimeLikeTermInDecomp}
\end{equation}
and
\begin{equation}
\sum_{k\in\mathbb{Z}}\left\Vert P_{k}\Xi_{X}\right\Vert _{L_{t}^{1}(L_{x}^{2})}\lesssim1,\label{eq:CompensatedTerm}
\end{equation}
with the implicit constants depending only on $n$ the dimension of
$\mathbb{R}^{n}$, the energy bound $\mathcal{E}$, the rapidity constant
$\zeta$ and the cut-off $\chi$ (most notably on $\left\Vert \partial_{t}\chi\right\Vert _{L_{t}^{\infty}}$). \end{prop}
\begin{proof}
We start by noting that, expressing $\partial_{t}$ as a linear combination
of $X$ and $\partial_{x_{1}}$ via (\ref{eq:ExpressionForX-1}),
it suffices to consider the spatial gradient $\chi\nabla_{x}\phi$.

For low frequencies, we proceed claiming immediately:
\begin{equation}
\left\Vert \chi P_{\leq0}\nabla_{x}\phi\right\Vert _{L_{t,x}^{2}}\lesssim\left\Vert \phi-c\right\Vert _{L_{t}^{\infty}(L_{x}^{2})[-1,1]},\label{eq:LowFrequencies-1}
\end{equation}
which simply follows from the finite band property (\ref{eq:FiniteBandLess}),
passing to $L_{t}^{\infty}(L_{x}^{2})$ as necessary. This is an acceptable
contribution.

For high modulations, we claim:
\begin{equation}
\left\Vert \sum_{k\in\mathbb{Z}}Q_{\geq k+10}P_{k}[\chi\nabla_{x}\phi]\right\Vert _{L_{t,x}^{2}}\lesssim\left\Vert X\phi\right\Vert _{L_{t,x}^{2}[-1,1]}+\left\Vert \phi-c\right\Vert _{L_{t}^{\infty}(L_{x}^{2})[-1,1]},\label{eq:HighModulationsGradient}
\end{equation}
and the idea here, as in \citep{TataruSterbenzWaveReg}, is to note
that the vector field $X$ being time-like, the Fourier multiplier
$X^{-1}\nabla_{x}Q_{\geq k+10}\widetilde{P}_{k}$, where $\widetilde{P}_{k}=P_{k-1\leq\cdot\leq k+1}$,
has symbol smooth and bounded uniformly in $k\in\mathbb{Z}$. By Plancherel
in $L_{t,x}^{2}$, this gives rise to the favorable elliptic estimate:
\begin{align}
\left\Vert Q_{\geq k+10}P_{k}[\chi\nabla_{x}\phi]\right\Vert _{L_{t,x}^{2}}\lesssim & \left\Vert \chi P_{k}X\phi\right\Vert _{L_{t,x}^{2}}+\left\Vert (\partial_{t}\chi)P_{k}\phi\right\Vert _{L_{t,x}^{2}},\label{eq:EllipticX}
\end{align}
and so (\ref{eq:HighModulationsGradient}) follows square-summing
in $k$ the above and dropping the cut-off. This is again acceptable.

The main term to consider is $Q_{<k+10}P_{k}(\chi\nabla_{x}\phi)$
with $k>0$, and for this we rely on the wave maps equation (\ref{eq:WMeq}),
that we trick slightly to make the vector $X$ to appear, introducing
the operator:
\begin{equation}
\Delta_{x,\beta}:=(1-\beta^{2})\partial_{x_{1}}^{2}+\partial_{x_{2}}^{2},\,\,\,\beta:=\mathrm{tanh}(\zeta)\in[0,1),\label{eq:SpatialBetaLaplacian}
\end{equation}
which is elliptic in the frequency region considered. So, using (\ref{eq:ExpressionForX-1}),
together with (\ref{eq:GeometricId}), we rewrite the wave maps equation
(\ref{eq:WMeq}) as:
\begin{align}
\Delta_{x,\beta}(\chi\phi)= & -\chi(\phi\partial_{\alpha}\phi^{\dagger}-\partial_{\alpha}\phi\phi^{\dagger})\partial^{\alpha}\phi\label{eq:WaveMapEqBeta}\\
 & +\mathrm{sech}^{2}(\zeta)(X-2\mathrm{sinh}(\zeta)\partial_{x_{1}})(\chi X\phi)-\mathrm{sech}(\zeta)(\partial_{t}\chi)X\phi,\nonumber 
\end{align}
and inverting $\Delta_{x,\beta}$ we have:
\[
P_{>0}\chi\nabla_{x}\phi=\frac{\nabla_{x}}{\Delta_{x,\beta}}P_{>0}(\Delta_{x,\beta}(\chi\phi)),
\]
holding in $\mathcal{S}(\mathbb{R}\times\mathbb{R}^{2})$, hence let
us treat each term in (\ref{eq:WaveMapEqBeta}) one by one.

Considering second line in (\ref{eq:WaveMapEqBeta}), we control the
first two terms by claiming, for any $k\in\mathbb{Z}$:
\begin{equation}
\left\Vert \nabla_{x}\frac{\nabla_{t,x}}{\Delta_{x,\beta}}Q_{<k+10}P_{k}(\chi X\phi)\right\Vert _{L_{t,x}^{2}}\lesssim\left\Vert P_{k}X\phi\right\Vert _{L_{t,x}^{2}[-1,1]},\label{eq:TimeLikeError(I)}
\end{equation}
which follows immediately discarding, via Plancherel in $L_{t,x}^{2}$,
the Fourier multiplier $\nabla_{x}\nabla_{t,x}\Delta_{x,\beta}^{-1}Q_{<k+10}\widetilde{P}_{k}$
of symbol bounded uniformly in $k\in\mathbb{Z}$, and dropping the
time cut-off $\chi$. For the third term, we have, for any $k\in\mathbb{Z}$:
\begin{equation}
\left\Vert \frac{\nabla_{x}}{\Delta_{x,\beta}}Q_{<k+10}P_{k}[(\partial_{t}\chi)X\phi]\right\Vert _{L_{t,x}^{2}}\lesssim2^{-k}\left\Vert \partial_{t}\chi\right\Vert _{L_{t,x}^{\infty}}\left\Vert P_{k}X\phi\right\Vert _{L_{t,x}^{2}[-1,1]},\label{eq:TimeLikeError(II)}
\end{equation}
where we discarded by Plancherel in $L_{t,x}^{2}$ the Fourier multiplier
$2^{k}\nabla_{x}\Delta_{x,\beta}^{-1}Q_{<k+10}\widetilde{P}_{k}$,
having here again the symbol bounded uniformly in $k\in\mathbb{Z}$.
Therefore, square-summing over $k>0$, both (\ref{eq:TimeLikeError(I)})
and (\ref{eq:TimeLikeError(II)}) lead to acceptable contributions.

We consider now the non-linear term on the first line of (\ref{eq:WaveMapEqBeta}).
Let us introduce some notation for the connection matrices:
\begin{equation}
\Omega_{\alpha}:=\phi\partial_{\alpha}\phi^{\dagger}-\partial_{\alpha}\phi\phi^{\dagger},\,\,\,\mathrm{with}\,\,\,\partial^{\alpha}\Omega_{\alpha}=0\,\,\,\mathrm{and}\,\,\,\left\Vert \Omega_{\alpha}\right\Vert _{L_{t}^{\infty}(L_{x}^{2})[-1,1]}\lesssim1,\label{eq:ConnectionMatrixNotation}
\end{equation}
by (\ref{eq:ConservationLaw}), respectively the global energy bound
(\ref{eq:EnergyBounAssumtion}) and the boundedness of the wave map.
We claim then the following compensation estimate:
\begin{equation}
\sum_{k>0}\left\Vert \frac{\nabla_{x}}{\Delta_{x,\beta}}Q_{<k+10}P_{k}(\chi\Omega_{\alpha}\partial^{\alpha}\phi)\right\Vert _{L_{t}^{1}(L_{x}^{2})}\lesssim1.\label{eq:ClaimDivCurlControl}
\end{equation}

Thanks to the conservation law, the term $\Omega_{\alpha}\partial^{\alpha}\phi$
exhibits and a div-curl type structure, and we should treat this using
the Littlewood-Paley trichotomy in very much the same standard way
as the actual div-curl structure, see Taylor's monograph \citep{TaylorTools}.
We start by writing:
\begin{align}
P_{k}\left(\chi\Omega_{\alpha}\partial^{\alpha}\phi\right)= & \, P_{k}[(\partial_{t}\chi)\sum_{k_{1},k_{2}\geq k-6:\left|k_{1}-k_{2}\right|\leq O(1)}\Omega_{\alpha,k_{1}}\phi_{k_{2}}\label{eq:HighHigh}\\
 & +\partial^{\alpha}\sum_{k_{1},k_{2}\geq k-6:\left|k_{1}-k_{2}\right|\leq O(1)}\chi\Omega_{\alpha,k_{1}}\phi_{k_{2}}\nonumber \\
 & +\chi\Omega_{\alpha,\leq k-7}\partial^{\alpha}\phi_{k-3\leq\cdot\leq k+3}\nonumber \\
 & +\chi\Omega_{\alpha,k-3\leq\cdot\leq k+3}\partial^{\alpha}\phi_{\leq k-7}],\nonumber 
\end{align}
where $\Omega_{\alpha,k_{1}}:=P_{k_{1}}\Omega_{\alpha}$ and similarly
for $\phi_{k_{2}}$, $\Omega_{\alpha,\leq k_{1}}$, etc. We are going
to prove claim (\ref{eq:ClaimDivCurlControl}) for each of the terms
in (\ref{eq:HighHigh}) separately. Note that the Fourier multipliers:
\begin{equation}
\frac{\nabla_{x}\nabla_{t,x}}{\Delta_{x,\beta}}Q_{<k+10}\widetilde{P}_{k}\,\,\,\mathrm{and}\,\,\,\frac{2^{k}\nabla_{x}}{\Delta_{x,\beta}}Q_{<k+10}\widetilde{P}_{k},\label{eq:ShittyMultipliersDispo}
\end{equation}
are disposable, which is essentially contained in Lemma \ref{lem:Dispo}
(precomposing, for example, with the space-time LP-projections to
$\left|\tau\right|+\left|\xi\right|\sim2^{k}$ that we don't use here
otherwise). This justifies the fact that we can work with the space
$L_{t}^{1}(L_{x}^{2})$ instead of $L_{t,x}^{2}$ (on which, of course,
(\ref{eq:ShittyMultipliersDispo}) are bounded by Plancherel).

Let us start with the high-high interactions on the first and second
lines of (\ref{eq:HighHigh}), for which we control (\ref{eq:ClaimDivCurlControl}),
discarding the multipliers (\ref{eq:ShittyMultipliersDispo}) and
dropping $2^{-k}\partial_{t}\chi$ for the first term, by:
\begin{align}
 & \sum_{k>0}\left\Vert P_{k}\sum_{k_{1},k_{2}\geq k-6:\left|k_{1}-k_{2}\right|\leq O(1)}\Omega_{\alpha,k_{1}}\phi_{k_{2}}\right\Vert _{L_{t}^{1}(L_{x}^{2})[-1,1]}\label{eq:HighHighContribution}\\
 & \lesssim\sup_{t\in[-1,1]}\sum_{k>0}2^{k}\sum_{k_{1},k_{2}\geq k-6:\left|k_{1}-k_{2}\right|\leq O(1)}\left\Vert \Omega_{\alpha,k_{1}}(t)\phi_{k_{2}}(t)\right\Vert _{L_{x}^{1}},\nonumber 
\end{align}
where we applied Bernstein's inequality (\ref{eq:Bernstein}), commuted
the sum $\sum_{k>0}$ with $L_{t}^{1}$ and discarded $P_{k}$. Using
Cauchy-Schwarz in $L_{x}^{1}$ and recalling the finite band property
(\ref{eq:FiniteBandEqual}) for $\phi_{k_{2}}$, we can bound the
contribution of (\ref{eq:HighHighContribution}) via:
\[
\sup_{t\in[-1,1]}\sum_{k>0}\sum_{k_{1},k_{2}\geq k-6:\left|k_{1}-k_{2}\right|\leq O(1)}2^{-(k_{2}-k)}\left\Vert \Omega_{\alpha,k_{1}}(t)\right\Vert _{L_{x}^{2}}\left\Vert \nabla_{x}\phi_{k_{2}}(t)\right\Vert _{L_{x}^{2}},
\]
and summing this over $k>0$, letting $i:=k_{1}-k_{2}$ and $j:=k_{2}-k$,
we obtain:
\begin{align*}
 & \sup_{t\in[-1,1]}\sum_{i=O(1)}\sum_{j\geq O(1)}2^{-j}\sum_{k>0}\left\Vert \Omega_{\alpha,k+j+i}(t)\right\Vert _{L_{x}^{2}}\left\Vert \nabla_{x}\phi_{k+j}(t)\right\Vert _{L_{x}^{2}}\\
 & \lesssim\sup_{t\in[-1,1]}\left(\sum_{k_{1}\geq O(1)}\left\Vert \Omega_{\alpha,k_{1}}(t)\right\Vert _{L_{x}^{2}}^{2}\right)^{\frac{1}{2}}\left(\sum_{k_{2}\geq O(1)}\left\Vert \nabla_{x}\phi_{k_{2}}(t)\right\Vert _{L_{x}^{2}}^{2}\right)^{\frac{1}{2}},
\end{align*}
where we have used Cauchy-Schwarz in $k$. By the global energy bound,
we get that high-high interactions make an acceptable contribution
to (\ref{eq:ClaimDivCurlControl}).

Finally, let us consider the contribution of the paraproducts from
lines three and four in (\ref{eq:HighHigh}), and we focus on the
latter as the former is treated in the same way by symmetry (or in
fact, could have already been absorbed in the argument for high-high
interactions). Here, the div-curl structure is not playing any role,
and is actually counter-productive. Hence, discarding the second multiplier
from (\ref{eq:ShittyMultipliersDispo}) and commuting the discrete
sum $\sum_{k>0}$ with $L_{t}^{1}$ as previously, it suffices control:
\[
\sup_{t\in[-1,1]}\sum_{k>0}2^{-k}\left\Vert P_{k}[\Omega_{\alpha,k-3\leq\cdot\leq k+3}(t)\partial^{\alpha}\phi_{\leq k-7}(t)]\right\Vert _{L_{x}^{2}}.
\]
Recalling the embedding (\ref{eq:HardyBesovEmbedding}) we are reduced
to showing:
\[
\sup_{t\in[-1,1]}\left\Vert \sum_{k>0}P_{k}[\Omega_{\alpha,k-3\leq\cdot\leq k+3}(t)\partial^{\alpha}\phi_{\leq k-7}(t)]\right\Vert _{F_{2}^{0,1}(\mathbb{R}^{2})}\lesssim1.
\]

Using the duality $(F_{2}^{0,1})'=F_{2}^{0,\infty}$, as discussed
in section \ref{sub:Some-harmonic-analysis.}, we take an arbitrary
$\varphi\in F_{2}^{0,\infty}$ together with a representation $\varphi=\sum_{k\geq0}\varphi_{k}$
in $\mathcal{S}_{x}'$ such that each $\varphi_{k}$ has Fourier support
in $\left|\xi\right|\sim2^{k}$ ($\left|\xi\right|\lesssim1$ for
$\varphi_{0}$) and: 
\[
\parallel(\sum_{k\geq0}|\varphi_{k}|^{2})^{1/2}\parallel_{L_{x}^{\infty}}\leq2\left\Vert \varphi\right\Vert _{F_{2}^{0,\infty}}.
\]
Then, recalling the fact that LP-projections are self-adjoint, we
must show that:
\[
\sum_{j=O(1)}\sum_{k\geq0}\int\left|\Omega_{\alpha,k-3\leq\cdot\leq k+3}(t)\partial^{\alpha}\phi_{\leq k-7}(t)\varphi_{k+j}\right|dx\lesssim\left\Vert \varphi\right\Vert _{F_{2}^{0,\infty}},
\]
with the convention that $\varphi_{k}$ with $k$ negative simply
stands for $\varphi_{0}$. Using Cauchy-Schwartz we bound this via:
\[
\left\Vert (\sum_{k\geq0}|\Omega_{\alpha,k-3\leq\cdot\leq k+3}(t)|^{2})^{1/2}\right\Vert _{L_{x}^{2}}\left\Vert \sup_{k\in\mathbb{Z}}|P_{\leq k}\nabla_{t,x}\phi(t)|\right\Vert _{L_{x}^{2}}\,\sum_{j=O(1)}\left\Vert (\sum_{k\geq0}|\varphi_{k+j}|^{2})^{1/2}\right\Vert _{L_{x}^{\infty}}.
\]
It is a well-known fact from harmonic analysis, to which we shall
refer as the \emph{Littlewood-Paley square function estimate}, see
e.g. \citep{TaylorTools}, that:
\[
\parallel(\sum_{k\in\mathbb{Z}}|\Omega_{\alpha,k}(t)|^{2})^{1/2}\parallel_{L_{x}^{2}}\lesssim\left\Vert \Omega_{\alpha}(t)\right\Vert _{L_{x}^{2}}\,\,\,\mathrm{and}
\]
\[
\parallel\sup_{k\in\mathbb{Z}}\left|P_{\leq k}\nabla_{t,x}\phi(t)\right|\parallel_{L_{x}^{2}}\lesssim\left\Vert \nabla_{t,x}\phi(t)\right\Vert _{L_{x}^{2}}.
\]
Hence, by the global energy bound, the contribution of the paraproducts
is acceptable. Therefore we have shown the compensation estimate (\ref{eq:ClaimDivCurlControl}). 

Proposition \ref{prop:Compensation-estimate.} is proved.
\end{proof}
We present now a compensation estimate for higher order time-like
derivatives of wave maps as considered in the previous proposition.
It holds up to a non-linear bulk, essentially quadratic in the gradient
and local in time, that we shall consider on neck regions later in
the proof of the weak Besov $\dot{B}_{\infty}^{1,2}$ decay estimate
in Lemma \ref{lem:Besov-control.}. Parts of this estimate are non-linear,
and will be established via a duality argument in the spirit of the
energy collapsing result itself.

As for Proposition \ref{prop:Compensation-estimate.}, the conservation
law (\ref{eq:ConservationLaw}) is absolutely crucial, and so our
arguments do not generalize directly to the case of a general target
beyond the Euclidean sphere $\mathbb{S}^{n-1}$. 
\begin{lem}
\label{lem:HigherOrderTimeLike}Consider a wave map $\phi:[-1,1]\times\mathbb{R}^{2}\rightarrow\mathbb{S}^{n-1}$
with the same set-up as in Proposition \ref{prop:Compensation-estimate.},
then we have the following decomposition holding in $\mathcal{S}(\mathbb{R}\times\mathbb{R}^{2})$,
using notation from (\ref{eq:ConnectionMatrixNotation}):
\begin{align}
\mathrm{sech}^{2}(\zeta)\chi X^{2}\phi= & -\sum_{k\in\mathbb{Z}}P_{k}\left[\chi((1-\beta^{2})\partial_{x_{1}}\Omega_{x_{1}}+\partial_{x_{2}}\Omega_{x_{2}})(P_{>k+10}\phi)\right]\label{eq:HigherOrderWeakDecomp}\\
 & +\mathrm{sech}^{2}(\zeta)\chi\left(-\Omega_{X}X\phi+\mathrm{sinh}(\zeta)(\Omega_{X}\partial_{x_{1}}\phi+\Omega_{x_{1}}X\phi)\right)\nonumber \\
 & +\Pi_{X},\nonumber 
\end{align}
the error term satisfying: 
\begin{align}
 & \sum_{k\in\mathbb{Z}}2^{-2k}\left\Vert P_{k}\Pi_{X}\right\Vert _{L_{t,x}^{2}[-1,1]}^{2}\label{eq:HigherOrderWeakEstimate}\\
 & \lesssim(1+\left\Vert X\phi\right\Vert _{L_{t,x}^{2}[-1,1]}+\left\Vert \phi-c\right\Vert _{L_{t}^{\infty}(L_{x}^{2})[-1,1]})(\left\Vert X\phi\right\Vert _{L_{t,x}^{2}[-1,1]}+\left\Vert \phi-c\right\Vert _{L_{t}^{\infty}(L_{x}^{2})[-1,1]}),\nonumber 
\end{align}
with the same dependence for the implicit constant as in Proposition
\ref{prop:Compensation-estimate.}.\end{lem}
\begin{proof}
Let us start with the frequency space-like region, that we can treat
directly and for which we claim the stronger estimate:
\begin{equation}
2^{-k}\left\Vert P_{k}Q_{<k+10}(\chi X^{2}\phi)\right\Vert _{L_{t,x}^{2}[-1,1]}\lesssim\left\Vert P_{k}X\phi\right\Vert _{L_{t,x}^{2}[-1,1]},\label{eq:LowModulationsWeak}
\end{equation}
for any $k\in\mathbb{Z}$. To see this, we simply commute $X$ with
the time cut-off $\chi$, getting:
\[
2^{-k}\left\Vert P_{k}Q_{<k+10}(\chi X^{2}\phi)\right\Vert _{L_{t,x}^{2}[-1,1]}\lesssim\left\Vert P_{k}(\chi X\phi)\right\Vert _{L_{t,x}^{2}}+2^{-k}\left\Vert P_{k}Q_{<k+10}(\partial_{t}\chi X\phi)\right\Vert _{L_{t,x}^{2}[-1,1]},
\]
where for the first term we discarded the multiplier $2^{-k}X\widetilde{P}_{k}Q_{<k+10}$
using Plancherel in $L_{t,x}^{2}$. Regarding the second one, passing
to $L_{t}^{\infty}(L_{x}^{2})$, which is possible as we are working
over a bounded time interval in (\ref{eq:LowModulationsWeak}), we
can apply the inversion formula for the space-time Fourier transform
$\mathcal{F}$, to get:
\[
2^{-k}\left\Vert P_{k}Q_{<k+10}(\partial_{t}\chi X\phi)\right\Vert _{L_{t}^{\infty}(L_{x}^{2})}\lesssim2^{-k}\left\Vert \mathcal{F}P_{k}Q_{<k+10}(\partial_{t}\chi X\phi)\right\Vert _{L_{\tau}^{1}(L_{\xi}^{2})},
\]
combining Minkowski's inequality and then Plancherel in $L_{x}^{2}$.
But the integrand on the RHS has $\tau$-support of length $O(2^{k})$,
hence we can bound this simply via:
\[
\left\Vert \mathcal{F}P_{k}(\partial_{t}\chi X\phi)\right\Vert _{L_{\tau}^{\infty}(L_{\xi}^{2})}\lesssim\left\Vert P_{k}(\partial_{t}\chi X\phi)\right\Vert _{L_{t}^{1}(L_{x}^{2})}\lesssim\left\Vert \partial_{t}\chi\right\Vert _{L_{t}^{2}(L_{x}^{\infty})}\left\Vert P_{k}X\phi\right\Vert _{L_{t,x}^{2}[-1,1]},
\]
where we applied the inversion formula for $\mathcal{F}^{-1}$ this
time (note that this argument is essentially a manifestation of Bernstein's
one dimensional inequality). This gives claim (\ref{eq:LowModulationsWeak})
as desired.

For high modulations, we use the wave maps equation as in (\ref{eq:WaveMapEqBeta}).
Following the Littlewood-Paley trichotomy (passing to the convention
$\phi_{k}:=P_{k}\phi$, etc. as before), we write:
\begin{align*}
P_{k}Q_{\geq k+10}(\mathrm{sech}^{2}(\zeta)\chi X^{2}\phi)= & \, P_{k}Q_{\geq k+10}\left[\Delta_{x,\beta}(\chi\phi)+2\,\mathrm{sech}^{2}(\zeta)\,\mathrm{sinh}(\zeta)\chi\partial_{x_{1}}X\phi\right.\\
 & +\mathrm{sech}^{2}(\zeta)\chi\left(-\Omega_{X}X\phi+\mathrm{sinh}(\zeta)(\Omega_{X}\partial_{x_{1}}\phi+\Omega_{x_{1}}X\phi)\right)\\
 & +\chi\Omega_{x,\beta}\cdot\nabla_{x}\phi_{\leq k+10}+\chi\nabla_{x}\cdot(\Omega_{x,\beta}\phi_{>k+10})\\
 & -\left.\chi(\nabla_{x}\cdot\Omega_{x,\beta})(\phi_{>k+10})\right],
\end{align*}
where we set: 
\[
\Omega_{X}:=\mathrm{cosh}(\zeta)\Omega_{t}+\mathrm{sinh}(\zeta)\Omega_{x_{1}}\,\,\,\mathrm{and}\,\,\,\Omega_{x,\beta}:=(1-\beta^{2})\Omega_{x_{1}}\partial_{x_{1}}+\Omega_{x_{2}}\partial_{x_{1}},
\]
recalling (\ref{eq:ExpressionForX-1}), with ``$\cdot$'' standing
for the Euclidean inner product. From there, we add and subtract the
frequency space-like part of the terms on second and last lines above,
and use the conservation law (\ref{eq:ConservationLaw}), that we
rewrite as:
\[
\nabla_{x}\cdot\Omega_{x,\beta}=\mathrm{sech}^{2}(\zeta)(X\Omega_{X}-\mathrm{sinh}(\zeta)(\partial_{x_{1}}\Omega_{X}+X\Omega_{x_{1}})).
\]

This yields the following decomposition:
\begin{align}
P_{k}Q_{\geq k+10}(\mathrm{sech}^{2}(\zeta)\chi X^{2}\phi)= & \, P_{k}Q_{\geq k+10}\left[\Delta_{x,\beta}(\chi\phi)+2\,\mathrm{sech}^{2}(\zeta)\,\mathrm{sinh}(\zeta)\chi\partial_{x_{1}}X\phi\right]\label{eq:MAIN_Auxiliary_DECOMP}\\
 & +\mathrm{sech}^{2}(\zeta)\chi P_{k}\left[-\Omega_{X}X\phi+\mathrm{sinh}(\zeta)(\Omega_{X}\partial_{x_{1}}\phi+\Omega_{x_{1}}X\phi)\right]\nonumber \\
 & +\mathrm{sech}^{2}(\zeta)Q_{<k+10}[\psi_{k}^{(1)}+\psi_{k}^{(2)}+\psi_{k}^{(3)}+\psi_{k}^{(4)}]\nonumber \\
 & +Q_{\geq k+10}[\varphi_{k}^{(1)}+\varphi_{k}^{(2)}]-P_{k}\left[\chi(\nabla_{x}\cdot\Omega_{x,\beta})(\phi_{>k+10})\right],\nonumber 
\end{align}
where we define:

\begin{align*}
\psi_{k}^{(1)}:= & \,\chi P_{k}\left[\Omega_{X}X\phi_{\leq k+10}-\mathrm{sinh}(\zeta)(\Omega_{X}\partial_{x_{1}}\phi_{\leq k+10}+\Omega_{x_{1}}X\phi_{\leq k+10})\right],\\
\psi_{k}^{(2)}:= & \, P_{k}\left[(X\chi)(-\Omega_{X}+\mathrm{sinh}(\zeta)\Omega_{x_{1}})\phi_{>k+10}\right]\\
\psi_{k}^{(3)}:= & \, P_{k}\left[[X-\mathrm{sinh}(\zeta)\partial_{x_{1}}](\chi\Omega_{X}\phi_{>k+10})\right],\\
\psi_{k}^{(4)}:= & \, P_{k}\left[-\mathrm{sinh}(\zeta)X(\chi\Omega_{x_{1}}\phi_{>k+10})\right],
\end{align*}
as well as:
\begin{align*}
\varphi_{k}^{(1)}:= & \, P_{k}\left[\chi\Omega_{x,\beta}\cdot\nabla_{x}\phi_{\leq k+10}\right],\\
\varphi_{k}^{(2)}:= & \, P_{k}\left[\chi\nabla_{x}\cdot[\Omega_{x,\beta}\phi_{>k+10}]\right].
\end{align*}

We proceed proving the estimate (\ref{eq:HigherOrderWeakEstimate})
for the first line of (\ref{eq:MAIN_Auxiliary_DECOMP}) and each of
the $\psi_{k}^{(i)}$ and $\varphi_{k}^{(i)}$ separately.

For the Laplacian, inverting $X$, we have the stronger estimate:
\[
2^{-k}\left\Vert P_{k}Q_{\geq k+10}\Delta_{x,\beta}(\chi\phi)\right\Vert _{L_{t,x}^{2}[-1,1]}\lesssim\left\Vert P_{k}X\phi\right\Vert _{L_{t,x}^{2}[-1,1]}+\left\Vert P_{k}\phi\right\Vert _{L_{t,x}^{2}[-1,1]},
\]
that follows immediately by discarding, via Plancherel in $L_{t,x}^{2}$,
the Fourier multiplier $2^{-k}X^{-1}\Delta_{x,\beta}\widetilde{P}_{k}Q_{\geq k+10}$
having symbol bounded uniformly in $k\in\mathbb{Z}$, which leads
to an acceptable contribution.

For the second term on the RHS of (\ref{eq:MAIN_Auxiliary_DECOMP})
we immediately have:
\[
2^{-k}\left\Vert P_{k}Q_{\geq k+10}[\chi\partial_{x_{1}}X\phi]\right\Vert _{L_{t,x}^{2}}\lesssim\left\Vert P_{k}X\phi\right\Vert _{L_{t,x}^{2}[-1,1]},
\]
by the finite band property (\ref{eq:FiniteBandEqual}), which is
acceptable.

Regarding $\psi_{k}^{(1)}$, we remark that it has a paraproduct structure
and so at least one of the factors will be frequency localized to
$\left|\xi\right|\sim2^{k}$, which is favorable for square-summing.
More precisely, discarding $Q_{<k+10}$ before dropping the cut-off
$\chi$, and using Bernstein's inequality (\ref{eq:Bernstein}) to
pass to $L_{t}^{2}(L_{x}^{1})$, it is enough to note that for any
$1\leq p,q,r\leq n$ and any time slice $t\in[-1,1]$:
\begin{align*}
 & \sum_{k\in\mathbb{Z}}\sum_{k'=k+O(1)}\left(\left\Vert (\phi^{p}X\phi^{q}\nabla_{t,x}\phi_{k'}^{r})(t)\right\Vert _{L_{x}^{1}}^{2}+\left\Vert (\phi^{p}\nabla_{t,x}\phi^{q}X\phi_{k'}^{r})(t)\right\Vert _{L_{x}^{1}}^{2}\right.\\
 & \left.+\left\Vert (P_{k'}[\phi^{p}\nabla_{t,x}\phi^{q}]X\phi^{r})(t)\right\Vert _{L_{x}^{1}}^{2}+\left\Vert (P_{k'}[\phi^{p}X\phi^{q}]\nabla_{t,x}\phi^{r})(t)\right\Vert _{L_{x}^{1}}^{2}\right)\\
 & \lesssim\left\Vert \nabla_{t,x}\phi(t)\right\Vert _{L_{x}^{2}}^{2}\left\Vert X\phi(t)\right\Vert _{L_{x}^{2}}^{2},
\end{align*}
by Cauchy-Schwarz. Upon integrating in time, this is an acceptable
contribution by the energy bound (\ref{eq:EnergyBounAssumtion}).

For the expression $\psi_{k}^{(2)}$, it is already convenient to
proceed via a duality argument:
\begin{align*}
\sum_{k\in\mathbb{Z}}2^{-2k}\left\Vert Q_{<k+10}\psi_{k}^{(2)}\right\Vert _{L_{t,x}^{2}}^{2} & \lesssim\sum_{k\in\mathbb{Z}}\left\Vert \psi_{k}^{(2)}\right\Vert _{L_{t,x}^{1}}2^{-k}\left\Vert Q_{<k+10}\psi_{k}^{(2)}\right\Vert _{L_{t}^{\infty}(L_{x}^{2})}\\
 & \lesssim\left(\sum_{k\in\mathbb{Z}}2^{k}\left\Vert \psi_{k}^{(2)}\right\Vert _{L_{t,x}^{1}}\right)\left(\sup_{k\in\mathbb{Z}}\left\Vert \psi_{k}^{(2)}\right\Vert _{L_{t,x}^{1}}\right),
\end{align*}
where we used Bernstein (\ref{eq:Bernstein}) for the first factor,
and for the second one we proceeded as for the frequency space-like
term (\ref{eq:LowModulationsWeak}), using time frequency localization
to estimate it via the Fourier inversion formula:
\[
2^{-k}\left\Vert Q_{<k+10}\psi_{k}^{(2)}\right\Vert _{L_{t}^{\infty}(L_{x}^{2})}\lesssim\left\Vert \psi_{k}^{(2)}\right\Vert _{L_{t}^{1}(L_{x}^{2})}.
\]
The first factor is universally bounded for us, as for any $1\leq p,q,r\leq n$:
\[
\sum_{k\in\mathbb{Z}}2^{k}\left\Vert P_{k}[(\partial_{t}\chi)(\phi^{p}\nabla_{t,x}\phi^{q})\phi_{>k+10}^{r}]\right\Vert _{L_{t,x}^{1}}\lesssim\left\Vert \partial_{t}\chi\right\Vert _{L_{t}^{1}(L_{x}^{\infty})}\left\Vert \nabla_{t,x}\phi\right\Vert _{L_{t}^{\infty}(L_{x}^{2})[-1,1]}^{2},
\]
which follows directly from the analogous treatment of high-high interactions
in the proof of Proposition \ref{prop:Compensation-estimate.}. On
the other hand, the second factor is controlled via:
\[
\left\Vert P_{k}[(\partial_{t}\chi)(\phi^{p}\nabla_{t,x}\phi^{q})\phi_{>k+10}^{r}]\right\Vert _{L_{t,x}^{1}}\lesssim\left\Vert \partial_{t}\chi\right\Vert _{L_{t}^{1}(L_{x}^{\infty})}\left\Vert \nabla_{t,x}\phi\right\Vert _{L_{t}^{\infty}(L_{x}^{2})[-1,1]}\left\Vert \phi-c\right\Vert _{L_{t}^{\infty}(L_{x}^{2})[-1,1]},
\]
which yields an acceptable contribution to the non-linear part of
(\ref{eq:HigherOrderWeakEstimate}).

Regarding $\psi_{k}^{(3)}$, it is a linear combination of:
\begin{align*}
 & \sum_{k\in\mathbb{Z}}2^{-2k}\left\Vert Q_{<k+10}P_{k}\nabla_{t,x}(\chi\Omega_{X}\phi_{>k+10})\right\Vert _{L_{t,x}^{2}}^{2}\\
 & \lesssim\left\Vert \sum_{k\in\mathbb{Z}}\sum_{k_{1},k_{2}\geq k+O(1):\left|k_{1}-k_{2}\right|\leq O(1)}2^{-(k_{2}-k)}\left\Vert \chi\Omega_{X,k_{1}}(t)\right\Vert _{L_{x}^{2}}\left\Vert \nabla_{x}\phi_{k_{2}}(t)\right\Vert _{L_{x}^{2}}\right\Vert _{L_{t}^{2}}^{2},
\end{align*}
where we discarded via Plancherel in $L_{t,x}^{2}$ the Fourier multiplier
$2^{-k}\nabla_{t,x}Q_{<k+10}\widetilde{P}_{k}$ having bounded symbol,
passed from $\ell^{2}$ to $\ell^{1}$ summation in $k$ after commuting
time integration with the discrete sum $\sum_{k}$, and applied Bernstein
(\ref{eq:Bernstein}) with Cauchy-Schwarz. This contribution is directly
seen to be bounded by $O(\left\Vert X\phi\right\Vert _{L_{t,x}^{2}[-1,1]}^{2}\left\Vert \nabla_{t,x}\phi\right\Vert _{L_{t}^{\infty}(L_{x}^{2})[-1,1]}^{2})$
as required.

The terms $\psi_{k}^{(4)}$, $\varphi_{k}^{(1)}$ and $\varphi_{k}^{(2)}$
are similar and require a duality argument relying heavily on their
compensated structure to obtain estimate (\ref{eq:HigherOrderWeakEstimate})
at $\ell^{2}$ modulation. 

First for $\psi_{k}^{(4)}$, using the self-adjointness of $Q_{<k+10}$
and then commuting $\sum_{k}$ with time integration, we have:
\begin{align*}
 & \sum_{k\in\mathbb{Z}}2^{-2k}\left\Vert Q_{<k+10}\psi_{k}^{(4)}\right\Vert _{L_{t,x}^{2}}^{2}\\
 & \lesssim\left\Vert \sum_{k\in\mathbb{Z}}2^{-k}\left\Vert (Q_{<k+10}^{2}\psi_{k}^{(4)})(t)\right\Vert _{L_{x}^{2}}\right\Vert _{L_{t}^{2}}\cdot\left\Vert \sup_{k\in\mathbb{Z}}2^{-k}\left\Vert \psi_{k}^{(4)}(t)\right\Vert _{L_{x}^{2}}\right\Vert _{L_{t}^{2}[-1,1]}.
\end{align*}
For the first factor we claim that it is universally bounded due to
its compensated structure. Indeed, passing to the Hardy space on each
time slice via the embedding (\ref{eq:HardyBesovEmbedding}), we estimate
it by:
\[
\left\Vert (\sum_{k\in\mathbb{Z}}|Q_{<k+10}^{2}\psi_{k}^{(4)}|^{2})^{\frac{1}{2}}\right\Vert _{L_{t}^{2}(L_{x}^{1})}\lesssim\left\Vert (\sum_{k\in\mathbb{Z}}|2^{k}P_{k}[\chi\Omega_{x_{1}}\phi_{>k+10}]|^{2})^{\frac{1}{2}}\right\Vert _{L_{t}^{2}(L_{x}^{1})},
\]
where we relied on the Calderón-Zygmund theory for the Littlewood-Paley
square function and the vector valued operator $(2^{-k}XQ_{<k+10}^{2}\widetilde{P}_{k})_{k\in\mathbb{Z}}$,
precomposing with the space-time LP-projections to $\left|\tau\right|+\left|\xi\right|\sim2^{k}$
as necessary. From there, proceeding as previously, we immediately
bound the latter by $O(\left\Vert \nabla_{x}\phi\right\Vert _{L_{t}^{\infty}(L_{x}^{2})[-1,1]}^{2})$
as required.

The set-up is similar for $\varphi_{k}^{(1)}$ and $\varphi_{k}^{(2)}$.
Here however, being at high modulations, we start by inverting the
time-like vector $X$ for one of the factors. Then, using the skew-adjointness
of $2^{k}X^{-1}Q_{\geq k+10}$, but proceeding identically to the
above otherwise, we obtain:
\begin{align*}
 & \sum_{k\in\mathbb{Z}}2^{-2k}\int\int(Q_{\geq k+10}\varphi_{k}^{(i)})(\frac{X}{X}Q_{\geq k+10}\varphi_{k}^{(i)})dxdt\\
 & \lesssim\left\Vert \sum_{k\in\mathbb{Z}}2^{-k}\left\Vert (\frac{2^{k}}{X}Q_{\geq k+10}^{2}\varphi_{k}^{(i)})(t)\right\Vert _{L_{x}^{2}}\right\Vert _{L_{t}^{2}}\cdot\left\Vert \sup_{k\in\mathbb{Z}}2^{-2k}\left\Vert X\varphi_{k}^{(i)}(t)\right\Vert _{L_{x}^{2}}\right\Vert _{L_{t}^{2}[-1,1]}\\
 & \lesssim\sup_{j\geq10}\left\Vert (\sum_{k\in\mathbb{Z}}|\frac{2^{k+j}}{X}Q_{k+j}\widetilde{Q}_{k+j}\varphi_{k}^{(i)}|^{2})^{\frac{1}{2}}\right\Vert _{L_{t}^{2}(L_{x}^{1})}\cdot\left\Vert \sup_{k\in\mathbb{Z}}2^{-2k}\left\Vert X\varphi_{k}^{(i)}(t)\right\Vert _{L_{x}^{2}}\right\Vert _{L_{t}^{2}[-1,1]}\\
 & \lesssim\left\Vert (\sum_{k\in\mathbb{Z}}|\varphi_{k}^{(i)}|^{2})^{\frac{1}{2}}\right\Vert _{L_{t}^{2}(L_{x}^{1})}\cdot\left\Vert \sup_{k\in\mathbb{Z}}2^{-2k}\left\Vert X\varphi_{k}^{(i)}(t)\right\Vert _{L_{x}^{2}}\right\Vert _{L_{t}^{2}[-1,1]},
\end{align*}
where $\widetilde{Q}_{k+j}=Q_{k+j-1\leq\cdot\leq k+j+1}$ is the slightly
enlarged modulation projection, and we relied as previously on Calderón-Zygmund
theory to discard the vector valued operator $(2^{k+j}X^{-1}Q_{k+j}\widetilde{Q}_{k+j}\widetilde{P}_{k})_{k\in\mathbb{Z}}$,
precomposing with the space-time LP-projections to $\left|\tau\right|+\left|\xi\right|\sim2^{k+j}$
as necessary, for any integer $j\geq10$. 

From there, we note that the first factor is bounded by $O(\left\Vert \nabla_{x}\phi\right\Vert _{L_{t}^{\infty}(L_{x}^{2})[-1,1]}^{2})$
as required. This follows essentially from the arguments used to treat
the high-high interactions and the paraproducts, for $\varphi_{k}^{(1)}$
and $\varphi_{k}^{(2)}$ respectively, in the proof of Proposition
\ref{prop:Compensation-estimate.} that we shall not reproduce here.

Given this, to prove estimate (\ref{eq:HigherOrderWeakEstimate})
for the terms $\psi_{k}^{(4)}$, $\varphi_{k}^{(1)}$ and $\varphi_{k}^{(2)}$,
it is enough by (\ref{eq:FiniteBandEqual}) and (\ref{eq:Bernstein})
to establish the following couple of weak estimates:
\begin{align}
2^{-k}\left\Vert XP_{k}[(\chi\phi^{p}\nabla_{x}\phi^{q})(P_{\leq k+10}\nabla_{x}\phi^{r})](t)\right\Vert _{L_{x}^{1}} & \lesssim\left\Vert X\phi(t)\right\Vert _{L_{x}^{2}}+\left\Vert (\phi-c)(t)\right\Vert _{L_{x}^{2}},\label{eq:WEAK_ONE}\\
2^{-k}\left\Vert XP_{k}[\chi\phi^{p}(\nabla_{x}\phi^{q})(\phi_{>k+10}^{r})](t)\right\Vert _{L_{x}^{2}} & \lesssim\left\Vert X\phi(t)\right\Vert _{L_{x}^{2}}+\left\Vert (\phi-c)(t)\right\Vert _{L_{x}^{2}},\label{eq:WEAK_TWO}
\end{align}
for any $1\leq p,q,r\leq n$ and any time slice $t\in[-1,1]$. 

Consider (\ref{eq:WEAK_ONE}). For convenience, let us suppress the
time $t$ from the notation. Moving $X$ inside the bracket, we first
differentiate the time cut-off getting by Cauchy-Schwarz:
\[
2^{-k}\left\Vert (\partial_{t}\chi)(\phi^{p}\nabla_{x}\phi^{q})(P_{\leq k+10}\nabla_{x}\phi^{r})]\right\Vert _{L_{x}^{1}}\lesssim\left\Vert \phi\right\Vert _{L_{x}^{\infty}}\left\Vert \nabla_{x}\phi\right\Vert _{L_{x}^{2}}\left\Vert \phi-c\right\Vert _{L_{x}^{2}},
\]
where we relied on the finite band property (\ref{eq:FiniteBandLess})
for $\phi^{r}$, which is a permissible bound for (\ref{eq:WEAK_ONE}). 

Next, if $X$ falls on $\phi^{p}$, then we have:
\[
2^{-k}\left\Vert \chi X\phi^{p}\nabla_{x}\phi^{q}(P_{\leq k+10}\nabla_{x}\phi^{r})]\right\Vert _{L_{x}^{1}}\lesssim\left\Vert X\phi^{p}\right\Vert _{L_{x}^{2}}\left\Vert \chi\nabla_{x}\phi^{q}\right\Vert _{L_{x}^{2}}\left\Vert \phi^{r}\right\Vert _{L_{x}^{\infty}},
\]
with again the finite band property (\ref{eq:FiniteBandLess}) applied
to $\phi^{r}$, but this time in $L_{x}^{\infty}$, and this is an
acceptable bound.

When $X$ falls on $\nabla_{x}\phi^{q}$, we shall first insert the
projection $P_{\leq k+O(1)}$ in front of $\phi^{p}X\nabla_{x}\phi^{q}$,
which is possible by the localization of $\nabla_{x}\phi_{\leq k+10}^{r}$,
and untangle the high-high interactions:
\begin{align*}
P_{\leq k+O(1)}(\phi^{p}X\nabla_{x}\phi^{q})= & \, P_{\leq k+O(1)}[\phi_{\leq k+O(1)}^{p}X\nabla_{x}\phi_{\leq k+O(1)}^{q}\\
 & +\sum_{k_{1},k_{2}\geq k+O(1):\left|k_{1}-k_{2}\right|\leq O(1)}\phi_{k_{1}}^{p}X\nabla_{x}\phi_{k_{2}}^{q}].
\end{align*}
Given this decomposition, we have for the low frequency interactions:
\begin{align*}
 & 2^{-k}\left\Vert \chi\phi_{\leq k+O(1)}^{p}(X\nabla_{x}\phi_{\leq k+O(1)}^{q})(\nabla_{x}\phi_{\leq k+10}^{r})\right\Vert _{L_{x}^{1}}\\
 & \lesssim\left\Vert \phi^{p}\right\Vert _{L_{x}^{\infty}}\left\Vert X\phi^{q}\right\Vert _{L_{x}^{2}}\left\Vert \nabla_{x}\phi^{r}\right\Vert _{L_{x}^{2}},
\end{align*}
where we used the finite band property (\ref{eq:FiniteBandLess})
for $\phi^{q}$, and this is acceptable. For the high-high frequency
interactions:
\begin{align*}
 & \sum_{k_{1},k_{2}\geq k+O(1):\left|k_{1}-k_{2}\right|\leq O(1)}2^{-k}\left\Vert \chi\phi_{k_{1}}^{p}(X\nabla_{x}\phi_{k_{2}}^{q})(\nabla_{x}\phi_{\leq k+10}^{r})\right\Vert _{L_{x}^{1}}\\
 & \lesssim\left\Vert \phi^{r}\right\Vert _{L_{x}^{\infty}}\sum_{k_{1},k_{2}\geq k+O(1):\left|k_{1}-k_{2}\right|\leq O(1)}\left\Vert \nabla_{x}\phi_{k_{1}}^{p}\right\Vert _{L_{x}^{2}}\left\Vert X\phi_{k_{2}}^{q}\right\Vert _{L_{x}^{2}},
\end{align*}
where we have used the finite band property (\ref{eq:FiniteBandLess})
for $\phi^{r}$ in $L_{x}^{\infty}$, and transferred the spatial
gradient from $\phi^{q}$ to $\phi^{p}$ by relying on (\ref{eq:FiniteBandEqual})
this time and the fact that $\left|k_{1}-k_{2}\right|\leq O(1)$.
This control is acceptable applying the discrete Cauchy-Schwarz inequality
in $k_{1}=k_{2}+O(1)$.

The last case we need to consider, in order to finish with (\ref{eq:WEAK_ONE}),
is when $X$ falls on $\phi^{r}$. This follows however at once, applying
(\ref{eq:FiniteBandLess}) to the latter:
\[
2^{-k}\left\Vert \chi\phi^{p}(\nabla_{x}\phi^{q})(P_{\leq k+10}\nabla_{x}X\phi^{r})]\right\Vert _{L_{t}^{2}(L_{x}^{1})}\lesssim\left\Vert \phi^{p}\right\Vert _{L_{t,x}^{\infty}}\left\Vert \nabla_{x}\phi^{q}\right\Vert _{L_{t}^{\infty}(L_{x}^{2})[-1,1]}\left\Vert X\phi^{r}\right\Vert _{L_{t,x}^{2}[-1,1]},
\]
which is certainly acceptable and gives (\ref{eq:WEAK_ONE}).

The estimate (\ref{eq:WEAK_TWO}) is very much similar to (\ref{eq:WEAK_ONE}).
As previously, we move $X$ into the bracket, first estimating the
term when the derivative falls on the time cut-off, passing initially
to $L_{x}^{1}$ via Bernstein's inequality (\ref{eq:Bernstein}):
\[
\left\Vert (\partial_{t}\chi)(\phi^{p}\nabla_{x}\phi^{q})(\phi_{>k+10}^{r})\right\Vert _{L_{x}^{1}}\lesssim\left\Vert \phi\right\Vert _{L_{x}^{\infty}}\left\Vert \nabla_{x}\phi\right\Vert _{L_{x}^{2}}\left\Vert \phi-c\right\Vert _{L_{x}^{2}},
\]
simply noting that $P_{>k+10}\phi=P_{>k+10}(\phi-c)$ and then discarding
the LP-projection. When $X$ differentiates $\phi^{p}$, we pass again
to $L_{x}^{1}$, and then immediately get:
\[
\left\Vert \chi(X\phi^{p})(\nabla_{x}\phi^{q})(\phi_{>k+10}^{r})\right\Vert _{L_{x}^{1}}\lesssim\left\Vert X\phi^{p}\right\Vert _{L_{x}^{2}}\left\Vert \nabla_{x}\phi^{q}\right\Vert _{L_{x}^{2}}\left\Vert \phi^{r}\right\Vert _{L_{x}^{\infty}}.
\]
Both estimates are acceptable for (\ref{eq:WEAK_TWO}).

We consider now the term with $X$ falling on $\phi^{q}$, and untangling
the high-high interactions in the product we should regroup together
$\phi^{p}$ and $\phi^{r}$, obtaining: 
\begin{align*}
P_{k}[(\phi^{p}X\nabla_{x}\phi^{q})(\phi_{>k+10}^{r})]= & \, P_{k}[P_{\leq k+O(1)}(\phi^{p}\phi_{>k+10}^{r})X\nabla_{x}\phi_{\leq k+O(1)}^{q}\\
 & \sum_{k_{1},k_{2}\geq k+O(1):\left|k_{1}-k_{2}\right|\leq O(1)}P_{k_{1}}(\phi^{p}\phi_{>k+10}^{r})X\nabla_{x}\phi_{k_{2}}^{q}].
\end{align*}
Now, given this decomposition, we control the first term directly
by applying the finite band property (\ref{eq:FiniteBandLess}) to
$\phi^{q}$ :
\[
2^{-k}\left\Vert \chi P_{\leq k+O(1)}(\phi^{p}\phi_{>k+10}^{r})X\nabla_{x}\phi_{\leq k+O(1)}^{q}\right\Vert _{L_{x}^{2}}\lesssim\left\Vert \phi^{p}\phi_{>k+10}^{r}\right\Vert _{L_{x}^{\infty}}\left\Vert X\phi^{q}\right\Vert _{L_{x}^{2}},
\]
which is acceptable by the boundedness of wave maps. For the high-high
interactions we proceed as for (\ref{eq:WEAK_ONE}) above, passing
initially to $L_{x}^{1}$ via Bernstein's inequality (\ref{eq:Bernstein})
and transferring the spatial gradient $\nabla_{x}$ from $\phi^{q}$
to $\phi^{p}\phi_{>k+10}^{r}$ via the finite band property (\ref{eq:FiniteBandEqual}),
which gives:
\begin{align*}
 & \sum_{k_{1},k_{2}\geq k+O(1):\left|k_{1}-k_{2}\right|\leq O(1)}\left\Vert \chi P_{k_{1}}(\phi^{p}\phi_{>k+10}^{r})X\nabla_{x}\phi_{k_{2}}^{q}\right\Vert _{L_{x}^{1}}\\
 & \lesssim\sum_{k_{1},k_{2}\geq k+O(1):\left|k_{1}-k_{2}\right|\leq O(1)}\left\Vert P_{k_{1}}\nabla_{x}(\phi^{p}\phi_{>k+10}^{r})\right\Vert _{L_{x}^{2}}\left\Vert X\phi_{k_{2}}^{q}\right\Vert _{L_{x}^{2}},
\end{align*}
and using the discrete Cauchy-Schwarz, we can bound this via:
\[
(\left\Vert \nabla_{x}\phi^{p}\right\Vert _{L_{x}^{2}}\left\Vert \phi^{r}\right\Vert _{L_{x}^{\infty}}+\left\Vert \phi^{p}\right\Vert _{L_{x}^{\infty}}\left\Vert \nabla_{x}\phi^{r}\right\Vert _{L_{x}^{2}})\left\Vert X\phi^{q}\right\Vert _{L_{x}^{2}},
\]
which is certainly acceptable. 

Lastly, if $X$ differentiates $\phi^{r}$, we pass to $L_{x}^{1}$
and this immediately yields the desired control:
\[
\left\Vert \chi\phi^{p}(\nabla_{x}\phi^{q})(P_{>k+10}X\phi^{r})\right\Vert _{L_{x}^{1}}\lesssim\left\Vert \phi^{p}\right\Vert _{L_{x}^{\infty}}\left\Vert \nabla_{x}\phi^{q}\right\Vert _{L_{x}^{2}}\left\Vert X\phi^{r}\right\Vert _{L_{x}^{2}},
\]
hence we have (\ref{eq:WEAK_TWO}).

Lemma \ref{lem:HigherOrderTimeLike} is proved.
\end{proof}

\section{Bubbling analysis\label{sec:Bubbling-analysis}}

In this section we prove our main Theorem \ref{thm:Main}. We start
by recording, in the lemma just below, some of the important properties
of the wave map $\phi$, we were considering in the statement of the
threshold Theorem \ref{thm:ThresholdConjecture}, at the final rescaling
obtained by Sterbenz and Tataru in section 6.6 of \citep{TataruSterbenzWaveReg}.
\begin{lem}
\label{lem:FinalRescaling}\emph{(Sterbenz and Tataru \citep{TataruSterbenzWaveReg})}.
The maps $\{\phi_{\nu}\}_{\nu\in\mathbb{N}}$ from Theorem \ref{thm:ThresholdConjecture}
represent a sequence of smooth wave maps of bounded energy on increasingly
large domains of the forward light cone $C$:
\begin{equation}
\phi_{\nu}:C_{[\varsigma_{\nu},\varsigma_{\nu}^{-1}]}\longrightarrow\mathbb{S}^{n-1},\,\,\,\mathcal{E}_{S_{t}}[\phi_{\nu}]\leq\mathcal{E}\,\,\,\forall t\in[\varsigma_{\nu},\varsigma_{\nu}^{-1}],\label{eq:EnergyBound}
\end{equation}
where $\varsigma_{\nu}\downarrow0$ as $\nu\rightarrow\infty$, with
the following properties:

$\bullet$ There exists a sequence $\epsilon_{\nu}\downarrow0$, with
$\epsilon_{\nu}^{\frac{1}{2}}\ll\varsigma_{\nu}$, such that:
\begin{equation}
\mathcal{F}_{[\varsigma_{\nu},\varsigma_{\nu}^{-1}]}[\phi_{\nu}]<\epsilon_{\nu}^{\frac{1}{2}}\mathcal{E};\label{eq:Extension&Flux0}
\end{equation}

$\bullet$ A decay to the self-similar mode holds:
\begin{equation}
\int\int_{C_{[\varsigma_{\nu},\varsigma_{\nu}^{-1}]}^{\epsilon_{\nu}^{\frac{1}{2}}}}\frac{1}{\rho}\left|\partial_{\rho}\phi_{\nu}\right|^{2}dxdt\lesssim\left|\log\epsilon_{\nu}\right|^{-\frac{1}{2}}\mathcal{E},\label{eq:AsymptoticSelfSim}
\end{equation}
where $\rho=(t^{2}-r^{2})^{\frac{1}{2}}$ and $\partial_{\rho}=\frac{1}{\rho}(t\partial_{t}+r\partial_{r})$
is the scaling vector field which we recall is uniformly time-like
$\mu(\partial_{\rho},\partial_{\rho})=-1$;

$\bullet$ There is a uniform amount of energy $\mathcal{E}_{c}>0$
getting concentrated by the maps $\phi_{\nu}$ in the interior of
the light cone:
\begin{equation}
\frac{1}{2}\int_{\left|x\right|<\gamma_{c}t_{0}}\left|\nabla_{t,x}\phi_{\nu}(t_{0})\right|^{2}dx\geq\mathcal{E}_{c}\,\,\,\forall t_{0}\in[\varsigma_{\nu},\varsigma_{\nu}^{-1}],\label{eq:TimeLikeEnergyConcentration}
\end{equation}
for some $0<\gamma_{c}<1$.
\end{lem}
Let us write here a few lines of comments regarding the above lemma,
referring the reader to \citep{TataruSterbenzWaveReg} for more details.
Given a sequence of concentration points $(t_{\nu},x_{\nu})$ for
the energy dispersion norm:
\[
2^{-k_{\nu}}\left|P_{k_{\nu}}\nabla_{t,x}\phi(t_{\nu},x_{\nu})\right|>\epsilon(\mathbb{S}^{n-1},\mathcal{E}),
\]
with $t_{\nu}\rightarrow0$ in the case of a finite time blow-up,
or $t_{\nu}\rightarrow+\infty$ in a non-scattering scenario, the
sequence $\epsilon_{\nu}\downarrow0$ is chosen such that:
\[
\mathcal{F}_{[\epsilon_{\nu}t_{\nu},t_{\nu}]}[\phi]<\epsilon_{\nu}^{\frac{1}{2}}\mathcal{E}.
\]
In \citep{TataruSterbenzWaveReg}, sections 6.3 and 6.4, the authors
use the above lower bound to prove that there is a non-trivial amount
of time-like energy concentrating on the time slice $S_{t_{\nu}}$.
As we shall later rely on those results in Section \ref{sub:Dispersive-property},
we gathered them in Lemma \ref{lem:NoTimeLikeEnergyImplyDispersive}
here. From there, a weighted energy estimate (see Lemma 3.4 in \citep{TataruSterbenzWaveReg})
propagates this energy backwards in time, leading to (\ref{eq:TimeLikeEnergyConcentration})
for any $t\in[\epsilon_{\nu}^{1/2}t_{\nu},\epsilon_{\nu}^{1/4}t_{\nu}]$. 

In parallel to this, a Morawetz type estimate (see Lemma 3.3 in \citep{TataruSterbenzWaveReg})
and the pigeonhole principle enable Sterbenz and Tataru to find a
sequence of time intervals $[\tau_{\nu},N_{\nu}\tau_{\nu}]\subset[\epsilon_{\nu}^{1/2},\epsilon_{\nu}^{1/4}]$,
with $N_{\nu}=\exp(\sqrt{\left|\log\epsilon_{\nu}\right|})$, such
that the following decay estimate holds:
\[
\int\int_{C_{[\tau_{\nu},N_{\nu}\tau_{\nu}]}^{\epsilon_{\nu}}}\frac{1}{\rho}\left|\partial_{\rho}[\phi(t_{\nu}t,t_{\nu}x)]\right|^{2}dxdt\lesssim\left|\log\epsilon_{\nu}\right|^{-\frac{1}{2}}\mathcal{E},
\]
see section 6.6 in \citep{TataruSterbenzWaveReg}. Then for the final
rescaling, the authors in \citep{TataruSterbenzWaveReg} choose $t_{\nu}\tau_{\nu}$
for the scales $\lambda_{\nu}^{0}$ (or $\lambda_{\nu}^{\infty}$),
obtaining a sequence of wave maps $\phi(\lambda_{\nu}^{0}\cdot)$
with the desired properties on the growing cones $C_{[1,N_{\nu}]}$.
In our case, it will be more convenient (for notational purposes mainly,
as to respect the CMC foliation in Section \ref{sub:Concentration-compactness}
below), to asymptotically cover all of forward light cone $C^{0}$,
so we should simply fix any:
\[
t_{\nu}\tau_{\nu}\ll\lambda_{\nu}^{0},\lambda_{\nu}^{\infty}\ll N_{\nu}t_{\nu}\tau_{\nu},
\]
and choose then $\varsigma_{\nu}\downarrow0$ decaying slowly enough,
for Lemma \ref{lem:FinalRescaling} to hold.

Finally, we bring reader's attention here to our convention that,
in any of the results stated in this last section, we assume (\ref{eq:EnergyBound})-(\ref{eq:TimeLikeEnergyConcentration})
holding without mentioning it. In fact, one might directly consider
those as the assumptions under which claims of Theorem \ref{thm:Main}
are made.

\subsection{Blow-up analysis for asymptotically self-similar sequences of wave
maps.\label{sub:Concentration-compactness}}

We start the proof of Theorem \ref{thm:Main} with a study of the
energy concentration sets. Our approach here will be close in spirit
to the work of Freire, Müller and Struwe \citep{FreiMulStruCompact}.
We will rely on a monotonicity lemma for asymptotically self-similar
wave maps, see Lemma \ref{lem:MonotonicityLemma} below, which is
a rough analogue of part (ii) from Lemma 1.7 in Lin's work \citep{Lin},
but mainly parallels the computations in the proof of Morawetz type
estimates from section 3 of \citep{TataruSterbenzWaveReg}. Note that
we do not use here the fact that our target manifold is a sphere.

It will be convenient to use hyperbolic coordinates, also known as
CMC foliation of the (forward) light cone $C^{0}$, where we recall
that $C^{0}$ denotes the open interior of the forward light cone,
$C^{0}=C\setminus(\partial C\cup\left\{ (0,0)\right\} )$. Those are
defined by:
\[
t=\rho\,\mathrm{cosh}(y),\,\,\, r=\rho\,\mathrm{sinh}(y)\,\,\,\mathrm{and}\,\,\,\theta.
\]
Associated to those coordinates, we recall the expression for the
volume element:
\[
dV:=rdtdrd\theta=\rho^{2}\mathrm{sinh}(y)d\rho dyd\theta,
\]
and for the hyperbolic planes $\mathbb{H}_{\rho_{0}}^{2}=\left\{ \rho=\rho_{0}\right\} $
the area element:
\[
dA_{\rho_{0}}:=\rho_{0}^{2}\,\mathrm{sinh}(y)dyd\theta,
\]
with respect to the Minkowski metric $\mu$ on $\mathbb{R}^{2+1}$.
These formulae will be useful below applying Stokes' theorem in the
hyperbolic annulus $\left\{ \rho_{1}\leq\rho\leq\rho_{2}\right\} $.
Let us also record here that, using the identities:
\[
\partial_{t}=\frac{t}{\rho}\partial_{\rho}-\frac{r}{\rho^{2}}\partial_{y},\,\,\,\partial_{r}=\frac{t}{\rho^{2}}\partial_{y}-\frac{r}{\rho}\partial_{\rho},
\]
one computes, for a smooth map $\phi$ into $\mathbb{S}^{n-1}$:
\begin{align}
\partial^{\gamma}\phi^{\dagger}\partial_{\gamma}\phi & =-\left|\partial_{t}\phi\right|^{2}+\left|\partial_{r}\phi\right|^{2}+\frac{1}{r^{2}}\left|\partial_{\theta}\phi\right|^{2}\nonumber \\
 & =-\left|\partial_{\rho}\phi\right|^{2}+\frac{1}{\rho^{2}}\left|\nabla_{\mathbb{H}^{2}}\phi\right|^{2},\label{eq:DensityHyperCoord}
\end{align}
where $\nabla_{\mathbb{H}^{2}}$ denotes the gradient on the unit
hyperboloid $\mathbb{H}^{2}:=\mathbb{H}_{1}^{2}$:
\[
\left|\nabla_{\mathbb{H}^{2}}\phi\right|^{2}=\left|\partial_{y}\phi\right|^{2}+\frac{1}{\mathrm{sinh}^{2}(y)}\left|\partial_{\theta}\phi\right|^{2}.
\]

For every given $\rho_{0}>0$, let us define the Radon measures:
\[
\sigma_{\nu,\rho_{0}}:=\left(\left|\partial_{\rho}\phi_{\nu}\right|^{2}+\frac{1}{\rho^{2}}\left|\nabla_{\mathbb{H}^{2}}\phi_{\nu}\right|^{2}\right)dA_{\rho_{0}}\in\mathcal{R}(\mathbb{H}_{\rho_{0}}^{2}).
\]
We can naturally view them as measures on the unit hyperbolic plane
$\mathbb{H}^{2}$ since for any given test function $\varphi$ on
$\mathbb{H}^{2}$, that we should view as a function $\varphi(y,\theta)$
independent of $\rho$ on the whole of the light cone $C^{0}$, we
have:
\[
\int\varphi d\sigma_{\nu,\rho_{0}}=\int_{\mathbb{H}^{2}}\left(\left|\partial_{\rho}\phi_{\nu}(\rho_{0})\right|^{2}+\frac{1}{\rho_{0}^{2}}\left|\nabla_{\mathbb{H}^{2}}\phi_{\nu}(\rho_{0})\right|^{2}\right)\varphi(y,\theta)\rho_{0}^{2}\,\mathrm{sinh}(y)dyd\theta.
\]

Using the decay (\ref{eq:AsymptoticSelfSim}) to a self-similar mode,
we can establish the following asymptotic monotonicity property for
the family $\left\{ \sigma_{\nu,\rho}\right\} _{\rho>0}\subset\mathcal{R}(\mathbb{H}^{2})$. 
\begin{lem}
\label{lem:MonotonicityLemma} For every pair $\rho_{2}>\rho_{1}>0$
and every $\lambda>0$, we have the decay:
\begin{equation}
\int_{\rho_{1}}^{\rho_{2}}\left(\int\varphi d\sigma_{\nu,\rho_{0}}\right)d\rho_{0}-\int_{\rho_{1}+\lambda}^{\rho_{2}+\lambda}\left(\int\varphi d\sigma_{\nu,\rho_{0}}\right)d\rho_{0}\longrightarrow0,\label{eq:MonotonicityFormula-1}
\end{equation}
holding as $\nu\rightarrow+\infty$ for any test function $\varphi\in C_{0}^{\infty}(\mathbb{H}^{2})$.\end{lem}
\begin{proof}
Given a continuously differentiable vector field $\psi=\psi^{\beta}\partial_{\beta}$
compactly supported in $(y,\theta)$, contracting the stress-energy
tensor $T[\phi_{\nu}]$ with $\psi$, we obtain the associated Noether
current:
\[
^{(\psi)}P_{\alpha}=T_{\alpha\beta}[\phi_{\nu}]\psi^{\beta}.
\]
Hence, if we set:
\[
D_{\left\{ \rho'\leq\rho\leq\rho''\right\} }(\psi):=\int_{\left\{ \rho'\leq\rho\leq\rho''\right\} }\partial^{\alpha}\left(^{(\psi)}P_{\alpha}\right)dV=\int_{\left\{ \rho'\leq\rho\leq\rho''\right\} }T_{\alpha\beta}[\phi_{\nu}]\partial^{\alpha}\psi^{\beta}dV,
\]
where we relied on the conservation law (\ref{eq:StressTensorDiv0}),
and:
\[
B_{\tilde{\rho}}(\psi):=\int_{\left\{ \rho=\tilde{\rho}\right\} }{}^{(\psi)}P(\partial_{\rho})dA_{\tilde{\rho}}=\int_{\left\{ \rho=\tilde{\rho}\right\} }T_{\alpha\beta}[\phi_{\nu}]\frac{x^{\alpha}}{\tilde{\rho}}\psi^{\beta}dA_{\tilde{\rho}},
\]
where our convention follows $x^{0}:=t$ and $x_{0}=-t$, so that
$x^{\alpha}=\mu^{\alpha\gamma}x_{\gamma}$, applying Stokes' theorem
over the region $\left\{ \rho_{0}\leq\rho\leq\rho_{0}+\lambda\right\} $
leads to the identity:
\begin{equation}
D_{\left\{ \rho_{0}\leq\rho\leq\rho_{0}+\lambda\right\} }(\psi)=B_{\rho_{0}}(\psi)-B_{\rho_{0}+\lambda}(\psi).\label{eq:StokesLemma}
\end{equation}

Taking $\psi=\varphi(y,\theta)\partial_{\rho}$, we compute using
the expression (\ref{eq:StressEnergyTensor}) for $T_{\alpha\beta}[\phi_{\nu}]$:
\[
D_{\left\{ \rho_{0}\leq\rho\leq\rho_{0}+\lambda\right\} }(\psi)=\int_{\left\{ \rho_{0}\leq\rho\leq\rho_{0}+\lambda\right\} }\left(\frac{1}{\rho}\left|\partial_{\rho}\phi_{\nu}\right|^{2}\varphi+\partial_{\rho}\phi_{\nu}^{\dagger}\partial_{\alpha}\phi_{\nu}\partial^{\alpha}\varphi\right)dV,
\]
and for the boundary terms:
\begin{align*}
B_{\tilde{\rho}}(\psi) & =\int_{\left\{ \rho=\tilde{\rho}\right\} }\left(\left|\tilde{\rho}\partial_{\rho}\phi_{\nu}\right|^{2}+\tilde{\rho}^{2}\frac{1}{2}\partial^{\gamma}\phi_{\nu}^{\dagger}\partial_{\gamma}\phi_{\nu}\right)\frac{\varphi}{\tilde{\rho}^{2}}dA_{\tilde{\rho}}\\
 & =\frac{1}{2}\int\varphi d\sigma_{\nu,\tilde{\rho}},
\end{align*}
where to pass to the second line we have used the identity (\ref{eq:DensityHyperCoord}).
Therefore, plugging the above back into (\ref{eq:StokesLemma}) we
obtain:
\begin{align*}
\int\varphi d\sigma_{\nu,\rho_{0}}-\int\varphi d\sigma_{\nu,\rho_{0}+\lambda}= & \,2\int_{\left\{ \rho_{0}\leq\rho\leq\rho_{0}+\lambda\right\} }\left(\frac{1}{\rho}\left|\partial_{\rho}\phi_{\nu}\right|^{2}\varphi+\partial_{\rho}\phi_{\nu}^{\dagger}\partial_{\alpha}\phi_{\nu}\partial^{\alpha}\varphi\right)dV.
\end{align*}
Integrating over $\rho_{0}\in[\rho_{1},\rho_{2}]$ and using Cauchy-Schwarz
for the second term on RHS above, appealing to the decay (\ref{eq:AsymptoticSelfSim})
and the global energy bound (\ref{eq:EnergyBound}), we obtain (\ref{eq:MonotonicityFormula-1}).
Hence Lemma \ref{lem:MonotonicityLemma} is proved.
\end{proof}
From now on we restrict ourselves to the time interval $1\leq t\leq2$.
We will study there the sets in space-time where our wave maps concentrate
a non-trivial amount of energy as in the work of Freire, Müller and
Struwe \citep{FreiMulStruCompact}, where some general statements
about the structure of energy concentration loci can be found (for
instance, it is shown in Proposition 4.1 and Theorem B.1 of \citep{FreiMulStruCompact}
that, upon passing to a suitable subsequence, the concentration set
of an energy threshold will be contained in a finite union of Lipschitz
curves). Our assumptions however enable us to go beyond \citep{FreiMulStruCompact}
via more elementary arguments and prove that picking a suitable subsequence
will lead to an energy concentration set which is in fact given by
a finite collection of time-like geodesics, relying on Lemmata \ref{prop:Simple-compactness-result.}
and \ref{lem:MonotonicityLemma}. 

To use the latter, we remark that for a fixed open domain $U$ with
closure $\overline{U}\subset C_{[\frac{1}{2},3]}^{0}$, we have:
\begin{equation}
\frac{1}{C}\left|\nabla_{t,x}\phi_{\nu}\right|^{2}\leq\left|\partial_{\rho}\phi_{\nu}\right|^{2}+\frac{1}{\rho^{2}}\left|\nabla_{\mathbb{H}^{2}}\phi_{\nu}\right|^{2}\leq C\left|\nabla_{t,x}\phi_{\nu}\right|^{2}\,\,\,\mathrm{on}\,\,\, U,\label{eq:EquivalenceControl}
\end{equation}
with $C:=C(\mathrm{dist}(U,\partial C_{[\frac{1}{2},3]}))$, and this
will enable us to transfer control back and forward between the Radon
measures $\sigma_{\nu,\rho}$ and the energy densities $\left|\nabla_{t,x}\phi_{\nu}\right|^{2}dxdt$
of which we want to study the concentration sets (with the small energy
compactness Lemma \ref{prop:Simple-compactness-result.} enabling
us to obtain some uniformity in time).
\begin{lem}
\label{lem:ConcentrationCompactness}There exists a subsequence of
$\left\{ \phi_{\nu}\right\} _{\nu\in\mathbb{N}}$ restricting to which,
without changing notation, we can find a finite collection of time-like
geodesics $\varrho_{1},\ldots,\varrho_{I}$ passing through the origin
in Minkowski space such that defining the energy concentration set
by:
\[
\Sigma:=\left\{ (t,x)\in C_{[1,2]}^{0}\,:\,\liminf_{\nu\rightarrow\infty}\mathcal{E}_{B_{r}(x)}[\phi_{\nu}](t)>\epsilon_{s}\,\,\,\forall r>0\right\} ,
\]
we have:
\[
\Sigma=C_{[1,2]}^{0}\cap\bigcup_{i=1}^{I}\varrho_{i},
\]
and away from $\Sigma$, there exist a wave map $\phi$ satisfying:
\[
\partial_{\rho}\phi=0\,\,\,\mathrm{\mathit{on}}\,\,\, C_{[1,2]}^{0}\setminus\Sigma\,\,\,\mathit{\mathit{with}}\,\,\,\phi\in(H_{t,x}^{\frac{3}{2}-\epsilon})_{loc}\left(C_{[1,2]}^{0}\setminus\Sigma\right),
\]
for any $0<\epsilon<\frac{1}{2}$, of finite energy on $C_{[1,2]}^{0}$,
$\mathcal{E}_{S_{t}^{0}}[\phi]\leq\mathcal{E}$ $\forall t\in[1,2]$,
such that:
\begin{equation}
\phi_{\nu}\longrightarrow\phi\,\,\,\mathrm{\mathit{on}}\,\,\,\left(C_{t}^{0}(H_{x}^{1})\cap C_{t}^{1}(L_{x}^{2})\right)_{loc}\left(C_{[1,2]}^{0}\setminus\Sigma\right),\label{eq:THING2}
\end{equation}
as dictated by Lemma \ref{prop:Simple-compactness-result.}.\end{lem}
\begin{proof}
In view of the asymptotic monotonicity provided by Lemma \ref{lem:MonotonicityLemma},
let us denote for a set $U\subset S_{t=1}$ the cone over $U$ by:
\[
C(U):=\left\{ \lambda(t,x)\,:\,\lambda>0,\, x\in U\,\,\,\mathrm{at}\,\,\, t=1\right\} ,
\]
and by $C_{I}(U):=C(U)\cap C_{I}$ the corresponding truncation to
a time interval $I$.

Considering the time slice $S_{1}^{0}$, given the global energy bound
(\ref{eq:EnergyBound}) we can pass to a subsequence for $\left\{ \phi_{\nu}\right\} _{\nu\in\mathbb{N}}$,
without changing notation, such that for some Radon measure $\iota\in\mathcal{R}(S_{1}^{0})$
we have:
\begin{equation}
\left|\nabla_{t,x}\phi_{\nu}(1)\right|^{2}dx\rightharpoonup\iota\,\,\,\mathrm{in}\,\,\,\mathcal{R}(S_{1}^{0}),\label{eq:Time1Concentration}
\end{equation}
from where we also see that there exist only finitely many points
$\left\{ x_{i}\right\} _{i=1}^{I}\subset S_{1}^{0}$ such that:
\begin{equation}
\left\{ x_{i}\right\} _{i=1}^{I}=\left\{ x\in B_{1}\,:\,\lim_{\nu\rightarrow\infty}\mathcal{E}_{B_{r}(x)}[\phi_{\nu}](1)>\epsilon_{s}\,\,\,\forall r>0\right\} ,\label{eq:RaysDefinition}
\end{equation}
and we set $\varrho_{i}:=C(\left\{ x_{i}\right\} )$. 

Let us start by showing that:
\begin{equation}
\Sigma\subset C_{[1,2]}^{0}\cap\bigcup_{i=1}^{I}\varrho_{i},\label{eq:Thing1half}
\end{equation}
obtaining on the way claim (\ref{eq:THING2}). Fix any point $x_{0}\in S_{1}^{0}\setminus\Sigma$,
then there exists a radius $r_{1}=r_{1}(x_{0})>0$ such that for all
$\nu\in\mathbb{N}$: 
\[
\mathcal{E}_{B_{r_{1}}(x_{0})}[\phi_{\nu}]\leq\epsilon_{s},
\]
hence by the energy-flux identity (\ref{eq:Energy-FluxIdentity}),
shrinking $r_{1}$ to $r_{2}>0$ as necessary, we obtain that:
\[
\sup_{t\in[1-3r_{2},1+3r_{2}]}\mathcal{E}_{B_{3r_{2}}(x_{0})}[\phi_{\nu}](t)\leq\epsilon_{s}.
\]
By the decay assumption (\ref{eq:AsymptoticSelfSim}), we can apply
the compactness Lemma \ref{prop:Simple-compactness-result.} obtaining
that on a subsequence $\left\{ \phi_{\nu'}\right\} _{\nu'\in\mathbb{N}}$
we have convergence in $C_{t}^{0}(H_{x}^{1})\cap C_{t}^{1}(L_{x}^{2})$
to a wave map $\phi$ in $[1-r_{2},1+r_{2}]\times B_{r_{2}}(x_{0})$,
satisfying $\partial_{\rho}\phi=0$ and having regularity as dictated
by (\ref{eq:LimitRegularity}) there.

Hence, given any positive constant $\eta>0$ there exist a radius
$r_{\eta}>0$ such that: 

\[
\sup_{\nu'\in\mathbb{N}}\sup_{t\in[1-r_{\eta},1+r_{\eta}]}\mathcal{E}_{C(B_{r_{\eta}}(x_{0}))}[\phi_{\nu'}](t)\leq\eta.
\]
Therefore, using (\ref{eq:EquivalenceControl}) we get for any test
function $\varphi(y,\theta)$ on the hyperboloid $\mathbb{H}^{2}$,
having support in $C(B_{r_{\eta}}(x_{0}))\cap\mathbb{H}^{2}$ and
satisfying $0\leq\varphi\leq1$, the bound: 
\[
\sup_{\nu'\in\mathbb{N}}\frac{1}{\rho_{2}-\rho_{1}}\int_{\rho_{1}}^{\rho_{2}}\left(\int\varphi d\sigma_{\nu',\rho}\right)d\rho\lesssim\eta,
\]
for some suitably chosen $0<\rho_{1}<\rho_{2}$. The implicit constant
here does not depend on the parameter $\eta$, and in fact depends
only on the distance of the point $(1,x_{0})$ to the null boundary.

Recalling Lemma \ref{lem:MonotonicityLemma}, we obtain by (\ref{eq:MonotonicityFormula-1})
for every fixed $\lambda>0$ the estimate:
\[
\limsup_{\nu'\rightarrow\infty}\frac{1}{\rho_{2}-\rho_{1}}\int_{\rho_{1}+\lambda}^{\rho_{2}+\lambda}\left(\int\varphi d\sigma_{\nu',\rho}\right)d\rho\lesssim\eta.
\]
Given this, shrinking $r_{2}$ to $r_{3}=r_{3}(x_{0},\eta)>0$ and
picking a suitable cut-off function $\varphi$ on $\mathbb{H}^{2}$
as necessary, we can rely on the other inequality in (\ref{eq:EquivalenceControl})
this time and the energy-flux identity (\ref{eq:Energy-FluxIdentity})
to find, arguing via the pigeonhole principle, a finite cover of:
\[
C_{[1,2]}(B_{r_{3}}(x_{0}))\subset\bigcup_{j=1}^{N}[t_{j}-s_{j},t_{j}+s_{j}]\times B_{s_{j}}(y_{j})
\]
with $N=N(x_{0},\eta)\in\mathbb{N}$ satisfying:
\[
\bigcup_{j=1}^{N}[t_{j}-3s_{j},t_{j}+3s_{j}]\times B_{3s_{j}}(y_{j})\subset C_{[\frac{1}{2},3]}(B_{r_{3}}(x_{0})),
\]
and such that:
\[
\limsup_{\nu'\rightarrow\infty}\sup_{t\in[t_{j}-3s_{j},t_{j}+3s_{j}]}\mathcal{E}_{B_{3s_{j}}(y_{j})}[\phi_{\nu'}](t)\lesssim\eta,\,\,\, j=1,\ldots,N,
\]
where the implicit constant is independent of $\eta$. Hence, choosing
$\eta>0$ small enough we can claim:
\[
\limsup_{\nu'\rightarrow\infty}\sup_{t\in[t_{j}-3s_{j},t_{j}+3s_{j}]}\mathcal{E}_{B_{3s_{j}}(y_{j})}[\phi_{\nu'}](t)\leq\frac{1}{2}\epsilon_{s},\,\,\, j=1,\ldots,N,
\]
with $N=N(x_{0})$ and $r_{3}=r_{3}(x_{0})$ now.

Proceeding this way for a countable dense set of points $x_{0}\in S_{1}^{0}\setminus\Sigma$,
we obtain ultimately a countable cover of $C_{[1,2]}^{0}\setminus\cup_{i}\varrho_{i}$
that we can use together with the compactness Lemma \ref{prop:Simple-compactness-result.}
to construct a subsequence for $\left\{ \phi_{\nu}\right\} _{\nu\in\mathbb{N}}$
via the diagonal process, to which we restrict ourselves without changing
notation this time, such that (\ref{eq:THING2}) hold for a wave map
$\phi\in(H_{t,x}^{3/2-\epsilon})_{loc}(C_{[1,2]}^{0}\setminus\cup_{i}\varrho_{i})$
with $\partial_{\rho}\phi=0$. By construction, it can be seen immediately
that the obtained map $\phi$ has energy bounded by $\mathcal{E}$
and we note the argument also yields (\ref{eq:Thing1half}) as desired.

To finish the proof of the lemma, we need to get the reverse inclusion
to (\ref{eq:Thing1half}). This follows however from a simple argument
by contradiction: suppose that there exists a point $(s_{i},y_{i})\in\varrho_{i}$
which is not contained in $\Sigma$. We can then run the above proof
with $(s_{i},y_{i})$ instead of $(1,x_{0})$ and obtain that the
full ray $\varrho_{i}$ is not contained in $\Sigma$, but that contradicts
the definition of $x_{i}$ from (\ref{eq:RaysDefinition}). Lemma
\ref{lem:ConcentrationCompactness} is therefore proved.
\end{proof}
To close the proof of the first part of Theorem \ref{thm:Main} it
is enough now to prove that the wave map $\phi$ obtained above must
in fact be constant. For this point, we will rely on a folklore fact
that finite energy self-similar wave maps do not exist in dimension
$2+1$ which we state in Proposition \ref{prop:SelfSimilar} below.
A self-contained proof of this proposition can be found in the work
of Sterbenz and Tataru \citep{TataruSterbenzWaveReg} (see section
4 there). 
\begin{prop}
\label{prop:SelfSimilar}Let $\phi$ be a smooth wave map in the interior
of the forward light cone $C^{0}$, having finite energy, $\mathcal{E}_{S_{t}^{0}}[\phi]\lesssim1$
$\forall t>0$, and satisfying the self-similarity condition $\partial_{\rho}\phi=0$.
Then $\phi$ must be constant.
\end{prop}
Consider the wave map $\phi$ from Lemma \ref{lem:ConcentrationCompactness}.
By homogeneity, we can extend it to:
\[
\phi:C^{0}\setminus\bigcup_{i=1}^{I}\varrho_{i}\longrightarrow\mathbb{S}^{n-1},
\]
with finite energy $\mathcal{E}_{S_{t}^{0}}[\phi]\leq\mathcal{E}$
$\forall t>0$, locally in $H_{t,x}^{\frac{3}{2}-\epsilon}$ and satisfying
$\partial_{\rho}\phi=0$. Let us note here that we were considering
the unit time interval $[1,2]$ in (\ref{eq:THING2}) just in order
to simplify the task of keeping track of the dependence of implicit
constants. It is easy to see that the arguments above lead to local
convergence of the sequence $\phi_{\nu}$ to the map $\phi$ on all
of $C^{0}\setminus\cup_{i}\varrho_{i}$. This is however a purely
qualitative statement.

Restricting $\phi$ to the unit hyperbolic plane $\mathbb{H}^{2}$
gives rise to a harmonic map of locally finite energy, by (\ref{eq:EquivalenceControl}),
defined away from a finite set of points given by $\mathbb{H}^{2}\cap\bigcup_{i=1}^{I}\varrho_{i}$.
By the regularity theory due to Hélein \citep{Helein}, we obtain
in fact a smooth harmonic map away from the above collection of points.
But then, by the removable singularity theorem of Sacks and Uhlenbeck
\citep{SacksUhlenbeck} we can extend $\phi$ to a smooth harmonic
map on the whole of the hyperbolic plane $\mathbb{H}^{2}$, which
in turn means that, by homogeneity again, we could have extended $\phi$
across the rays $\varrho_{i}$ to a smooth finite energy self-similar
wave map on $C^{0}$. By Proposition \ref{prop:SelfSimilar}, $\phi$
has to be a constant.

The first point of Theorem \ref{thm:Main} is therefore established,
given that $\Sigma$ must be non-trivial by the concentration of time-like
energy assumption (\ref{eq:TimeLikeEnergyConcentration}).

\subsection{Dispersive property for null-concentration.\label{sub:Dispersive-property}}

This short section is devoted to the description of the parts of the
sequence that escape into the null boundary. We proceed first, borrowing
arguments from section 6.1 of \citep{TataruSterbenzWaveReg}, by constructing
extensions for the maps $\phi_{\nu}$ outside the light cone with
asymptotically vanishing energy there (we note that, if considering
the non-scattering problem, those have been already constructed in
section 6.2 of \citep{TataruSterbenzWaveReg}).

Relying on the flux decay estimate (\ref{eq:Extension&Flux0}) and
using the angular part of $\mathcal{F}_{[\varsigma_{\nu},\varsigma_{\nu}^{-1}]}[\phi_{\nu}]$,
see the expression in (\ref{eq:Energy-FluxIdentity}), we can find
by the pigeonhole principle a sequence $\tau_{\nu}\in[2,3]$ such
that:
\[
\int_{\partial S_{\tau_{\nu}}}\left|r^{-1}\partial_{\theta}\phi_{\nu}(\tau_{\nu})\right|^{2}d\theta\lesssim\epsilon_{\nu}^{\frac{1}{2}}.
\]
Hence, as in Remark \ref{Rem:UniformityInTime}, we get that $\phi_{\nu}(\partial S_{\tau_{\nu}})$
is contained in a chart of radius $O(\epsilon_{\nu}^{1/4})$ and so
we can build smooth spatial extensions $\phi'_{\nu}[\tau_{\nu}]\in T(\mathbb{S}^{n-1})$
of $\phi{}_{\nu}[\tau_{\nu}]$, satisfying the energy control:
\[
\mathcal{E}[\phi'_{\nu}](\tau_{\nu})-\mathcal{E}_{S_{\tau_{\nu}}}[\phi{}_{\nu}]\lesssim\epsilon_{\nu}^{\frac{1}{2}}.
\]
We solve then the wave maps equation with initial data $\phi'_{\nu}[\tau_{\nu}]$
backwards in time for $t\in[\varsigma_{\nu},\tau_{\nu}]$. By the
finite speed of propagation property, the solution agrees with $\phi_{\nu}$
on $C_{[\varsigma_{\nu},\tau_{\nu}]}$, hence let us denote it by
$\phi_{\nu}$ (abusing slightly notation). Moreover, relying again
on the assumption (\ref{eq:Extension&Flux0}) and using the conservation
of energy law (\ref{eq:EnergyConservation}) together with the energy-flux
identity (\ref{eq:Energy-FluxIdentity}), we propagate to all of the
time interval $[\varsigma_{\nu},\tau_{\nu}]$ the smallness of the
energy exterior to the light cone:
\[
\sup_{t\in[\varsigma_{\nu},\tau_{\nu}]}\left(\mathcal{E}[\phi{}_{\nu}](t)-\mathcal{E}_{S_{t}}[\phi{}_{\nu}]\right)\lesssim\epsilon_{\nu}^{\frac{1}{2}},
\]
which in particular guarantees smoothness of the extension on all
of $[\varsigma_{\nu},\tau_{\nu}]\times\mathbb{R}^{2}$.

Another consequence of the flux decay estimate (\ref{eq:Extension&Flux0})
that we record here, is the following weighted control:
\begin{equation}
\sup_{t\in[1,2]}\int_{S_{t}}\frac{1}{(t-\left|x\right|+\epsilon_{\nu})^{\frac{1}{2}}}\left(\left|L\phi_{\nu}(t)\right|^{2}+\left|r^{-1}\partial_{\theta}\phi_{\nu}(t)\right|^{2}\right)dx\lesssim1,\label{eq:WeightedControl}
\end{equation}
direct consequence of Lemma 3.2 in \citep{TataruSterbenzWaveReg},
and constitutes an important ingredient in the elimination of sharp
pockets of null energy (see section 6.3 of \citep{TataruSterbenzWaveReg}).

Regarding the interior of the cone, by the previous section we can
pick a monotonically decreasing sequence of scales $\delta_{\nu}\downarrow0$,
starting with $\delta_{0}:=\frac{1}{10}\mathrm{dist}(\cup_{i}\varrho_{i},\partial C_{[1,2]})$,
such that: 
\begin{equation}
\lim_{\nu\rightarrow\infty}\sup_{t_{0}\in[1,2]}\mathcal{E}_{S_{t_{0}}^{\delta_{\nu}}\setminus\cup_{i}B_{\delta_{\nu}}(\varrho_{i}(t_{0}))}[\phi_{\nu}]=0,\label{eq:LowestConcentrationScale}
\end{equation}
which are in some sense the slowest concentration scales, i.e. have
the property that:
\begin{equation}
\phi_{\nu}(t_{0}+\delta_{\nu}t,\varrho_{i}(t_{0})+\delta_{\nu}x)\longrightarrow c_{\phi}\in\mathbb{S}^{n-1}\,\,\,\mathrm{on}\,\,\,([-4,4]\times B_{4})\setminus\varrho_{i},\label{eq:ConstantOutside}
\end{equation}
locally in $C_{t}^{0}(H_{x}^{1})\cap C_{t}^{1}(L_{x}^{2})$, where
the constant $c_{\phi}$ corresponds to the wave map $\phi$ from
(\ref{eq:THING2}), for any given $t_{0}\in(1,2)$ and $i=1,\ldots,I$.
This can be obtained upon taking $\delta_{\nu}$ tending slower to
$0$, which will not break condition (\ref{eq:LowestConcentrationScale}).
Hence, by pigeonholing, we can choose a sequence of radii $\sigma_{\nu}=\sigma_{\nu}(t_{0},i)\in(3,4)$
such that:
\[
\int_{\partial B_{\sigma_{\nu}}}\left|\nabla_{t,x}\left[\phi_{\nu}(t_{0}+\delta_{\nu}t,\varrho_{i}(t_{0})+\delta_{\nu}x)\right]\right|^{2}d\theta\longrightarrow0,
\]
which enables us, as before, to construct extensions into $B_{\sigma_{\nu}}$
that have asymptotically vanishing energy. That is we cut off the
bubbles from the body of the map. More precisely, we choose a sequence
of maps $(\varpi_{i,t_{0},\nu},\partial_{t}\varpi_{i,t_{0},\nu})\in T(\mathbb{S}^{n-1})$
defined on $B_{\sigma_{\nu}}$ such that:
\[
\nabla_{t,x}\left[\phi_{\nu}(t_{0}+\delta_{\nu}\cdot,\varrho_{i}(t_{0})+\delta_{\nu}\cdot)|_{\left\{ t=0\right\} \times B_{4}\setminus B_{\sigma_{\nu}}}+\varpi_{i,t_{0},\nu}(\cdot)\right]\longrightarrow0\,\,\,\mathrm{in}\,\,\, L_{x}^{2}(B_{4}),
\]
and performing this surgery for each $i=1,\ldots,I$, we obtain smooth
maps:
\[
\varpi_{t_{0},\nu}[t_{0}]:=\phi_{\nu}[t_{0}]|_{\mathbb{R}_{x}^{2}\setminus\cup_{i}B_{\delta_{\nu}\sigma_{\nu}}(\varrho_{i}(t_{0}))}+\sum_{i=1}^{I}(\varpi_{i,t_{0},\nu},\frac{1}{\delta_{\nu}}\partial_{t}\varpi_{i,t_{0},\nu})\left(\frac{x-\varrho_{i}(t_{0})}{\delta_{\nu}}\right),
\]
satisfying by construction:
\begin{equation}
\nabla_{t,x}\varpi_{t_{0},\nu}(t_{0})\longrightarrow0\,\,\,\mathrm{in}\,\,\,(L_{x}^{2})_{loc}\left(\mathbb{R}^{2}\setminus\left\{ r=t_{0}\right\} \right).\label{eq:INTERIORDECAY}
\end{equation}

Moreover, fixing $t_{0}\in[1+\delta_{0},2-\delta_{0}]$, we can naturally
view $\varpi_{t_{0},\nu}[t_{0}]$ as defined on the time slice $S_{t_{0}}$,
and solve the wave maps equation with initial data $\varpi_{t_{0},\nu}[t_{0}]$
obtaining a smooth solution on $[t_{0}-\delta_{0},t_{0}+\delta_{0}]$
provided we work with $\nu$ large enough, relying on the finite speed
of propagation property (which tells us that $\varpi_{t_{0},\nu}$
agrees with $\phi_{\nu}$ near and beyond the null boundary, at least
away from $C_{[t_{0}-\delta_{0},t_{0}+\delta_{0}]}^{2\delta_{0}}$),
and the small energy regularity via (\ref{eq:INTERIORDECAY}). The
choice of $\delta_{0}$ is not the most optimal one, but here we are
rather concerned with its independence from $\nu$. It is immediate
then that, 
\[
\nabla_{t,x}\varpi_{t_{0},\nu}\longrightarrow0\,\,\,\mathrm{in}\,\,\, C_{t}^{0}(L_{x}^{2})_{loc}\left(([t_{0}-\delta_{0},t_{0}+\delta_{0}]\times\mathbb{R}^{2})\setminus\partial C_{[t_{0}-\delta_{0},t_{0}+\delta_{0}]}\right),
\]
as desired in Theorem \ref{thm:Main}, and furthermore the weighted
estimate (\ref{eq:WeightedControl}) is inherited by the maps $\varpi_{t_{0},\nu}$:
\begin{equation}
\sup_{t\in[t_{0}-\tau,t_{0}+\tau]}\int_{S_{t}}\frac{1}{(t-\left|x\right|+\epsilon_{\nu})^{\frac{1}{2}}}\left(\left|L\varpi_{t_{0},\nu}(t)\right|^{2}+\left|r^{-1}\partial_{\theta}\varpi_{t_{0},\nu}(t)\right|^{2}\right)dx\lesssim1,\label{eq:NewWeightedControl}
\end{equation}
giving us the possibility to apply the following lemma of Sterbenz
and Tataru from \citep{TataruSterbenzWaveReg} (see sections 6.3 and
6.4 there), to get the energy dispersion norm of $\varpi_{t_{0},\nu}$
asymptotically vanishing and conclude on the second point of Theorem
\ref{thm:Main}.
\begin{lem}
\emph{\label{lem:NoTimeLikeEnergyImplyDispersive}(Sterbenz and Tataru
\citep{TataruSterbenzWaveReg})}. Consider tuples $\{(\varphi_{\nu},\partial_{t}\varphi_{\nu})\}_{\nu\in\mathbb{N}}$
of Schwartz functions on $\mathbb{R}^{2}$ satisfying, for some sequence
$\epsilon_{\nu}\downarrow0$ and a bound $\mathcal{E}>0$:
\[
\left\Vert \nabla_{t,x}\varphi_{\nu}\right\Vert _{L_{x}^{2}}^{2}\lesssim\mathcal{E},\,\,\,\left\Vert \nabla_{t,x}\varphi_{\nu}\right\Vert _{L_{x}^{2}(\mathbb{R}^{2}\setminus B_{1})}^{2}\lesssim\epsilon_{\nu}^{\frac{1}{2}}\mathcal{E},
\]
\[
\int_{B_{1}}\frac{1}{(1-\left|x\right|+\epsilon_{\nu})^{\frac{1}{2}}}\left(\left|L\varphi_{\nu}\right|^{2}+\left|r^{-1}\partial_{\theta}\varphi_{\nu}\right|^{2}\right)dx\lesssim\mathcal{E},
\]
such that for some given $\epsilon>0$:
\[
\sup_{k}\left(2^{-k}\left\Vert P_{k}\nabla_{t,x}\varphi_{\nu}\right\Vert _{L_{x}^{\infty}}\right)>\epsilon.
\]
Then, there exist constants $0<\gamma(\epsilon,\mathcal{E})<1$ and
$\varepsilon(\epsilon,\mathcal{E})>0$ for which:
\[
\int_{B_{\gamma(\epsilon,\mathcal{E})}}\left|\nabla_{t,x}\varphi_{\nu}\right|^{2}dx\geq\varepsilon(\epsilon,\mathcal{E}),\,\,\,\forall\nu\in\mathbb{N}.
\]

\end{lem}

\subsection{Asymptotic decomposition.\label{sub:Asymptotic-decomposition.}}

We have reduced the proof of Theorem \ref{thm:Main} to carrying out
the bubbling analysis for our sequence of wave maps $\left\{ \phi_{\nu}\right\} _{\nu\in\mathbb{N}}$
near the set of time-like energy concentration:
\begin{equation}
(\cup_{i}B_{\delta_{\nu}}(\varrho_{i}))\cap C_{[1,2]}^{0}\subset C^{\delta_{0}},\label{eq:DeltaDistanceFromNull}
\end{equation}
recalling the set-up from Section \ref{sub:Dispersive-property},
where $\delta_{0}>0$ controls the distance to the null boundary $\partial C$
of the light cone, on which dependence of our constants will be considered
universal. The dynamics of the maps $\phi_{\nu}$ near distinct rays
$\varrho_{i}$ are completely disjoint and to get the claimed asymptotic
decomposition from Theorem \ref{thm:Main} we will have to select
the time slices $t_{\nu}$ rather carefully.

To start, in order to obtain from the decay assumption (\ref{eq:AsymptoticSelfSim})
the asymptotic stationarity at all scales for some suitably chosen
time slices, we consider a sequence of positive functions on the time
interval $[1,2]$ defined by: 
\[
\zeta_{\nu}(t):=\int_{S_{t}^{\delta_{0}}}\left|\partial_{\rho}\phi_{\nu}(t)\right|^{2}dx,
\]
so that $\left\Vert \zeta_{\nu}\right\Vert _{L_{t}^{1}[1,2]}\rightarrow0$
by (\ref{eq:AsymptoticSelfSim}). Then, looking at the corresponding
Hardy-Littlewood maximal functions:
\[
\mathcal{M}\zeta_{\nu}(s):=\sup_{r>0}\frac{1}{r}\int_{s-r}^{s+r}\zeta_{\nu}(t)dt,
\]
the well-known maximal inequality of Hardy-Littlewood tells us that
for any $\lambda>0$:
\[
\left|\left\{ \mathcal{M}\zeta_{\nu}>\lambda\right\} \right|\lesssim\frac{1}{\lambda}\left\Vert \zeta_{\nu}\right\Vert _{L_{t}^{1}}.
\]
Therefore taking a sequence $\lambda_{\nu}\sim\left\Vert \zeta_{\nu}\right\Vert _{L_{t}^{1}}^{1/2}\downarrow0$
decaying slowly enough compared to $\left\Vert \zeta_{\nu}\right\Vert _{L_{t}^{1}}$,
we can select a sequence of time slices $\left\{ t_{\nu}\right\} _{\nu\in\mathbb{N}}\subset(1+\delta_{0},2-\delta_{0})$
such that:
\begin{equation}
\mathcal{M}\zeta_{\nu}(t_{\nu})\longrightarrow0.\label{eq:MaxFunSlice}
\end{equation}
We should note here that this will not be quite the final sequence
of time slices we will claim the soliton resolution on as we might
need to perturb it a little at scales $\delta_{\nu}$. 

From there, we have to study for each $i=1,\ldots,I$, a sequence
of wave maps obtained from $\phi_{\nu}$, upon translating by $(t_{\nu},\varrho_{i}(t_{\nu}))$
and rescaling by $\delta_{\nu}$, which gives us by (\ref{eq:ConstantOutside}):
\begin{equation}
\widetilde{\phi}_{i,\nu}(\cdot):=\phi_{\nu}(t_{\nu}+\delta_{\nu}\cdot,\varrho_{i}(t_{\nu})+\delta_{\nu}\cdot)\longrightarrow c_{\phi}\,\,\,\mathrm{on}\,\,\,\left([-4,4]\times B_{4}\right)\setminus\varrho_{i},\label{eq:NewConstantOutside}
\end{equation}
locally in $C_{t}^{0}(H_{x}^{1})\cap C_{t}^{1}(L_{x}^{2})$. Moreover
from (\ref{eq:MaxFunSlice}), denoting by $X_{i}$ the unit constant
time-like vector field pointing in the direction of the line $\varrho_{i}$,
we have:
\begin{equation}
\left\Vert X_{i}\widetilde{\phi}_{i,\nu}\right\Vert _{L_{t,x}^{2}([-4,4]\times B_{4})}\longrightarrow0.\label{eq:ConstantTimeLike}
\end{equation}

Proceeding as in Remark \ref{Rem:UniformityInTime}, we interpolate
smoothly between $\widetilde{\phi}_{i,\nu}[0]$ and the constant initial
data $(c_{\phi},0)\in T(\mathbb{S}^{n-1})$ on $B_{4}\setminus B_{3}$,
replacing the map $\widetilde{\phi}_{i,\nu}$ with a wave map $\phi_{i,\nu}$
agreeing with the latter on $[-\frac{3}{2},\frac{3}{2}]\times B_{3/2}$
and constant outside $B_{6}$ (at most) for $t\in[-\frac{3}{2},\frac{3}{2}]$
by finite speed of propagation. This introduces an error of asymptotically
vanishing energy on this time interval, safely by (\ref{eq:NewConstantOutside}).
In fact, from the construction it is immediate that:
\begin{equation}
\phi_{i,\nu}-c_{\phi}\longrightarrow0\,\,\,\mathrm{in}\,\,\, C_{t}^{0}(L_{x}^{2})[-\frac{3}{2},\frac{3}{2}],\label{eq:StrongConstant}
\end{equation}
which improves to locally in $C_{t}^{0}(H_{x}^{1})\cap C_{t}^{1}(L_{x}^{2})$
away from $\varrho_{i}$, and we still have decay in a time-like direction:
\begin{equation}
\left\Vert X_{i}\phi_{i,\nu}\right\Vert _{L_{t,x}^{2}[-\frac{3}{2},\frac{3}{2}]}\longrightarrow0.\label{eq:LastTimeLikeDecay}
\end{equation}

Let us fix a smooth time cut-off $\chi(t)\in C_{0}^{\infty}(-\frac{3}{2},\frac{3}{2})$,
identically 1 on $[-1,1]$, so that we get now in position to apply
Proposition \ref{prop:Compensation-estimate.}, obtaining from (\ref{eq:CompensatedDecomposition})
the following decomposition:
\[
\chi\nabla_{t,x}\phi_{i,\nu}=\Theta_{i,\nu}+\Xi_{i,\nu},\,\,\,\mathrm{with}:
\]
\begin{equation}
\left\Vert \Theta_{i,\nu}\right\Vert _{L_{t,x}^{2}}\lesssim\left\Vert X_{i}\phi_{i,\nu}\right\Vert _{L_{t,x}^{2}[-\frac{3}{2},\frac{3}{2}]}+\left\Vert \phi_{i,\nu}-c_{\phi}\right\Vert _{L_{t}^{\infty}(L_{x}^{2})[-\frac{3}{2},\frac{3}{2}]},\label{eq:AppliedCompDec}
\end{equation}
\begin{equation}
\sum_{k\in\mathbb{Z}}\left\Vert P_{k}\Xi_{i,\nu}\right\Vert _{L_{t}^{1}(L_{x}^{2})}\lesssim1.\label{eq:AppliedCompDec(ii)}
\end{equation}
Furthermore, applying Lemma \ref{lem:HigherOrderTimeLike}, we get
from (\ref{eq:HigherOrderWeakDecomp}) a decomposition for second
order time-like derivative of $\phi_{i,\nu}$: 
\[
\chi X_{i}^{2}\phi_{i,\nu}=\Gamma_{i,\nu}+\Pi_{i,\nu},
\]
where the first item is a linear combination of:
\[
\sum_{k\in\mathbb{Z}}P_{k}\nabla_{x}[\Omega_{x}^{i,\nu}(P_{>k+10}\phi_{i,\nu})],\,\,\,\sum_{k\in\mathbb{Z}}P_{k}[\Omega_{x}^{i,\nu}(P_{\leq k+10}\nabla_{x}\phi_{i,\nu})],\,\,\,\mathrm{and}\,\,\,\Omega_{t,x}^{i,\nu}\nabla_{t,x}\phi_{i,\nu},
\]
\[
\mathrm{with}\,\,\,\Omega_{\alpha}^{i,\nu}:=\phi_{i,\nu}\partial_{\alpha}\phi_{i,\nu}^{\dagger}-\partial_{\alpha}\phi_{i,\nu}\phi_{i,\nu}^{\dagger},
\]
while the second one satisfies (\ref{eq:HigherOrderWeakEstimate}):
\begin{align}
 & \sum_{k\in\mathbb{Z}}2^{-2k}\left\Vert P_{k}\Pi_{i,\nu}\right\Vert _{L_{t,x}^{2}[-\frac{3}{2},\frac{3}{2}]}^{2}\label{eq:AppliedHighOrderDec}\\
 & \lesssim(1+\left\Vert X_{i}\phi_{i,\nu}\right\Vert _{L_{t,x}^{2}[-\frac{3}{2},\frac{3}{2}]})\left\Vert X_{i}\phi_{i,\nu}\right\Vert _{L_{t,x}^{2}[-\frac{3}{2},\frac{3}{2}]}\nonumber \\
 & \,\,\,\,\,\,\,\,\,+(1+\left\Vert \phi_{i,\nu}-c_{\phi}\right\Vert _{L_{t}^{\infty}(L_{x}^{2})[-\frac{3}{2},\frac{3}{2}]})\left\Vert \phi_{i,\nu}-c_{\phi}\right\Vert _{L_{t}^{\infty}(L_{x}^{2})[-\frac{3}{2},\frac{3}{2}]}.\nonumber 
\end{align}

We note that the implicit constants, including the factors in the
linear combination for $\Gamma_{i,\nu}$, depend only on the energy
bound $\mathcal{E}$ from (\ref{eq:EnergyBounAssumtion}) and the
distance $\delta_{0}$ to the null boundary $\partial C$ from (\ref{eq:DeltaDistanceFromNull}),
hence can be considered universal for the rest of the argument. 

With this understood, we define non-negative functions $\vartheta_{i,\nu}$,
$\xi_{i,\nu}$, $\zeta{}_{i,\nu}$, and $\pi_{i,\nu}$ for $i=1,\ldots,I$
and $t\in[-1,1]$, setting:
\[
\theta_{i,\nu}(t):=\left\Vert \Theta_{i,\nu}(t)\right\Vert _{L_{x}^{2}}^{2}\,\,\,\mathrm{with}\,\,\,\left\Vert \theta_{i,\nu}\right\Vert _{L_{t}^{1}}\longrightarrow0,
\]
\[
\xi_{i,\nu}(t):=\sum_{k\in\mathbb{Z}}\left\Vert P_{k}\Xi_{i,\nu}(t)\right\Vert _{L_{x}^{2}}\,\,\,\mathrm{with}\,\,\,\left\Vert \xi_{i,\nu}\right\Vert _{L_{t}^{1}}=\sum_{k\in\mathbb{Z}}\left\Vert P_{k}\Xi_{i,\nu}(t)\right\Vert _{L_{t}^{1}(L_{x}^{2})}\lesssim1,
\]

\[
\zeta_{i,\nu}(t):=\left\Vert X_{i}\phi_{i,\nu}(t)\right\Vert _{L_{x}^{2}}^{2},\,\,\,\mathrm{so}\,\,\,\mathrm{that}\,\,\,\left\Vert \zeta_{i,\nu}\right\Vert _{L_{t}^{1}}\longrightarrow0,
\]
as well as:
\[
\pi_{i,\nu}(t)=\sum_{k\in\mathbb{Z}}2^{-2k}\left\Vert P_{k}\Pi_{i,\nu}(t)\right\Vert _{L_{x}^{2}}^{2},\,\,\,\mathrm{where}\,\,\,\left\Vert \pi_{i,\nu}\right\Vert _{L_{t}^{1}}=\sum_{k\in\mathbb{Z}}2^{-2k}\left\Vert P_{k}\Pi_{i,\nu}(t)\right\Vert _{L_{t,x}^{2}}^{2}\longrightarrow0,
\]
by (\ref{eq:AppliedCompDec}) and (\ref{eq:AppliedCompDec(ii)}),
(\ref{eq:ConstantTimeLike}) and (\ref{eq:StrongConstant}), and finally
(\ref{eq:AppliedHighOrderDec}).

We will now choose a sequence of time slices where we uniformly control
$\theta_{i,\nu}$ and have all of the other functions above asymptotically
decaying. This will be used to prove decay of the weak Besov norm
$\dot{B}_{\infty}^{1,2}$ on the neck regions, and ultimately get
the energy collapsing there via the control on $\theta_{i,\nu}$.
At the same time, to start this argument, we shall build first the
weak bubble tree decomposition. To do so, one relies on the small
energy compactness result from Lemma \ref{prop:Simple-compactness-result.}
(which, for example, enables one to extract solitons from the standard
concentration-compactness procedure). Hence, for that reason, we will
need to control the maximal function $\mathcal{M}\zeta_{i,\nu}$ corresponding
to $\left\Vert X_{i}\phi_{i,\nu}(t)\right\Vert _{L_{x}^{2}}^{2}$
as well. 

Let us take $\lambda_{\nu}^{\theta_{i}}\sim\left\Vert \theta_{i,\nu}\right\Vert _{L_{t}^{1}}^{1/2}\downarrow0$,
$(\lambda_{\nu}^{\xi_{i}})^{-1}\left\Vert \xi_{i,\nu}\right\Vert _{L_{t}^{1}}<\epsilon$
for some arbitrarily small $\epsilon>0$ to be fixed according to
(\ref{eq:PREGOODSLICES}) below, as well as $\lambda_{\nu}^{\zeta_{i}}\sim\left\Vert \zeta_{i,\nu}\right\Vert _{L_{t}^{1}}^{1/2}\downarrow0$
and $\lambda_{\nu}^{\pi_{i}}\sim\left\Vert \pi_{i,\nu}\right\Vert _{L_{t}^{1}}^{1/2}\downarrow0$.
Hence, applying Chebyshev's inequality and the maximal inequality
of Hardy-Littlewood for $\mathcal{M}\zeta_{i,\nu}$, we get:
\begin{align}
 & \sum_{i=1}^{I}\left(\left|\left\{ \theta_{i,\nu}>\lambda_{\nu}^{\theta_{i}}\right\} \right|+\left|\left\{ \xi_{i,\nu}>\lambda_{\nu}^{\xi_{i}}\right\} \right|+\left|\left\{ \zeta_{i,\nu}>\lambda_{\nu}^{\zeta_{i}}\right\} \right|+\left|\left\{ \pi_{i,\nu}>\lambda_{\nu}^{\pi_{i}}\right\} \right|\right.\label{eq:PREGOODSLICES}\\
 & \,\,\,\,\,\,\,\left.+\left|\left\{ \mathcal{M}\zeta_{i,\nu}>\lambda_{\nu}^{\zeta_{i}}\right\} \right|\right)<\frac{1}{10}.\nonumber 
\end{align}
Therefore, we can choose a sequence of time slices $\left\{ t_{\nu}\right\} _{\nu\in\mathbb{N}}\subset[-\frac{1}{2},\frac{1}{2}]$,
that we may assume simply to be $t_{\nu}=0$ upon translating the
maps $\phi_{i,\nu}$ by $(t_{\nu},\varrho_{i}(t_{\nu}))$ without
changing notation for $\phi_{i,\nu}$ (and working on $[-\frac{1}{2},\frac{1}{2}]\times B_{6}$),
such that for all $i=1,\ldots,I$ we have the following control:
\begin{equation}
\theta_{i,\nu}(0)\longrightarrow0,\,\,\,\xi_{i,\nu}(0)\lesssim1,\,\,\,\zeta_{i,\nu}(0)\longrightarrow0,\,\,\,\pi_{i,\nu}(0)\longrightarrow0,\label{eq:GoodSliceProps}
\end{equation}
\[
\mathrm{and}\,\,\,\mathcal{M}\zeta_{i,\nu}(0)\longrightarrow0.
\]
These are the final time slices that we will consider and obtain the
asymptotic decomposition on, as claimed in our main theorem. We start
doing bubbling analysis on them just below. Here we just add the remark
that, upon working in (\ref{eq:PREGOODSLICES})-(\ref{eq:GoodSliceProps})
with the maximal functions for $\theta_{i,\nu}$, $\xi_{i,\nu}$ and
$\pi_{i,\nu}$ as well, it should be clear by end of the argument
that we can also get the energy collapsing result for almost every
time slice strictly within the lifespan of the fastest concentrating
solitons.

In the following lemma we present a preliminary version of the soliton
decomposition. It is essentially the one that we aim towards from
Theorem \ref{thm:Main}, but it contains errors that we shall call
\emph{necks} - those are wave maps on conformally degenerating annuli
such that once localized in space converge to a constant but when
considered on the whole annulus might carry a priori a non-trivial
amount of energy. Ruling out such a scenario will be the last step
in the proof of the main theorem. 

We note that the proof of this lemma relies on a covering argument
which goes back to at least Ding and Tian \citep{DingTian} and today
is pretty standard in the literature on bubbling analysis of harmonic
maps (and related areas, where some authors refer to as weak bubble
tree convergence). The lemma of course holds for any closed Riemannian
manifold as a target.
\begin{lem}
\label{lem:CoveringLemma}Passing to a subsequence, there exists for
each $i=1,\ldots,I$ a collection of $J_{i}\lesssim_{\mathcal{E}}1$
solitons $\omega_{j,i}$, $j=1,\ldots,J_{i}$, with corresponding
concentration points $a_{i,\nu}^{j}\in B_{1}$ converging to the origin,
and scales $\lambda_{i,\nu}^{j}\downarrow0$ satisfying the orthogonality
relations:
\begin{equation}
\frac{\lambda_{i,\nu}^{j}}{\lambda_{i,\nu}^{j'}}+\frac{\lambda_{i,\nu}^{j'}}{\lambda_{i,\nu}^{j}}+\frac{|a_{i,\nu}^{j}-a_{i,\nu}^{j'}|^{2}}{\lambda_{i,\nu}^{j}\lambda_{i,\nu}^{j'}}\longrightarrow\infty,\label{eq:OrthogonalityRelation}
\end{equation}
as $\nu\rightarrow\infty$ for $j$ and $j'$ distinct, such that:
\[
\phi_{i,\nu}(\lambda_{i,\nu}^{j}t,a_{\nu}^{j}+\lambda_{i,\nu}^{j}x)\longrightarrow\omega_{j,i}(t,x)\,\,\,\mathit{\mathit{on}}\,\,\,\mathbb{R}^{2+1}\setminus\cup_{q}\varrho_{q}^{j,i},
\]
locally in $C_{t}^{0}(H_{x}^{1})\cap C_{t}^{1}(L_{x}^{2})$ for a
collection of at most $J_{i}-1$ time-like geodesics $\varrho_{q}^{i,j}$
with direction $X_{i}$. Moreover, setting for any fixed positive
constant $C>0$:
\[
\lambda_{\mathrm{min},\nu}:=C\cdot\min_{i,j}\left\{ \lambda_{i,\nu}^{j}\right\} ,\,\,\,\nu\in\mathbb{N},
\]
we have the following asymptotic decomposition holding for $t\in[-\lambda_{\mathrm{min},\nu},\lambda_{\mathrm{min},\nu}]$:
\begin{equation}
\phi_{i,\nu}(t,x)-c_{\phi}=\sum_{j=1}^{J_{i}}\left(\omega_{j,i}\left(\frac{t}{\lambda_{i,\nu}^{j}},\frac{x-a_{i,\nu}^{j}}{\lambda_{i,\nu}^{j}}\right)-\omega_{j,i}(\infty)\right)+\mathcal{N}_{i,\nu}(t,x)+o_{L_{t}^{\infty}(\dot{H}_{x}^{1}\times L_{x}^{2})}(1),\label{eq:DecomWithNecks}
\end{equation}
where $\mathcal{N}_{i,\nu}$ stands for the wave map $\phi_{i,\nu}$
restricted to a collection of $K_{i}\lesssim_{\mathcal{E}}1$ sequences
of degenerating annuli:
\begin{equation}
[-\frac{r_{i,\nu}^{k}}{2},\frac{r_{i,\nu}^{k}}{2}]\times\left(B_{R_{i,\nu}^{k}}(x_{i,\nu}^{k})\setminus B_{r_{i,\nu}^{k}}(x_{i,\nu}^{k})\right)\subset[-\frac{1}{2},\frac{1}{2}]\times B_{3}\,\,\,\mathit{with}\,\,\,\lambda_{\mathrm{min},\nu}\ll r_{i,\nu}^{k}\ll R_{i,\nu}^{k},\label{eq:NeckDomains}
\end{equation}
such that we have:
\begin{equation}
\sup_{r_{i,\nu}^{k}\leq r\leq\frac{1}{2}R_{i,\nu}^{k}}\sup_{t\in[-\frac{r}{2},\frac{r}{2}]}\mathcal{E}_{B_{2r}(x_{i,\nu}^{k})\setminus B_{r}(x_{i,\nu}^{k})}[\phi_{i,\nu}](t)\longrightarrow0,\label{eq:NeckProperty}
\end{equation}
holding for each $k=1,\ldots,K_{i}$.\end{lem}
\begin{proof}
Let us fix $i=1,\ldots,I$, and suppress this subscript in the argument
below to lighten the notation. In the same spirit, we also never change
notation here whenever passing to a subsequence for $\{\phi_{\nu}\}_{\nu\in\mathbb{N}}$
while using Lemma \ref{prop:Simple-compactness-result.} as it will
be clear from the construction that we obtain in the end a countable
cover of a suitable neighborhood of $\left\{ t=0\right\} \times B_{3}$
on which we can rely to build via the diagonal process a final subsequence
that satisfies the claims of Lemma \ref{lem:CoveringLemma}.

Pick a sequence of points $a_{\nu}^{1}\in B_{1}$ with radii $\lambda_{\nu}^{1}\downarrow0$
such that:
\begin{equation}
\mathcal{E}_{B_{2\lambda_{\nu}^{1}}(a_{\nu}^{1})}[\phi_{\nu}](0)=\epsilon_{s}.\label{eq:FirstSoliton}
\end{equation}
Note that such a concentration point is guaranteed to exist by the
results of Section \ref{sub:Concentration-compactness} and the compactness
Lemma \ref{prop:Simple-compactness-result.}, and that by (\ref{eq:NewConstantOutside})
any energy concentration point would have to converge to the origin. 

Consider the sequence of balls $B_{2^{k}\lambda_{\nu}^{1}}(a_{\nu}^{1})$
with $k$ a positive integer, and choose the lowest $K_{0}=K_{0}(\{a_{\nu}^{1}\}_{\nu\in\mathbb{N}})\in\mathbb{N}$
such that the functions:
\[
r_{\nu}:\overline{B_{2^{K_{0}-1}}(a_{\nu}^{1})}\longrightarrow\mathbb{R}_{>0}
\]
\begin{equation}
x\longmapsto r_{\nu}(x):=\sup\left\{ r>0\,:\,\mathcal{E}_{B_{r}(x)}[\phi_{\nu}(\lambda_{\nu}^{1}\cdot,a_{\nu}^{1}+\lambda_{\nu}^{1}\cdot)](0)\leq\epsilon_{s}\right\} ,\label{eq:rnuFunctions}
\end{equation}
which are continuous as the wave maps $\phi_{\nu}$ are smooth, admit
a collective positive lower bound $r':=\liminf_{x,\nu}r_{\nu}(x)>0$
(assuming $K_{0}$ exists, the case when it does not is treated later
when we describe convergence to solitons at infinity). As a preliminary
step, relying on (\ref{eq:GoodSliceProps}) and the compactness Lemma
\ref{prop:Simple-compactness-result.}, we can obtain for the rescalings
of the maps $\phi_{\nu}$ at $a_{\nu}^{1}$, upon passing to a subsequence,
that:
\[
\phi_{\nu}(\lambda_{\nu}^{1}t,a_{\nu}^{1}+\lambda_{\nu}^{1}x)\longrightarrow\omega_{1}(t,x)\,\,\,\mathrm{on}\,\,\,[-\frac{r'}{3},\frac{r'}{3}]\times B_{2^{K_{0}-1}},
\]
in $C_{t}^{0}(H_{x}^{1})\cap C_{t}^{1}(L_{x}^{2})$ for some wave
map $\omega_{1}$ with regularity as in (\ref{eq:LimitRegularity})
and satisfying $X\omega_{1}=0$, with $X$ standing for the constant
time-like vector field $X_{i}$ from (\ref{eq:ConstantTimeLike}).
Therefore, the map $\omega_{1}$ is part of a soliton.

The time interval $[-\frac{r'}{3},\frac{r'}{3}]$ for the convergence
above will be improved considerably below by recalling the methods
from Section \ref{sub:Concentration-compactness}, see the proof of
(\ref{eq:ConstructionOfTheSoliton}). Now, we shall proceed instead
describing further $\omega_{1}$ in space. Slightly abusing terminology,
let us refer to $\omega_{1}$ as a soliton already from here, bearing
in mind that we will prove it is one shortly.

By construction, we can find at least one sequence of concentration
points:
\begin{equation}
a_{\nu}^{2}\in B_{2^{K_{0}}\lambda_{\nu}^{1}}(a_{\nu}^{1})\setminus B_{2^{K_{0}-1}\lambda_{\nu}^{1}}(a_{\nu}^{1}),\label{eq:NewSoliton0}
\end{equation}
bubbling off on the top of the soliton $\omega_{1}$ in the sense
that:
\begin{equation}
\mathcal{E}_{B_{2\lambda_{\nu}^{2}}(a_{\nu}^{2})}[\phi_{\nu}](0)=\epsilon_{s},\,\,\,\lambda_{\nu}^{2}\ll\lambda_{\nu}^{1},\label{eq:NewSoliton}
\end{equation}
where it is quite important to note that we have an equality above,
a fact that must hold by the compactness Lemma \ref{prop:Simple-compactness-result.}.

Let us consider a new sequence of concentration points satisfying
(\ref{eq:NewSoliton0}) and (\ref{eq:NewSoliton}) like $\left\{ a_{\nu}^{2}\right\} _{\nu\in\mathbb{N}}$,
in other words forming itself above the scales $\lambda_{\nu}^{1}$
and converging, upon passing to a subsequence, in the closure of $B_{2^{K_{0}}\lambda_{\nu}^{1}}(a_{\nu}^{1})$,
so that it suffices to work in $B_{2^{K_{0}+1}\lambda_{\nu}^{1}}(a_{\nu}^{1})$.
There are of course uncountably many of those, given the existence
of a single one, $\left\{ a_{\nu}^{2}\right\} _{\nu\in\mathbb{N}}$,
but we are going to consider equivalent all those for which the orthogonality
condition (\ref{eq:OrthogonalityRelation}) holds and pick only one
representative per equivalence class. That is, if a sequence $\left\{ a'_{\nu}\right\} _{\nu\in\mathbb{N}}$
satisfies (\ref{eq:NewSoliton0}) and (\ref{eq:NewSoliton}) but in
addition also has $\lambda'_{\nu}\sim\lambda_{\nu}^{2}$ with:
\[
\frac{\left|a_{\nu}^{2}-a'_{\nu}\right|}{\lambda_{\nu}^{2}}\lesssim1,
\]
then one can see that the maps $\phi_{\nu}(\lambda_{\nu}^{2}t,a_{\nu}^{2}+\lambda_{\nu}^{2}x)$
and $\phi_{\nu}(\lambda'_{\nu}t,a'_{\nu}+\lambda'_{\nu}x)$ would
converge on $[-2^{-1},2^{-1}]\times B_{2^{-1}}$, upon passing to
a subsequence directly by Lemma \ref{prop:Simple-compactness-result.},
to the same soliton up to translation that we should denote by $\omega_{2}$
as it was initially obtained from $a_{\nu}^{2}$ once the procedure
we are describing now for the soliton $\omega_{1}$ is completed and
applied to the soliton $\omega_{2}$. Hence the sequence $\left\{ a'_{\nu}\right\} _{\nu\in\mathbb{N}}$
should be discarded keeping $\left\{ a_{\nu}^{2}\right\} _{\nu\in\mathbb{N}}$.

Given the orthogonality relations (\ref{eq:OrthogonalityRelation})
holding between any two sequences of concentration points as above,
we note that we are left with only finitely many possibilities, say
$\{a_{\nu}^{j}\}_{\nu\in\mathbb{N}}$ with $j=2,\ldots,J'$. This
follows from the fact we are considering a sequence of functions $\left\{ \nabla_{t,x}\phi_{\nu}\right\} _{\nu\in\mathbb{N}}\subset L_{x}^{2}$,
bounded by the global energy control assumption (\ref{eq:EnergyBounAssumtion}),
and with $\nabla_{t,x}\phi_{\nu}$ concentrating definite amounts
of its $L_{x}^{2}$ norm, namely $\sqrt{\epsilon_{s}}$, note the
\emph{equality} in (\ref{eq:NewSoliton}), at different frequency
and/or spatial scales so that we can conclude that, since $L_{x}^{2}$
is a Hilbert space, we should have: 
\[
J'\lesssim\frac{1}{\epsilon_{s}}\mathcal{E},
\]
which is a universal bound for us as desired.

The collection $\{a_{\nu}^{j}\}_{\nu\in\mathbb{N}}$, $j=2,\ldots,J'$,
gives rise to solitons $\omega_{j}$, one for each $j$, by the same
procedure as described for $\omega_{1}$ and so from now on we should
be running for each of them the same construction as we are currently
considering for $\omega_{1}$.

From the point of view of $\omega_{1}$, we can subdivide the above
collection of sequences of energy concentration points into disjoint
families by considering the limit points $b_{q}^{1}\in\overline{B_{2^{K_{0}}\lambda_{\nu}^{1}}(a_{\nu}^{1})}\setminus B_{2^{K_{0}-1}\lambda_{\nu}^{1}}(a_{\nu}^{1})$,
indexed by $q=1,\ldots,Q'$ for some integer $Q'\leq J'$, to which
the sequences converge once rescaled by $\lambda_{\nu}^{1}$. So for
any $r>0$ small but fixed, we have by Lemma \ref{prop:Simple-compactness-result.}:
\[
\phi_{\nu}(\lambda_{\nu}^{1}t,a_{\nu}^{1}+\lambda_{\nu}^{1}x)\longrightarrow\omega_{1}(t,x)\,\,\,\mathrm{on}\,\,\,[-\frac{r'}{3},\frac{r'}{3}]\times\left(\overline{B_{2^{K_{0}}}}\setminus\cup_{q}B_{r}(b_{q}^{1})\right),
\]
in $C_{t}^{0}(H_{x}^{1})\cap C_{t}^{1}(L_{x}^{2})$ since the functions
$r_{\nu}$ from (\ref{eq:rnuFunctions}) extended to $\overline{B_{2^{K_{0}}}}\setminus\cup_{q}B_{r}(b_{q}^{1})$
admit a collective lower bound $r':=\liminf_{x,\nu}r_{\nu}(x)>0$
(provided $r>0$ is fixed of course as $r'$ depends on it). Understanding
the behavior of the maps $\phi_{\nu}$ as $r\downarrow0$ is linked
to the convergence of $\phi_{\nu}$ to solitons at the spatial infinity
and this is when the neck domains enter into our picture. We shall
discuss this straight after we finish the construction of the soliton
$\omega_{1}$ (and so for the other ones, $\omega_{j}$ above, in
parallel).

Considering the annuli $\overline{B_{2^{K_{0}+k}\lambda_{\nu}^{1}}(a_{\nu}^{1})}\setminus B_{2^{K_{0}+k-1}\lambda_{\nu}^{1}}(a_{\nu}^{1})$
one after the other and studying as above whether there are new sequences
of concentration points satisfying (\ref{eq:NewSoliton}), upgrading
the collection $\{a_{\nu}^{j}\}_{\nu\in\mathbb{N}}$, $j=2,\ldots,J'$,
accordingly upon checking the orthogonality relation (\ref{eq:OrthogonalityRelation})
holds for each new member (we should not change the notation for the
upgraded version), we must a reach an integer $K_{1}=K_{1}(\{a_{\nu}^{1}\}_{\nu\in\mathbb{N}},\epsilon_{s},\mathcal{E})\in\mathbb{N}$
such that for any $k\geq K_{1}$ the functions $r_{\nu}$ from (\ref{eq:rnuFunctions})
once considered on $\overline{B_{2^{k}}(a_{\nu}^{1})}\setminus B_{2^{k-1}}(a_{\nu}^{1})$
would admit a positive collective lower bound there. Note that this
situation could have occurred without passing by the previous bubbling
analysis induced by the existence of the integer $K_{0}$, e.g. if
we would have picked up the fastest concentrating soliton initially
for $\omega_{1}$.

From there, we let $k\rightarrow\infty$ with $r\downarrow0$ and
fully construct the soliton $\omega_{1}$ in the sense that we claim:
\begin{equation}
\phi_{\nu}(\lambda_{\nu}^{1}t,a_{\nu}^{1}+\lambda_{\nu}^{1}x)\longrightarrow\omega_{1}(t,x)\,\,\,\mathrm{on}\,\,\,\mathbb{R}^{2+1}\setminus\cup_{q}\varrho_{q}^{1},\label{eq:ConstructionOfTheSoliton}
\end{equation}
locally in $C_{t}^{0}(H_{x}^{1})\cap C_{t}^{1}(L_{x}^{2})$ for a
finite collection of geodesics $\varrho_{q}^{1}$, $q=1,\ldots,Q'$,
each passing through the corresponding point $b_{q}^{1}$, all with
direction $X$, and such that $X\omega_{1}=0$ there. To prove (\ref{eq:ConstructionOfTheSoliton}),
we note that by (\ref{eq:GoodSliceProps}), used already above, we
have for any fixed bounded time interval the following decay estimate:
\begin{equation}
\int_{-s}^{s}\int_{\mathbb{R}^{2}}\left|X[\phi_{\nu}(\lambda_{\nu}^{1}t,a_{\nu}^{1}+\lambda_{\nu}^{1}x)]\right|^{2}dxdt=o(s),\label{eq:TimeLikeDecayScaled}
\end{equation}
and so denoting by $\Psi$ the Lorentz boost taking $\partial_{t}$
to $X$, if one considers the foliation induced by $\left\{ \Psi(\{t\}\times\mathbb{R}^{2})\right\} _{t\in\mathbb{R}}$
on the whole of Minkowski space $\mathbb{R}^{2+1}$ instead of the
CMC foliation in the interior of the forward light cone as in Lemmata
\ref{lem:MonotonicityLemma} and \ref{lem:ConcentrationCompactness},
the very same arguments would lead to the convergence claimed in (\ref{eq:ConstructionOfTheSoliton}).
Let us present some details, setting $\varphi_{\nu}=\phi_{\nu}(\lambda_{\nu}^{1}\cdot,a_{\nu}^{1}+\lambda_{\nu}^{1}\cdot)$.

Working on $\Psi^{-1}(\mathbb{R}^{2+1})$ we denote the coordinates
there by $x^{\bar{\alpha}}$, or $(\bar{t},\bar{x}_{1},\bar{x}_{2})$,
and writing $\bar{\varphi}_{\nu}:=\varphi_{\nu}\circ\Psi$ we get
by the Lorentz invariance of smooth wave maps that the associated
stress energy tensor $T_{\bar{\alpha}\bar{\beta}}[\bar{\varphi}_{\nu}]$
enjoys the conservation law $\partial^{\bar{\alpha}}T_{\bar{\alpha}\bar{\beta}}[\bar{\varphi}_{\nu}]=0$.
So, contracting $T[\bar{\varphi}_{\nu}]$ with the vector field $\chi(\bar{x})\partial_{\bar{t}}$,
for some continuously differentiable test function $\chi$ with $\partial_{\bar{t}}\chi=0$,
and integrating the divergence of the Noether current $\partial^{\bar{\alpha}}({}^{(\chi(\bar{x})\partial_{\bar{t}})}P_{\bar{\alpha}})$
over the strip $\bar{t}\in[t,t+\lambda]$ for any $t\in\mathbb{R}$
and positive constant $\lambda>0$ (similar considerations apply when
$\lambda<0$), we get by Stokes' theorem and the mentioned conservation
law:
\[
\int_{\{\bar{t}=t+\lambda\}}\left|\nabla_{\bar{t},\bar{x}}\bar{\varphi}_{\nu}\right|^{2}\chi d\bar{x}-\int_{\{\bar{t}=t\}}\left|\nabla_{\bar{t},\bar{x}}\bar{\varphi}_{\nu}\right|^{2}\chi d\bar{x}
\]
\[
=-2\int_{[t,t+\lambda]\times\mathbb{R}_{\bar{x}}^{2}}\partial_{\bar{t}}\bar{\varphi}_{\nu}^{\dagger}(\partial_{\bar{x}_{1}}\bar{\varphi}_{\nu}\partial_{\bar{x}_{1}}\chi+\partial_{\bar{x}_{2}}\bar{\varphi}_{\nu}\partial_{\bar{x}_{2}}\chi)d\bar{t}d\bar{x}.
\]

Hence, integrating the above identity over $t\in[t_{0},t_{1}]$ for
given $t_{0},t_{1}\in\mathbb{R}$, using the decay (\ref{eq:TimeLikeDecayScaled})
we obtain:
\begin{equation}
\int_{[t_{0},t_{1}]\times\mathbb{R}_{\bar{x}}^{2}}\left|\nabla_{\bar{t},\bar{x}}\bar{\varphi}_{\nu}\right|^{2}\chi d\bar{t}d\bar{x}-\int_{[t_{0}+\lambda,t_{1}+\lambda]\times\mathbb{R}_{\bar{x}}^{2}}\left|\nabla_{\bar{t},\bar{x}}\bar{\varphi}_{\nu}\right|^{2}\chi d\bar{t}d\bar{x}\longrightarrow0,\label{eq:AsymptoticMonotonicitySoliton}
\end{equation}
analogously to (\ref{eq:MonotonicityFormula-1}) from Lemma \ref{lem:MonotonicityLemma}.
To use this asymptotic monotonicity formula to propagate small energy
control, we note that we have $\left|\nabla_{\bar{t},\bar{x}}\bar{\varphi}_{\nu}\right|\sim\left|\nabla_{t,x}\varphi_{\nu}\right|$
with the implicit constant depending only on $X$, which is constant
and fixed. Therefore, proceeding as in Lemma \ref{lem:ConcentrationCompactness},
given any point $y\in\mathbb{R}^{2}\setminus\cup_{q}b_{q}^{1}$ and
a positive constant $\eta>0$, there exists a radius $r_{1}=r_{1}(y,\eta)>0$
such that:
\[
\sup_{\nu\in\mathbb{N}}\sup_{t\in[-r_{1},r_{1}]}\mathcal{E}_{B_{r_{1}}(y)}[\varphi_{\nu}](t)\leq\eta,
\]
which leads to the control:
\[
\sup_{\nu\in\mathbb{N}}\frac{1}{r_{1}}\int_{[-r_{1},r_{1}]\times B_{r_{1}}(y)}\left|\nabla_{t,x}\varphi_{\nu}\right|^{2}dtdx\lesssim\eta,
\]
that in turn gives us, precomposing with $\Psi$ and shrinking suitably
the radius to $r_{1}>r_{2}\gtrsim r_{1}$:
\[
\sup_{\nu\in\mathbb{N}}\frac{1}{r_{2}}\int_{[-r_{2}+\bar{s},r_{2}+\bar{s}]\times B_{r_{2}}(\bar{y})}\left|\nabla_{\bar{t},\bar{x}}\bar{\varphi}_{\nu}\right|^{2}d\bar{t}d\bar{x}\lesssim\eta,
\]
where $(\bar{s},\bar{y}):=\Psi^{-1}(0,y)$. By the decay estimate
(\ref{eq:AsymptoticMonotonicitySoliton}), we get that given any $\lambda\in\mathbb{R}$:
\[
\limsup_{\nu\in\mathbb{N}}\frac{1}{r_{2}}\int_{[-r_{2}+\bar{s}+\lambda,r_{2}+\bar{s}+\lambda]\times B_{r_{2}}(\bar{y})}\left|\nabla_{\bar{t},\bar{x}}\bar{\varphi}_{\nu}\right|^{2}d\bar{t}d\bar{x}\lesssim\eta,
\]
and so going back to $\varphi_{\nu}$ by precomposing with $\Psi^{-1}$,
shrinking further the radius to $r_{2}>r_{3}\gtrsim r_{2}$ we obtain
by the pigeonhole principle, using the energy flux identity (\ref{eq:Energy-FluxIdentity}),
the estimate:
\[
\limsup_{\nu\in\mathbb{N}}\left\Vert \nabla_{t,x}\varphi_{\nu}\right\Vert _{L_{t}^{\infty}(L_{x}^{2})\left(([-r_{3},r_{3}]\times B_{r_{3}}(y))+\lambda X\right)}\lesssim\eta,
\]
for any given $\lambda\in\mathbb{R}$, viewing naturally $X\in\mathbb{R}^{2+1}$.
All the implicit constants above being independent of $\eta$ (and
of $\lambda$, the dependence on which of our construction is hidden
in the limsup), we can choose $\eta$ small enough obtaining the small
energy control for any fixed $\lambda\in\mathbb{R}$:
\begin{equation}
\limsup_{\nu\in\mathbb{N}}\left\Vert \nabla_{t,x}\varphi_{\nu}\right\Vert _{L_{t}^{\infty}(L_{x}^{2})\left(([-r_{3},r_{3}]\times B_{r_{3}}(y))+\lambda X\right)}\leq\frac{1}{2}\epsilon_{s},\label{eq:SmallEnergyControlSoliton}
\end{equation}
with the radius $r_{3}=r_{3}(y)$. Therefore, picking suitable collections
of points $y\in\mathbb{R}^{2}\setminus\cup_{q}b_{q}^{1}$ and constants
$\lambda\in\mathbb{R}$, we construct a countable cover of $\mathbb{R}^{2+1}\setminus\cup_{q}\varrho_{q}^{1}$
such that relying on the estimates (\ref{eq:TimeLikeDecayScaled})
and (\ref{eq:SmallEnergyControlSoliton}) we can apply Lemma \ref{prop:Simple-compactness-result.}
to get a subsequence via the diagonal process for which the local
convergence claim (\ref{eq:ConstructionOfTheSoliton}) holds a desired.

Note that by construction $\omega_{1}$ has energy bounded by $\mathcal{E}$,
and so precomposing it with the Lorentz boost $\Psi$ we get a steady
in time finite energy harmonic map from $\mathbb{R}^{2}$ minus a
finite set of points (note that the energy of this harmonic map will
be smaller or equal to $\mathcal{E}[\omega_{1}]$, nothing travels
faster than light!). By the regularity theory of Hélein \citep{Helein}
the latter has to be smooth and by the removable singularity theorem
of Sacks and Uhlenbeck \citep{SacksUhlenbeck}, it extends smoothly
across the singular points. The outcome of this argument is therefore
that $\omega_{1}$ is a smooth finite energy wave map defined on the
whole of $\mathbb{R}^{2+1}$ with $X\omega_{1}=0$, i.e. a genuine
soliton as desired. 

The same holds of course for the solitons $\omega_{j}$, $j=2,\ldots,J'$,
but note that those do not of course constitute all the members of
the decomposition (\ref{eq:DecomWithNecks}) as parts of the maps
$\phi_{\nu}$ can get lost a priori at spatial infinity and in between
the solitons we are considering. We shall address this issue now.

Consider the scales $\left\{ \lambda_{\nu}^{1}\right\} _{\nu\in\mathbb{N}}$
corresponding to the soliton $\omega_{1}$. Fix an arbitrary small
$0<\varepsilon<\epsilon_{s}$, then by the pigeonhole principle there
exist an integer $K(\varepsilon)\geq K_{1}$ such that for any $k\in\mathbb{N}$
fixed: 
\begin{equation}
\mathcal{E}_{B_{2^{K(\varepsilon)+k}\lambda_{\nu}^{1}}(a_{\nu}^{1})\setminus B_{2^{K(\varepsilon)+k-1}\lambda_{\nu}^{1}}(a_{\nu}^{1})}[\phi_{\nu}](0)<\varepsilon,\label{eq:NeckFirstStep}
\end{equation}
for all $\nu$ large enough. Suppose that there exist a sequence of
smallest integer $k_{\nu}(\varepsilon)\geq K(\varepsilon)$, as $\nu$
gets large, such that the above inequality fails:
\[
\mathcal{E}_{B_{2^{k_{\nu}(\varepsilon)+1}\lambda_{\nu}^{1}}(a_{\nu}^{1})\setminus B_{2^{k_{\nu}(\varepsilon)}\lambda_{\nu}^{1}}(a_{\nu}^{1})}[\phi_{\nu}](0)\geq\varepsilon,
\]
and note that by construction we must have $k_{\nu}(\varepsilon)\rightarrow\infty$;
then we have found a new soliton on the top of which our previous
$\omega_{1}$ is concentrating, that we should denote by $\omega_{J'+1}$
so that setting $\lambda_{\nu}^{J'+1}:=2^{k_{\nu}(\varepsilon)-1}\lambda_{\nu}^{1}$
we can apply directly Lemma \ref{prop:Simple-compactness-result.},
by the choice of $k_{\nu}(\varepsilon)$ and (\ref{eq:NeckFirstStep}),
to get:
\[
\phi_{\nu}(\lambda_{\nu}^{J'+1}t,a_{\nu}^{1}+\lambda_{\nu}^{J'+1}x)\longrightarrow\omega_{J'+1}(t,x)\,\,\,\mathrm{in}\,\,\, C_{t}^{0}(H_{x}^{1})\cap C_{t}^{1}(L_{x}^{2})([-\frac{1}{4},\frac{1}{4}]\times(B_{1}\setminus B_{\frac{1}{2}})),
\]
\[
\mathrm{with}\,\,\,\,\mathcal{E}_{B_{4}\setminus B_{2}}[\phi_{\nu}(\lambda_{\nu}^{J'+1}\cdot,a_{\nu}^{1}+\lambda_{\nu}^{J'+1}\cdot)](0)\geq\varepsilon,
\]
and the analysis we carried for $\omega_{1}$ so far should also be
applied to $\omega_{J'+1}$ now. 

It should be clear that if no $k_{\nu}(\varepsilon)$ as above exist,
i.e. (\ref{eq:NeckFirstStep}) is not violated for any $k\in\mathbb{N}$
for $\nu$ large, then choosing $0<\varepsilon<\epsilon_{s}$ small
enough initially, by equality in (\ref{eq:FirstSoliton}) we must
have been working with $\omega_{1}$ and there should exist then a
sequence of integers $k'_{\nu}$ such that $2^{k'_{\nu}}\lambda_{\nu}^{1}\sim1$
and (\ref{eq:NeckFirstStep}) holding for any $k=1,\ldots,k'_{\nu}-K(\varepsilon)$,
with any $0<\varepsilon'<\varepsilon$ for larger $k\geq k'_{\nu}-K(\varepsilon)$
by (\ref{eq:NewConstantOutside}) as $\nu\rightarrow\infty$. The
map $\omega_{J'+1}$ would be standing for the constant $c_{\phi}$
in this case. 

For the other solitons $\omega_{j}$, with $j\geq2$, $k_{\nu}(\varepsilon)$
must exist and we could of course end up with $\omega_{1}$, or also
a constant (to which some authors refer to as a \emph{ghost bubble},
i.e. a soliton on the top of which two or more non-constant solitons
are concentrating but itself is constant) in which case we obviously
do not consider this as a new soliton. This brings us to the final
steps in the proof of Lemma \ref{lem:CoveringLemma}. 

In fact, in the above construction the constant $\varepsilon>0$ could
be arbitrarily small but was initially fixed and we would like now
to let it degenerate to $0$. We claim that in fact we can put ourselves
in a situation when for any smaller $0<\varepsilon'<\varepsilon$
the choice of the integers $k_{\nu}(\varepsilon')\in\mathbb{N}$ is
uniform in the sense that there exist positive integers $L(\varepsilon')\in\mathbb{N}$
independent of $\nu$ such that $k_{\nu}(\varepsilon')=k_{\nu}(\varepsilon)-L(\varepsilon')$,
that is:
\begin{equation}
\sup_{K(\varepsilon')\leq k\leq k_{\nu}(\varepsilon)-L(\varepsilon')}\mathcal{E}_{B_{2^{k+1}\lambda_{\nu}^{1}}(a_{\nu}^{1})\setminus B_{2^{k}\lambda_{\nu}^{1}}(a_{\nu}^{1})}[\phi_{\nu}](0)<\varepsilon',\label{eq:FirstNeck}
\end{equation}
for $\nu$ large enough. If this were to fail for some $\varepsilon'>0$,
we could find a sequence of scales, that we denote by $\lambda_{\nu}^{J'+2}$,
such that:
\begin{equation}
\mathcal{E}_{B_{2\lambda_{\nu}^{J'+2}}(a_{\nu}^{1})\setminus B_{\lambda_{\nu}^{J'+2}}(a_{\nu}^{1})}[\phi_{\nu}](0)>\varepsilon'\,\,\,\mathrm{and}\,\,\,\lambda_{\nu}^{J'+1}\ll\lambda_{\nu}^{J'+2}\ll\lambda_{\nu}^{1},\label{eq:BREAKNecCK}
\end{equation}
and that would give rise to new non-constant solitons at scale $\lambda_{\nu}^{J'+2}$
or above, in which case we have to redefine $\varepsilon$ as $\varepsilon'$.
Note that we can have only finitely many non-constant solitons forming
by the global energy bound (\ref{eq:EnergyBound}) since those cannot
have arbitrary small energy as this is not possible for harmonic 2-spheres,
and by (\ref{eq:BREAKNecCK}) they are asymptotically orthogonal in
$\dot{H}_{x}^{1}\times L_{x}^{2}$. Hence our procedure, applied to
every single soliton we have found so far, detects all of the solitons
in the claimed decomposition (\ref{eq:DecomWithNecks}) and we are
just left to characterize the regions in-between the domains of convergence
to solitons as neck regions, but this can be obtained directly from
(\ref{eq:FirstNeck}) as follows.

Upon changing notation, by the above remarks we can assume that (\ref{eq:FirstNeck})
holds. Now, we simply choose sequences $0<r_{\nu}^{1}\leq R_{\nu}^{1}$
tending to $0$ slowly enough so that for any $\varepsilon>0$ small
enough:
\[
\mathcal{E}_{B_{r_{\nu}^{1}}(a_{\nu}^{1})\setminus B_{2^{K(\varepsilon)}\lambda_{\nu}^{1}}(a_{\nu}^{1})}[\phi_{\nu}](0)\longrightarrow\mathcal{E}_{\mathbb{R}^{2}\setminus B_{2^{K(\varepsilon)}}(0)}[\omega_{1}](0)\,\,\,\mathrm{and}
\]
\[
\mathcal{E}_{B_{\lambda_{\nu}^{J'+1}}(a_{\nu}^{1})\setminus B_{R_{\nu}^{1}}(a_{\nu}^{1})}[\phi_{\nu}](0)\longrightarrow\mathcal{E}_{B_{1}\setminus\left\{ 0\right\} }[\omega_{J'+1}](0),
\]
then by (\ref{eq:FirstNeck}) there exits a sequence $\varepsilon_{\nu}^{1}=\varepsilon_{\nu}^{1}(r_{\nu}^{1},R_{\nu}^{1})\downarrow0$
such that:
\[
\sup_{r_{\nu}^{1}\leq r\leq\frac{1}{2}R_{\nu}^{1}}\sup_{t\in[-\frac{r}{2},\frac{r}{2}]}\mathcal{E}_{B_{2r}(a_{\nu}^{1})\setminus B_{r}(a_{\nu}^{1})}[\phi_{\nu}](t)<\varepsilon_{\nu}^{1}.
\]
If we know a priori that $r_{\nu}^{1}\sim R_{\nu}^{1}$, then we can
immediately absorb this part of the wave map $\phi_{\nu}$ into the
error term $o_{L_{t}^{\infty}(\dot{H}_{x}^{1}\times L_{x}^{2})}(1)$
in the decomposition (\ref{eq:DecomWithNecks}) and there is no loss
of energy between the considered solitons. Otherwise we should have
$r_{\nu}^{1}\ll R_{\nu}^{1}$, i.e. the annulus is conformally degenerating,
and this is precisely a neck in our terminology, as required. To prove
Theorem \ref{thm:Main} we must show that those terms can also be
absorbed into $o_{\dot{H}_{x}^{1}\times L_{x}^{2}}(1)$ upon picking
a suitable time slice, but that's the next and final step of the whole
argument. So far we have established Lemma \ref{lem:CoveringLemma}.\end{proof}
\begin{rem}
\label{RemShortTimeResolution} We note here that our techniques cannot
say anything more about the decomposition beyond the scales $\left\{ O(\lambda_{\mathrm{min},\nu})\right\} _{\nu\in\mathbb{N}}$
which is a central issue to address if one were to try understanding
the full soliton resolution conjecture. 

Let us also remark that there is also quite some freedom in fixing
the radii $R_{i,\nu}^{k}$ and $r_{i,\nu}^{k}$ defining the neck
domain, as for any positive integer $\ell\in\mathbb{N}$ which can
be arbitrarily large but fixed, we still have: 
\[
\sup_{2^{-\ell}r_{i,\nu}^{k}\leq r\leq2^{\ell}R_{i,\nu}^{k}}\sup_{t\in[-\frac{r}{2},\frac{r}{2}]}\mathcal{E}_{B_{2r}(x_{i,\nu}^{k})\setminus B_{r}(x_{i,\nu}^{k})}[\phi_{i,\nu}](t)\longrightarrow0,
\]
which follows directly from the characterization (\ref{eq:FirstNeck})
in the proof of Lemma \ref{lem:CoveringLemma} above. 
\end{rem}
Our aim now is to show energy collapsing for the necks $\mathcal{N}_{i,\nu}$,
that is a decay to zero for the $L_{x}^{2}$ norm of $\nabla_{t,x}\phi_{\nu}$
as $\nu\rightarrow+\infty$ on the degenerating annuli (\ref{eq:NeckDomains}).
We shall start by obtaining a decay in the weaker Besov $\dot{B}_{\infty}^{1,2}$
norm for $\mathcal{N}_{i,\nu}$, as consequence of the property (\ref{eq:NeckProperty}),
up to an error whose $\dot{H}_{x}^{1}$ norm is controlled by the
$L_{x}^{2}$ norm of $X\phi_{\nu}$ for some time-like vector field
$X$ that we will fix according to (\ref{eq:GoodSliceProps}) later.
This is the content of the following lemma.
\begin{lem}
\label{lem:Besov-control.}Consider a sequence of smooth wave maps
of bounded energy:
\begin{equation}
\phi_{\nu}:[-2^{N_{\nu}+O(1)},2^{N_{\nu}+O(1)}]\times\mathbb{R}^{2}\longrightarrow\mathbb{S}^{n-1},\,\,\,\left\Vert \nabla_{t,x}\phi_{\nu}\right\Vert _{L_{t}^{\infty}(L_{x}^{2})}^{2}\leq\mathcal{E},\label{eq:GlobalEnergyForBesov}
\end{equation}
obtained from Lemma \ref{lem:CoveringLemma} up to translating and
rescaling, where we are given two sequences of positive integers $n_{\nu},\, N_{\nu}\rightarrow+\infty$,
$n_{\nu}\ll N_{\nu}$, such that the neck property holds on $B_{2^{N_{\nu}}}\setminus B_{2^{n_{\nu}}}$:
\begin{equation}
\sup_{n_{\nu}\leq\ell\pm O(1)\leq N_{\nu}}\left\Vert \nabla_{t,x}\phi_{\nu}\right\Vert _{L_{t}^{\infty}(L_{x}^{2})\left([-2^{\ell-1},2^{\ell-1}]\times(B_{2^{\ell+1}}\setminus B_{2^{\ell}})\right)}\longrightarrow0.\label{eq:NeckDecayRate}
\end{equation}
Moreover, we assume the maps are asymptotically steady in the direction
of a constant time-like vector field $X$, standing for one of the
$X_{i}$'s from (\ref{eq:ConstantTimeLike}) which we can take to
be given by (\ref{eq:ExpressionForX-1}):
\begin{equation}
\left\Vert X\phi_{\nu}(0)\right\Vert _{L_{x}^{2}}\longrightarrow0,\label{eq:TimeLikeDecay}
\end{equation}
and the second order time-like derivatives satisfy: 
\begin{align*}
\Pi_{X,\nu}:= & \,\mathrm{sech}^{2}(\zeta)X^{2}\phi_{\nu}-\Omega_{\alpha}^{\nu}\partial^{\alpha}\phi_{\nu}\\
 & +\sum_{k\in\mathbb{Z}}P_{k}\left[\nabla_{x}\cdot(\Omega_{x,\beta}^{\nu}P_{>k+10}\phi_{\nu})+\Omega_{x,\beta}^{\nu}\cdot P_{\leq k+10}\nabla_{x}\phi_{\nu}\right],
\end{align*}
\begin{equation}
\sum_{k\in\mathbb{Z}}2^{-2k}\left\Vert P_{k}\Pi_{X,\nu}(0)\right\Vert _{L_{x}^{2}}^{2}\longrightarrow0,\label{eq:HigherOrderDecay}
\end{equation}
setting $\Omega_{\alpha}^{\nu}:=\phi_{\nu}\partial_{\alpha}\phi_{\nu}^{\dagger}-\partial_{\alpha}\phi_{\nu}\phi_{\nu}^{\dagger}$
and $\Omega_{x,\beta}^{\nu}:=(1-\beta^{2})\Omega_{x_{1}}\partial_{x_{1}}+\Omega_{x_{2}}\partial_{x_{1}}$.
Both assumptions are justified by (\ref{eq:GoodSliceProps}).

Then on the neck region, we can write for the map $\phi_{\nu}$: 
\[
\nabla_{t,x}\phi_{\nu}=\Upsilon_{\nu}\,\,\,\mathit{\mathit{on}}\,\,\,[-1,1]\times(B_{2^{N_{\nu}}}\setminus B_{2^{n_{\nu}}}),
\]
see (\ref{eq:PhysicalLPDecomposition}) in the proof, with $\Upsilon_{\nu}(t)\in C_{0}^{\infty}(B_{2^{N_{\nu}+1}}\setminus B_{2^{n_{\nu}-1}})$
for $t\in[-1,1]$ being of bounded energy $\left\Vert \Upsilon_{\nu}\right\Vert _{L_{t}^{\infty}(L_{x}^{2})[-1,1]}^{2}\lesssim\mathcal{E}$,
and satisfying the following weak decay estimate on $t=0$: 
\[
\sup_{k\in\mathbb{Z}}\left\Vert P_{k}\Upsilon_{\nu}(0)\right\Vert _{L_{x}^{2}}\longrightarrow0.
\]

\end{lem}
The strategy of our argument is roughly to replace, by using the decay
in the direction of the time-like vector field $X$, the sequence
of wave maps on neck domains under consideration with another one,
differing by an error of vanishing energy and converging locally to
a constant on the neck domain with more regularity than $\dot{H}_{x}^{1}\times L_{x}^{2}$
for $\phi_{\nu}$. However, because we need to obtain estimates that
are uniform in time, working on very short intervals, we should not
rely on the small energy regularity theory from Theorem \ref{thm:Small-data-waves}
and the direct use of Fourier restriction spaces, as in the proof
of the compactness result by Sterbenz and Tataru \citep{TataruSterbenzWaveReg}
(Proposition 5.1 there), but proceed directly via the wave maps equation
(\ref{eq:WMeq}) proving a weak $\dot{B}_{\infty}^{-1,2}$ decay estimate
for its quadratic structure in the gradient at high frequency (without
any null-structure involved, hence having target $\mathbb{S}^{n-1}$
is not specifically necessary for this part of the argument), and
then using Lemma \ref{lem:HigherOrderTimeLike} to control the second
order time-like derivatives (the latter though does involve the conservation
law (\ref{eq:ConservationLaw}) for wave maps into spheres).
\begin{proof}
As usual, having the required control in a time-like direction, it
is enough to consider the spatial gradient only. Now working on the
domain $[-1,1]\times(B_{2^{N_{\nu}}}\setminus B_{2^{n_{\nu}}})$,
we note it being arbitrarily rough in time as $n_{\nu},N_{\nu}\rightarrow+\infty$
degenerates. This is an additional difficulty, to be dealt with in
the present proof, in comparison to the analogous estimate for harmonic
maps, where $\varepsilon$-regularity is used on the domains $[-2^{\ell-1},2^{\ell-1}]\times(B_{2^{\ell+1}}\setminus B_{2^{\ell}})$
instead, see for the example the paper of Lin and Rivière \citep{LiR}
on page 188. 

Before taking the main line of the argument, let us start with some
preliminaries, fixing the decay rates for the assumptions of Lemma
\ref{lem:Besov-control.}, that is sequences $\iota_{\nu}\downarrow0$,
$\sigma_{\nu}\downarrow0$ and $\varepsilon_{\nu}\downarrow0$ for
which:
\begin{equation}
\sum_{k\in\mathbb{Z}}2^{-2k}\left\Vert P_{k}\Pi_{X,\nu}(0)\right\Vert _{L_{x}^{2}}^{2}\leq\iota_{\nu}^{2},\label{eq:HigherOrderDecay2}
\end{equation}
\begin{equation}
\left\Vert X\phi_{\nu}(0)\right\Vert _{L_{x}^{2}}\leq\sigma_{\nu},\label{eq:TimeLikeDecay2}
\end{equation}
\begin{equation}
\sup_{n_{\nu}\leq\ell\pm O(1)\leq N_{\nu}}\left\Vert \nabla_{t,x}\phi_{\nu}\right\Vert _{L_{t}^{\infty}(L_{x}^{2})\left([-2^{\ell-1},2^{\ell-1}]\times(B_{2^{\ell+1}}\setminus B_{2^{\ell}})\right)}\leq\varepsilon_{\nu},\label{eq:LocalEnergyDecay}
\end{equation}
corresponding to (\ref{eq:HigherOrderDecay}), (\ref{eq:TimeLikeDecay})
and (\ref{eq:NeckDecayRate}) respectively. Next, we consider, for
an arbitrary choice of integers $\ell_{\nu}$ between $n_{\nu}$ and
$N_{\nu}$, the sequence of wave maps:
\begin{equation}
\phi_{\nu,\ell_{\nu}}(\cdot):=\phi_{\nu}(2^{\ell_{\nu}}\cdot):[-2^{-4},2^{-4}]\times(B_{2^{3}}\setminus B_{2^{-3}})\longrightarrow\mathbb{S}^{n-1}.\label{eq:PhysicalDiadicMap}
\end{equation}
We build an extension $\psi_{\nu,\ell_{\nu}}$ of $\phi_{\nu,\ell_{\nu}}$,
as in Remark \ref{Rem:UniformityInTime}, by smoothly interpolating
on $(B_{2^{-2}}\setminus B_{2^{-3}})\cup(B_{2^{2}}\setminus B_{2^{2}-1})$
between $\phi_{\nu,\ell_{\nu}}[0]$ and $(c_{\ell_{\nu}},0)\in T(\mathbb{S}^{n-1})$,
for some suitably chosen sequence of constants $c_{\ell_{\nu}}=c_{\ell_{\nu}}(\phi_{\nu,\ell_{\nu}})$,
solving the wave maps equation for $\psi_{\nu,\ell_{\nu}}$ with initial
data of $\psi_{\nu,\ell_{\nu}}[0]$, such that scaling back and setting
$\psi_{\nu}^{\ell_{\nu}}(\cdot):=\psi_{\nu,\ell_{\nu}}(2^{-\ell_{\nu}}\cdot)$,
we have (denoting by $1_{\ell_{\nu}}$ the characteristic function
of $B_{2^{\ell_{\nu}+1}}\setminus B_{2^{\ell_{\nu}-1}}$ over the
time interval $[-2^{\ell_{\nu}-3},2^{\ell_{\nu}-3}]$):
\begin{equation}
\left\Vert \nabla_{t,x}\psi_{\nu}^{\ell_{\nu}}\right\Vert _{L_{t}^{\infty}(L_{x}^{2})}\lesssim\varepsilon_{\nu}\,\,\,\mathrm{and}\,\,\,1_{\ell_{\nu}}\phi_{\nu}=1_{\ell_{\nu}}\psi_{\nu}^{\ell_{\nu}},\label{eq:ExtensionEnergyControl}
\end{equation}
by (\ref{eq:LocalEnergyDecay}) and the finite speed of propagation
property respectively. 

From there, we construct a partition of unity over $[-1,1]\times(B_{2^{N_{\nu}}}\setminus B_{2^{n_{\nu}}})$
paralleling the Littlewood-Paley decomposition in frequency space.
For the spatial directions, we recall the non-negative radial bump
functions $m_{0}$ and $m_{\leq0}$ used in the definition of the
LP-projections $P_{0}$ and $P_{\leq0}$, but which this time, we
will use on the physical space setting: 
\[
\bar{m}_{0}(t,x):=m_{0}(|x|),\,\,\,\bar{m}_{\ell}(t,x):=\bar{m}_{0}(2^{-\ell}t,2^{-\ell}x),
\]
\[
\bar{m}_{\leq0}(t,x):=m_{\leq0}(|x|),\,\,\,\bar{m}_{\leq\ell}(t,x):=\bar{m}_{\leq0}(2^{-\ell}t,2^{-\ell}x).
\]
We get then the following ``physical LP-decomposition'':
\begin{equation}
\Upsilon_{\nu}:=(\bar{m}_{\leq N_{\nu}}-\bar{m}_{\leq n_{\nu}-1})\eta\nabla_{x}\phi_{\nu}=\sum_{\ell_{\nu}=n_{\nu}}^{N_{\nu}}\eta\bar{m}_{\ell_{\nu}}\nabla_{x}\phi_{\nu}.\label{eq:PhysicalLPDecomposition}
\end{equation}
where $\eta(t)$ stands for the rough cut-off to the time interval
$[-1,1]$, and of course it is immediate that $\left\Vert \Upsilon_{\nu}\right\Vert _{L_{t}^{\infty}(L_{x}^{2})[-1,1]}^{2}\lesssim\mathcal{E}$.
Moreover we note that, recalling the extensions (\ref{eq:ExtensionEnergyControl}),
we have $\eta\bar{m}_{\ell_{\nu}}\phi_{\nu}=\eta\bar{m}_{\ell_{\nu}}\psi_{\nu}^{\ell_{\nu}}$. 

Writing $\phi_{\nu}^{c}:=\phi_{\nu}-c_{\ell_{\nu}}$, for an arbitrary
sequence of maps corresponding to (\ref{eq:PhysicalDiadicMap}), and
similarly for $\phi_{\nu,\ell_{\nu}}^{c}$, together with the extensions
$\psi_{\nu}^{\ell_{\nu},c}$ and $\psi_{\nu,\ell_{\nu}}^{c}$ from
(\ref{eq:ExtensionEnergyControl}) which become compactly supported
by construction, we consider the commutator (denoting the cut-off
functions by $\chi_{\ell_{\nu}}:=\eta\bar{m}_{\ell_{\nu}}$): 
\begin{equation}
\chi_{\ell_{\nu}}\nabla_{x}\phi_{\nu}=\nabla_{x}(\chi_{\ell_{\nu}}\phi_{\nu}^{c})-(\nabla_{x}\chi_{\ell_{\nu}})\phi_{\nu}^{c},\label{eq:Commutator1}
\end{equation}
and start by treating the second term, for which we claim:
\begin{equation}
\left\Vert P_{k}[(\nabla_{x}\chi_{\ell_{\nu}})\phi_{\nu}^{c}]\right\Vert _{L_{t}^{\infty}(L_{x}^{2})}\lesssim2^{-\left|k+\ell_{\nu}\right|}\varepsilon_{\nu},\label{eq:Commutator1Control}
\end{equation}
for any $k\in\mathbb{Z}$. To see this, we rescale by $2^{\ell_{\nu}}$.
For high frequency scales $2^{k}\gtrsim1$, we can use the extra regularity,
the spatial derivative falling on the cut-off instead of the map,
available from: 
\begin{align*}
\left\Vert \nabla_{x}[(\nabla_{x}\bar{m}_{0})\phi_{\nu,\ell_{\nu}}^{c}]\right\Vert _{L_{t}^{\infty}(L_{x}^{2})}\lesssim & \left\Vert (\nabla_{x}^{2}\bar{m}_{0})\psi_{\nu,\ell_{\nu}}^{c}\right\Vert _{L_{t}^{\infty}(L_{x}^{2})}+\left\Vert (\nabla_{x}\bar{m}_{0})\nabla_{x}\phi_{\nu,\ell_{\nu}}\right\Vert _{L_{t}^{\infty}(L_{x}^{2})},
\end{align*}
introducing the extensions $\psi_{\nu,\ell_{\nu}}^{c}$, so that applying
Poincaré's inequality in $L_{x}^{2}$ for the first term, given the
spatial localization of $\psi_{\nu,\ell_{\nu}}^{c}$ at any given
time slice in the support of $\eta_{\ell_{\nu}}(\cdot):=\eta(2^{\ell_{\nu}}\cdot)$,
we get by the finite band property (\ref{eq:FiniteBandEqual}) and
the bound (\ref{eq:ExtensionEnergyControl}):
\[
\left\Vert \eta_{\ell_{\nu}}P_{k}[(\nabla_{x}\bar{m}_{0})\psi_{\nu,\ell_{\nu}}^{c}]\right\Vert _{L_{t}^{\infty}(L_{x}^{2})}\lesssim2^{-k}\varepsilon_{\nu},
\]
as desired. For low frequency scales $2^{k}\lesssim1$, by Cauchy-Schwarz
and Poincaré's inequalities, we have:
\[
\left\Vert \eta_{\ell_{\nu}}(\nabla_{x}\bar{m}_{0})\psi_{\nu,\ell_{\nu}}^{c}\right\Vert _{L_{t}^{\infty}(L_{x}^{1})}\lesssim\left\Vert \eta_{\ell_{\nu}}\nabla_{x}\psi_{\nu,\ell_{\nu}}\right\Vert _{L_{t}^{\infty}(L_{x}^{2})},
\]
dropping $\nabla_{x}\bar{m}_{0}$, and so using Bernstein's inequality
(\ref{eq:Bernstein}) we obtain here an exponential gain as well:
\[
\left\Vert \eta_{\ell_{\nu}}P_{k}[(\nabla_{x}\bar{m}_{0})\psi_{\nu,\ell_{\nu}}^{c}]\right\Vert _{L_{t}^{\infty}(L_{x}^{2})}\lesssim2^{k}\varepsilon_{\nu},
\]
by the energy bound (\ref{eq:ExtensionEnergyControl}). Hence, claim
(\ref{eq:Commutator1Control}) follows. 

We remark that, by the same argument, we get also control for the
low frequencies of the first term $\nabla_{x}(\chi_{\ell_{\nu}}\phi_{\nu}^{c})$
in the commutator:
\begin{equation}
\left\Vert P_{k}\nabla_{x}(\chi_{\ell_{\nu}}\phi_{\nu}^{c})\right\Vert _{L_{t}^{\infty}(L_{x}^{2})}\lesssim2^{k+\ell_{\nu}}\varepsilon_{\nu},\,\,\, k\leq-\ell_{\nu}+O(1),\label{eq:LowFrequencies}
\end{equation}
and so it remains to treat now the main terms, that is the LHS above
when $\ell_{\nu}\geq-k$, for which we should rely on the wave maps
equation, the time-like control assumption (\ref{eq:TimeLikeDecay2}),
as well as the favorable decay (\ref{eq:HigherOrderDecay2}) we already
have. 

Recalling the expression for the operator (\ref{eq:SpatialBetaLaplacian}),
we compute then:
\begin{align}
\Delta_{x,\beta}(\chi_{\ell_{\nu}}\phi_{\nu}^{c})= & (\Delta_{x,\beta}\chi_{\ell_{\nu}})\phi_{\nu}^{c}+2(1-\beta^{2})(\partial_{x_{1}}\chi_{\ell_{\nu}})(\partial_{x_{1}}\phi_{\nu})+2(\partial_{x_{2}}\chi_{\ell_{\nu}})(\partial_{x_{2}}\phi_{\nu})\label{eq:TheEquation}\\
 & -2\chi_{\ell_{\nu}}\mathrm{sech}^{2}(\zeta)\mathrm{sinh}(\zeta)\partial_{x_{1}}X\phi_{\nu}\nonumber \\
 & +\chi_{\ell_{\nu}}(\mathrm{sech}^{2}(\zeta)X^{2}\phi_{\nu}-\Omega_{\alpha}^{\nu}\partial^{\alpha}\phi_{\nu}).\nonumber 
\end{align}

Let us treat first the smooth terms on the first line of (\ref{eq:TheEquation}),
of which there are two types, $(\nabla_{x}^{2}\chi_{\ell_{\nu}})\psi_{\nu}^{\ell_{\nu},c}$
and $\nabla_{x}\chi_{\ell_{\nu}}\nabla_{t,x}\psi_{\nu}^{\ell_{\nu}}$,
the cut-off differentiated in a spatial direction, claiming for both
the control:
\begin{equation}
\left\Vert \frac{\nabla_{x}}{\Delta_{x,\beta}}P_{k}[(\nabla_{x}^{2}\chi_{\ell_{\nu}})\psi_{\nu}^{\ell_{\nu},c}+\nabla_{x}\chi_{\ell_{\nu}}\nabla_{t,x}\psi_{\nu}^{\ell_{\nu}}]\right\Vert _{L_{t}^{\infty}(L_{x}^{2})}\lesssim2^{-(k+\ell_{\nu})}\varepsilon_{\nu},\,\,\, k\geq-\ell_{\nu}.\label{eq:Error1}
\end{equation}
To show this, relying on Plancherel in $L_{x}^{2}$, we discard the
Fourier multiplier $2^{k}\nabla_{x}\Delta_{x,\beta}^{-1}\widetilde{P}_{k}$
(where $\widetilde{P}_{k}=P_{k-1\leq\cdot\leq k+1}$), having symbol
bounded uniformly in $k\in\mathbb{Z}$. Rescaling by $2^{\ell_{\nu}}$
we are brought to estimate for $k\geq O(1)$:
\[
2^{-k}\left\Vert \eta_{\ell_{\nu}}[(\nabla_{x}^{2}\bar{m}_{0})\psi_{\nu,\ell_{\nu}}^{c}+\nabla_{x}\bar{m}_{0}\nabla_{t,x}\psi_{\nu,\ell_{\nu}}]\right\Vert _{L_{t}^{\infty}(L_{x}^{2})},
\]
where the second term is directly seen to have the desired control
by (\ref{eq:LocalEnergyDecay}), whereas for the first one, given
the spatial support of the extension $\psi_{\nu,\ell_{\nu}}^{c}$,
we apply Poincaré's inequality in $L_{x}^{2}$ as before, which allows
us to conclude by (\ref{eq:ExtensionEnergyControl}).

The second line of (\ref{eq:TheEquation}) is an error term controlled
thanks to the time-like decay (\ref{eq:TimeLikeDecay2}) we have.
We first write:
\begin{align*}
\chi_{\ell_{\nu}}\nabla_{x}X\phi_{\nu} & =\nabla_{x}(\chi_{\ell_{\nu}}X\phi_{\nu})-(\nabla_{x}\chi_{\ell_{\nu}})X\psi_{\nu}^{\ell_{\nu}},
\end{align*}
and note that the second term here was already treated in (\ref{eq:Error1}),
and so we just need to show: 
\begin{equation}
\left\Vert \frac{\nabla_{x}^{2}}{\Delta_{x,\beta}}P_{k}\sum_{\ell_{\nu}=\max(-k,n_{\nu})}^{N_{\nu}}(\chi_{\ell_{\nu}}X\phi_{\nu})(0)\right\Vert _{L_{x}^{2}}\lesssim\sigma_{\nu},\label{eq:Error2}
\end{equation}
but this follows at once by Plancherel in $L_{x}^{2}$, as the Fourier
multiplier $\nabla_{x}^{2}\Delta_{x,\beta}^{-1}P_{k}$ has a bounded
symbol, dropping the cut-offs and relying on (\ref{eq:TimeLikeDecay2}).

Finally, we shall consider the delicate second order time-like derivatives
and the non-linear terms on the third line of (\ref{eq:TheEquation}).
As was already required for (\ref{eq:Error2}), we restrict ourselves
from now on to work exclusively over the time slice $t=0$. And to
lighten the notation, we shall not mention this explicitly anymore.

Thanks to the assumption (\ref{eq:HigherOrderDecay2}), we have already
partial control on them through $\Pi_{X,\nu}$, which however we need
to localize to the neck region $B_{2^{N_{\nu}+1}}\setminus B_{2^{\max(-k,n_{\nu})-1}}$.
In doing so, we first note that since $\bar{m}_{\leq0}$ was initially
fixed spatially Schwartz, we have:
\[
\left\Vert \nabla_{x}\widetilde{m}_{k,N_{\nu}}\right\Vert _{L_{x}^{2}}\lesssim1,\,\,\,\mathrm{where}\,\,\,\widetilde{m}_{k,N_{\nu}}:=\bar{m}_{\leq N_{\nu}}-\bar{m}_{\leq\max(-k,n_{\nu})-1},
\]
given that the above norm is scale invariant. Hence applying the Littlewood-Paley
trichotomy to $\widetilde{m}_{k,N_{\nu}}\Pi_{X,\nu}$, we get:
\begin{align*}
P_{k}(\widetilde{m}_{k,N_{\nu}}\Pi_{X,\nu})= & \, P_{k}[(P_{\leq k-7}\widetilde{m}_{k,N_{\nu}})(P_{k-3\leq\cdot\leq k+3}\Pi_{X,\nu})\\
 & +(P_{k-3\leq\cdot\leq k+3}\widetilde{m}_{k,N_{\nu}})(P_{\leq k-7}\Pi_{X,\nu})\\
 & +\sum_{k_{1},k_{2}\geq k-6:\left|k_{1}-k_{2}\right|\leq O(1)}(P_{k_{1}}\widetilde{m}_{k,N_{\nu}})(P_{k_{2}}\Pi_{X,\nu})].
\end{align*}

From there, using (\ref{eq:HigherOrderDecay2}), we estimate the low-high
interactions by:
\[
2^{-k}\left\Vert (P_{\leq k-7}\widetilde{m}_{k,N_{\nu}})(P_{k-3\leq\cdot\leq k+3}\Pi_{X,\nu})\right\Vert _{L_{x}^{2}}\lesssim\left\Vert \widetilde{m}_{k,N_{\nu}}\right\Vert _{L_{x}^{\infty}}\iota_{\nu},
\]
the high-low ones by:
\[
2^{-k}\left\Vert (P_{k-3\leq\cdot\leq k+3}\widetilde{m}_{k,N_{\nu}})(\sum_{k_{1}\leq k-7}P_{k_{1}}\Pi_{X,\nu})\right\Vert _{L_{x}^{2}}\lesssim\left\Vert \widetilde{m}_{k,N_{\nu}}\right\Vert _{L_{x}^{\infty}}\sum_{k_{1}\leq k-7}2^{-(k-k_{1})}\iota_{\nu},
\]
whereas for the high-high cascade we have:
\[
\sum_{k_{1},k_{2}\geq k-6:\left|k_{1}-k_{2}\right|\leq O(1)}2^{-k}\left\Vert \eta(P_{k_{1}}\widetilde{m}_{k,N_{\nu}})(P_{k_{2}}\Pi_{X,\nu})\right\Vert _{L_{x}^{2}}\lesssim(\sum_{k_{1}}2^{2k_{1}}\left\Vert P_{k_{1}}\widetilde{m}_{k,N_{\nu}}\right\Vert _{L_{x}^{2}}^{2})^{\frac{1}{2}}\iota_{\nu},
\]
where we have used Bernstein's inequality (\ref{eq:Bernstein}) passing
to $L_{x}^{1}$, and then Cauchy-Schwarz with the fact that $k_{1}=k_{2}+O(1)$.

Putting those estimates together we get the required control for $\widetilde{m}_{k,N_{\nu}}\Pi_{X,\nu}$:
\begin{equation}
\left\Vert \frac{\nabla_{x}}{\Delta_{x,\beta}}P_{k}\sum_{\ell_{\nu}=\max(-k,n_{\nu})}^{N_{\nu}}\chi_{\ell_{\nu}}\Pi_{X,\nu}\right\Vert _{L_{x}^{2}}\lesssim\iota_{\nu},\label{eq:Error3}
\end{equation}
by discarding the multiplier $2^{k}\nabla_{x}\Delta_{x,\beta}^{-1}\widetilde{P}_{k}$
and relying on the bounds for the cut-offs $\widetilde{m}_{k,N_{\nu}}$
discussed above.

We treat now the non-linear bulk left from Lemma \ref{lem:HigherOrderTimeLike},
decomposing it into:
\begin{align*}
B_{1}^{\nu}:= & \sum_{k\in\mathbb{Z}}P_{k}\nabla_{x}\cdot(\Omega_{x,\beta}^{\nu}\phi_{\nu}^{>k+10}),\\
B_{2}^{\nu}:= & \sum_{k\in\mathbb{Z}}P_{k}(\Omega_{x,\beta}^{\nu}\cdot\nabla_{x}\phi_{\nu}^{\leq k+10}),
\end{align*}
introducing the convenient notation $\phi_{\nu}^{k}:=P_{k}\phi_{\nu}$
(also later $\phi_{\nu,\ell_{\nu}}^{k}:=P_{k}\phi_{\nu,\ell_{\nu}}$
for the rescaled maps), etc. We want to treat this term perturbatively,
as in elliptic regularity theory, and so we proceed claiming first
the following $\dot{B}_{\infty}^{-1,2}$ estimate:

\begin{align}
 & \left\Vert \frac{\nabla_{x}}{\Delta_{x,\beta}}P_{k}\sum_{\ell_{\nu}=\max(-k,n_{\nu})}^{N_{\nu}}\chi_{\ell_{\nu}}B_{i}^{\nu}\right\Vert _{L_{x}^{2}}^{2}\label{eq:NonLinearDoublePrim}\\
 & \lesssim\sum_{\jmath\geq0}2^{-\jmath}\sum_{\ell}\left\Vert \chi_{\ell_{\nu}}B_{i}^{\nu}\right\Vert _{L_{t}^{\infty}(L_{x}^{1})}\left\Vert \chi_{\ell_{\nu}+\jmath}B_{i}^{\nu}\right\Vert _{L_{t}^{\infty}(L_{x}^{1})},\nonumber 
\end{align}
where the sums are such that both $\ell_{\nu}$ and $\ell_{\nu}+\jmath$
range between $\max(-k,n_{\nu})$ and $N_{\nu}$. 

Discarding the Fourier multiplier $2^{k}\nabla_{x}\Delta_{x,\beta}^{-1}\widetilde{P}_{k}$
via Plancherel in $L_{x}^{2}$, we note the Littlewood-Paley projection
$P_{k}$ in front of the sum in (\ref{eq:NonLinearDoublePrim}) is
crucial to handle the remaining factor $2^{-k}$. But frequency localization
induces spreading for the physical support by the uncertainty principle.
And so, we are not allowed to use a square-summing trick relying on
the finitely overlapping supports of $\chi_{\ell_{\nu}}B_{i}^{\nu}$.
On the other hand, this leakage is very much controllable given the
fact that $k\geq-\ell_{\nu}+O(1)$, which corresponds to high frequency
here. 

More precisely, let us bound the LHS of (\ref{eq:NonLinearDoublePrim})
via:
\[
2^{-2k}\sum_{\mu_{\nu}\geq\ell_{\nu}}\left|\int_{\mathbb{R}^{2}}\left[P_{k}(\chi_{\ell_{\nu}}B_{i}^{\nu})\right]\left[P_{k}(\chi_{\mu_{\nu}}B_{i}^{\nu})\right]dx\right|,
\]
with both $\ell_{\nu}$ and $\mu_{\nu}$ ranging between $\max(-k,n_{\nu})$
and $N_{\nu}$. By the self-adjointness of $P_{k}$, the summand above
can be estimated by:
\begin{align*}
 & \left\Vert \left[P_{k}^{2}(\chi_{\ell_{\nu}}B_{i}^{\nu})\right]\chi_{\mu_{\nu}}B_{i}^{\nu}\right\Vert _{L_{x}^{1}}\\
 & \leq\left\Vert P_{k}^{2}(\chi_{\ell_{\nu}}B_{i}^{\nu})\right\Vert _{L_{x}^{\infty}\left(\{\left|x\right|\sim2^{\mu_{\nu}}\}\right)}\left\Vert \chi_{\mu_{\nu}}B_{i}^{\nu}\right\Vert _{L_{x}^{1}}.
\end{align*}
Now, looking at the convolution kernel for $P_{k}^{2}$, analogue
to (\ref{eq:ConvolutionLPprojRepresentation}), we can estimate the
first factor on the RHS above by:
\[
\left\Vert P_{k}^{2}(\chi_{\ell_{\nu}}B_{i}^{\nu})\right\Vert _{L_{x}^{\infty}\left(\{\left|x\right|\sim2^{\mu_{\nu}}\}\right)}\lesssim2^{2k}2^{-(\mu_{\nu}-\ell_{\nu})}\left\Vert \chi_{\ell_{\nu}}B_{i}^{\nu}\right\Vert _{L_{x}^{1}},
\]
for $\mu_{\nu}\geq\ell_{\nu}\geq-k$, a refined version of Bernstein's
inequality (\ref{eq:Bernstein}). Hence, this leads us to estimate
the LHS of (\ref{eq:NonLinearDoublePrim}) by:
\[
\sum_{\mu_{\nu}\geq\ell_{\nu}}2^{-(\mu_{\nu}-\ell_{\nu})}\left\Vert \chi_{\ell_{\nu}}B_{i}^{\nu}\right\Vert _{L_{x}^{1}}\left\Vert \chi_{\mu_{\nu}}B_{i}^{\nu}\right\Vert _{L_{x}^{1}},
\]
as required. 

Given (\ref{eq:NonLinearDoublePrim}), we remark that summing one
of the factors we get a universal bound. This follows from the global
energy control (\ref{eq:GlobalEnergyForBesov}) since, by the finitely
overlapping supports of $\chi_{\ell_{\nu}}B_{i}^{\nu}$:
\[
\sum_{\ell_{\nu}}\left\Vert \chi_{\ell_{\nu}}B_{i}^{\nu}\right\Vert _{L_{x}^{1}}\lesssim\left\Vert B_{i}^{\nu}\right\Vert _{L_{x}^{1}},
\]
and in fact we have the stronger control:
\begin{equation}
\sum_{k\in\mathbb{Z}}\left\Vert P_{k}B_{1}^{\nu}\right\Vert _{L_{x}^{1}}+\left\Vert (\sum_{k\in\mathbb{Z}}|P_{k}B_{2}^{\nu}|^{2})^{\frac{1}{2}}\right\Vert _{L_{x}^{1}}\lesssim\mathcal{\mathcal{E}},\label{eq:L1controlDual}
\end{equation}
where for the former we have:
\[
\left\Vert P_{k}B_{1}^{\nu}\right\Vert _{L_{x}^{1}}\lesssim\sum_{k_{1},k_{2}\geq k+5:\left|k_{1}-k_{2}\right|\leq O(1)}2^{-(k_{2}-k)}\left\Vert P_{k_{1}}\Omega_{x,\beta}^{\nu}\right\Vert _{L_{x}^{2}}\left\Vert \nabla_{x}\phi_{\nu}^{k_{2}}\right\Vert _{L_{x}^{2}},
\]
applying initially the finite band property (\ref{eq:FiniteBandEqual}),
and then once again for $\phi_{\nu}^{k_{2}}$, and this can be summed
over $k\in\mathbb{Z}$ using discrete Cauchy-Schwarz in $k_{1}=k_{2}+O(1)$.
Whereas for the latter, we note that by the Littlewood-Paley trichotomy:
\begin{align*}
P_{k}(\Omega_{x,\beta}^{\nu}\cdot\nabla_{x}\phi_{\nu}^{\leq k+10})= & \, P_{k}[P_{\leq k-7}(\Omega_{x,\beta}^{\nu})\cdot\nabla_{x}\phi_{\nu}^{k-3\leq\cdot\leq k+3}\\
 & +P_{k-3\leq\cdot\leq k+3}(\Omega_{x,\beta}^{\nu})\cdot\nabla_{x}\phi_{\nu}^{\leq k-7}\\
 & +\sum_{k_{1},k_{2}\sim k}P_{k_{1}}(\Omega_{x,\beta}^{\nu})\cdot\nabla_{x}\phi_{\nu}^{k_{2}}],
\end{align*}
and so the first two terms correspond to paraproducts, already localized
to $\left|\xi\right|\sim2^{k}$, and therefore their sum in $k\in\mathbb{Z}$
lies in the homogeneous Hardy space $\dot{F}_{2}^{0,1}$ with bound
$O(\mathcal{\mathcal{E}})$, and for the last term the stronger estimate
in $\dot{B}_{1}^{0,1}$ with bound $O(\mathcal{\mathcal{E}})$ as
for $B_{1}^{\nu}$ holds, since the sum under $P_{k}$ is finite and
we can apply the discrete Cauchy-Schwarz inequality.

Hence, rescaling by $2^{\ell_{\nu}}$ and setting $B_{i}^{\nu,\ell_{\nu}}(\cdot)=2^{2\ell_{\nu}}B_{i}^{\nu}(2^{\ell_{\nu}}\cdot)$,
to obtain decay for (\ref{eq:NonLinearDoublePrim}) it suffices to
prove:
\begin{equation}
\sup_{n_{\nu}\leq\ell_{\nu}\leq N_{\nu}}\left\Vert \bar{m}_{0}B_{i}^{\nu,\ell_{\nu}}\right\Vert _{L_{x}^{1}}\leq o(\mathcal{\mathcal{E}}).\label{eq:LastTerms}
\end{equation}
This is direct manifestation of the perturbative nature of quadratic
non-linearities on neck regions, thanks to local energy decay (\ref{eq:LocalEnergyDecay}).
In our case, the argument is however slightly more involved because
our product structure is non-local. This represents however a minor
technicality only, and we shall treat this analogously to the previous
instances of physical support leakage.

Let us introduce two auxiliary parameters. Setting $\Omega_{x,\beta}^{\nu,\ell_{\nu}}(\cdot):=2^{\ell_{\nu}}\Omega_{x,\beta}^{\nu}(2^{\ell_{\nu}}\cdot)$,
by the local energy estimate (\ref{eq:LocalEnergyDecay}), we can
find sequences $\kappa_{\nu}\rightarrow+\infty$ and $\tilde{\varepsilon}_{\nu}\downarrow0$
such that:
\[
\left\Vert \bar{m}_{-10\leq\cdot\leq\kappa_{\nu}}\Omega_{x,\beta}^{\nu,\ell_{\nu}}\right\Vert _{L_{x}^{2}}\leq\tilde{\varepsilon}_{\nu},
\]
where we use the convention $\bar{m}_{k_{1}\leq\cdot\leq k_{2}}:=\bar{m}_{\leq k_{2}}-\bar{m}_{\leq k_{1}-1}$,
and similarly for $\bar{m}_{\geq k_{1}}:=1-\bar{m}_{\leq k_{1}-1}$.
Let us first treat the annulus determined so, and then the outer and
inner regions separately. 

For the annulus we can discard the cut-off $\bar{m}_{0}$. Regarding
$B_{1}^{\nu,\ell_{\nu}}$, we have:
\begin{align*}
 & \sum_{k\in\mathbb{Z}}\left\Vert P_{k}\nabla_{x}\cdot(\bar{m}_{-10\leq\cdot\leq\kappa_{\nu}}\Omega_{x,\beta}^{\nu,\ell_{\nu}}\phi_{\nu,\ell_{\nu}}^{>k+10})\right\Vert _{L_{x}^{1}}\\
 & \lesssim\sum_{k\in\mathbb{Z}}\,\,\,\sum_{k_{1},k_{2}\geq k+5:\left|k_{1}-k_{2}\right|\leq O(1)}2^{-(k_{1}-k)}\left\Vert P_{k_{1}}(\bar{m}_{-10\leq\cdot\leq\kappa_{\nu}}\Omega_{x,\beta}^{\nu,\ell_{\nu}})\right\Vert _{L_{x}^{2}}\left\Vert \nabla_{x}\phi_{\nu}^{k_{2}}\right\Vert _{L_{x}^{2}},
\end{align*}
where we have used the finite band property (\ref{eq:FiniteBandEqual})
as usual, and we control this by $O(\tilde{\varepsilon}_{\nu}\mathcal{\mathcal{E}}^{\frac{1}{2}})$
relying on the discrete Cauchy-Schwarz and $k_{1}=k_{2}+O(1)$, which
is acceptable for (\ref{eq:LastTerms}). For $B_{2}^{\nu,\ell_{\nu}}$,
we use Littlewood-Paley trichotomy as previously to get:
\begin{align*}
 & \left\Vert \sum_{k\in\mathbb{Z}}P_{k}(\bar{m}_{-10\leq\cdot\leq\kappa_{\nu}}\Omega_{x,\beta}^{\nu,\ell_{\nu}}\cdot\nabla_{x}\phi_{\nu,\ell_{\nu}}^{\leq k+10})\right\Vert _{L_{x}^{1}}\\
 & \lesssim\left\Vert \sup_{k_{1}\in\mathbb{Z}}|P_{\leq k_{1}-7}(\bar{m}_{-10\leq\cdot\leq\kappa_{\nu}}\Omega_{x,\beta}^{\nu})|\right\Vert _{L_{x}^{2}}\cdot\left\Vert (\sum_{k_{2}\in\mathbb{Z}}|\nabla_{x}\phi_{\nu,\ell_{\nu}}^{k_{2}-3\leq\cdot\leq k_{2}+3}|^{2})^{\frac{1}{2}}\right\Vert _{L_{x}^{2}}\\
 & +\left\Vert (\sum_{k_{1}\in\mathbb{Z}}|P_{k_{1}-3\leq\cdot\leq k_{1}+3}(\bar{m}_{-10\leq\cdot\leq\kappa_{\nu}}\Omega_{x,\beta}^{\nu})|^{2})^{\frac{1}{2}}\right\Vert _{L_{x}^{2}}\cdot\left\Vert \sup_{k_{2}\in\mathbb{Z}}|\nabla_{x}\phi_{\nu,\ell_{\nu}}^{\leq k_{2}-7}|\right\Vert _{L_{x}^{2}}\\
 & +\sum_{k\in\mathbb{Z}}\sum_{k_{1},k_{2}\sim k}\left\Vert P_{k_{1}}(\bar{m}_{-10\leq\cdot\leq\kappa_{\nu}}\Omega_{x,\beta}^{\nu})\right\Vert _{L_{x}^{2}}\cdot\left\Vert \nabla_{x}\phi_{\nu,\ell_{\nu}}^{k_{2}}\right\Vert _{L_{x}^{2}},
\end{align*}
and relying on the Littlewood-Paley square function estimate for the
first two terms, and simply the discrete Cauchy-Schwarz for the last,
we can bound the above by $O(\tilde{\varepsilon}_{\nu}\mathcal{\mathcal{E}}^{\frac{1}{2}})$
again. Therefore this is permissible contribution to (\ref{eq:LastTerms}). 

Now we treat the error terms. First, let us consider the outer region
defined by the cut-off $\bar{m}_{>\kappa_{\nu}}$. Writing: 
\begin{equation}
\left\Vert \bar{m}_{0}\sum_{k\in\mathbb{Z}}P_{k}B\right\Vert _{L_{x}^{1}}\lesssim\left\Vert \bar{m}_{0}\right\Vert _{L_{x}^{1}}\sum_{k\in\mathbb{Z}}\left\Vert P_{k}B\right\Vert _{L_{x}^{\infty}\left(\{2^{-1}\leq\left|x\right|\leq2\}\right)},\label{eq:AuxiliaryCrap}
\end{equation}
we proceed, first for:
\[
B:=\sum_{k\in\mathbb{Z}}P_{k}\nabla_{x}\cdot(\bar{m}_{>\kappa_{\nu}}\Omega_{x,\beta}^{\nu,\ell_{\nu}}\phi_{\nu,\ell_{\nu}}^{>k+10}),
\]
by considering the convolution kernel for the Fourier multiplier $\nabla_{x}P_{k}P_{k'}$,
with $k=k'+O(1)$, which gives:
\[
\left\Vert P_{k}B\right\Vert _{L_{x}^{\infty}\left(\{2^{-1}\leq\left|x\right|\leq2\}\right)}\lesssim_{N}\frac{2^{3k}}{(1+2^{k}2^{\kappa_{\nu}})^{N}}\left\Vert \bar{m}_{>\kappa_{\nu}}\Omega_{x,\beta}^{\nu,\ell_{\nu}}\phi_{\nu,\ell_{\nu}}^{>k'+10}\right\Vert _{L_{x}^{1}},
\]
for any positive integer $N\in\mathbb{N}$, bearing in mind the physical
support of $\bar{m}_{>\kappa_{\nu}}\Omega_{x,\beta}^{\nu,\ell_{\nu}}\phi_{\nu,\ell_{\nu}}^{>k+10}$.
Using this estimate, for high frequency scales, we choose $N=3$,
getting the following bound for the sum in $k\geq0$ from (\ref{eq:AuxiliaryCrap})
:
\[
2^{-3\kappa_{\nu}}\sum_{k\geq0}2^{-k}\left\Vert \bar{m}_{>\kappa_{\nu}}\Omega_{x,\beta}^{\nu,\ell_{\nu}}\right\Vert _{L_{x}^{2}}\sum_{k_{1}>k+10}2^{-(k_{1}-k)}\left\Vert \nabla_{x}\phi_{\nu,\ell_{\nu}}^{k_{1}}\right\Vert _{L_{x}^{2}},
\]
by the finite band property (\ref{eq:FiniteBandEqual}) for $\phi_{\nu,\ell_{\nu}}$.
This is immediately seen to be $o(\mathcal{E})$ as $\kappa_{\nu}\rightarrow+\infty$,
hence this contribution is acceptable. For the low frequency scales,
if we set $N=1$ above, we have for the sum over $k<0$ in (\ref{eq:AuxiliaryCrap}):
\[
2^{-\kappa_{\nu}}\sum_{k<0}2^{k}\left\Vert \bar{m}_{>\kappa_{\nu}}\Omega_{x,\beta}^{\nu,\ell_{\nu}}\right\Vert _{L_{x}^{2}}\sum_{k_{1}>k+10}2^{-(k_{1}-k)}\left\Vert \nabla_{x}\phi_{\nu,\ell_{\nu}}^{k_{1}}\right\Vert _{L_{x}^{2}}\lesssim o(\mathcal{E}),
\]
as desired, so the contribution of the outer region is controlled
for $B_{1}^{\nu,\ell_{\nu}}$. Regarding $B_{2}^{\nu,\ell_{\nu}}$,
we have to control (\ref{eq:AuxiliaryCrap}) with:
\[
B:=\sum_{k\in\mathbb{Z}}P_{k}(\bar{m}_{>\kappa_{\nu}}\Omega_{x,\beta}^{\nu,\ell_{\nu}}\cdot\nabla_{x}\phi_{\nu,\ell_{\nu}}^{\leq k+10}).
\]
Proceeding similarly to the above, we look at the convolution kernel
of $P_{k}P_{k'}$, with $k=k'+O(1)$, and given the spatial support
of $\bar{m}_{>\kappa_{\nu}}\Omega_{x,\beta}^{\nu,\ell_{\nu}}\phi_{\nu,\ell_{\nu}}^{\leq k+10}$,
we get the analogous estimate for $N\in\mathbb{Z}$:
\[
\left\Vert P_{k}B\right\Vert _{L_{x}^{\infty}\left(\{2^{-1}\leq\left|x\right|\leq2\}\right)}\lesssim_{N}\frac{2^{2k}}{(1+2^{k}2^{\kappa_{\nu}})^{N}}\left\Vert \bar{m}_{>\kappa_{\nu}}\Omega_{x,\beta}^{\nu,\ell_{\nu}}\cdot\nabla_{x}\phi_{\nu,\ell_{\nu}}^{\leq k'+10}\right\Vert _{L_{x}^{1}},
\]
so that choosing $N=3$ when $k\geq0$, and $N=1$ if $k<0$ as previously,
yields the control for (\ref{eq:AuxiliaryCrap}):
\[
2^{-\kappa_{\nu}}(\sum_{k\in\mathbb{Z}}2^{-\left|k\right|})\left\Vert \bar{m}_{>\kappa_{\nu}}\Omega_{x,\beta}^{\nu,\ell_{\nu}}\right\Vert _{L_{x}^{2}}\left\Vert \nabla_{x}\phi_{\nu,\ell_{\nu}}\right\Vert _{L_{x}^{2}}\lesssim o(\mathcal{E}),
\]
as desired, and this completes the treatment of the contribution to
(\ref{eq:LastTerms}) of the outer region.

Finally, we need to study the contribution of the interior region
defined by the support of $\bar{m}_{<-10}$, that we note being at
a definite amount of distance from the support of $\bar{m}_{0}$.
First, we remark that we have:
\begin{equation}
\left\Vert \bar{m}_{<-10}\Omega_{x,\beta}^{\nu,\ell_{\nu}}\right\Vert _{H_{x}^{-1}}\longrightarrow0,\label{eq:InnerConnectionDecay}
\end{equation}
and to see this, we start by getting an extension $\varphi_{\nu}$
of $\phi_{\nu,\ell_{\nu}}|_{B_{1}}$, equal to a suitably chosen constant
$c=c(\{\phi_{\nu,\ell_{\nu}}\}_{\nu\in\mathbb{N}})$, such that by
the construction of the sequence of wave maps and the covering in
Lemma \ref{lem:CoveringLemma}, we have $\varphi_{\nu}^{c}:=\varphi_{\nu}-c$
vanishing strongly in supercritical spaces:
\begin{equation}
\left\Vert \varphi_{\nu}^{c}\right\Vert _{H_{x}^{s}}\longrightarrow0,\,\,\, s<1.\label{eq:SuperCriticalDecayBubble}
\end{equation}
To establish (\ref{eq:InnerConnectionDecay}) it is enough to consider
$\widetilde{\varphi}_{\nu}\nabla_{x}\varphi_{\nu}$, where $\widetilde{\varphi}_{\nu}:=\bar{m}_{<-10}\varphi_{\nu}$.
For low frequencies:
\begin{align*}
\left\Vert P_{\leq0}(\widetilde{\varphi}_{\nu}\nabla_{x}\varphi_{\nu})\right\Vert _{L_{x}^{2}}\lesssim & \left\Vert \widetilde{\varphi}_{\nu}\right\Vert _{L_{x}^{\infty}}\left\Vert P_{\leq O(1)}\varphi_{\nu}^{c}\right\Vert _{L_{x}^{2}}\\
 & +\sum_{k_{1},k_{2}\geq O(1):\left|k_{1}-k_{2}\right|\leq O(1)}\left\Vert \nabla_{x}P_{k_{1}}\widetilde{\varphi}_{\nu}\right\Vert _{L_{x}^{2}}\left\Vert P_{k_{2}}\varphi_{\nu}^{c}\right\Vert _{L_{x}^{2}},
\end{align*}
where for the first term we have used (\ref{eq:FiniteBandLess}) to
discard $\nabla_{x}$, and for the second we passed initially to $L_{x}^{1}$
applying (\ref{eq:Bernstein}), and then transferred $\nabla_{x}$
from $\varphi_{\nu}^{c}$ to $\widetilde{\varphi}_{\nu}$ via (\ref{eq:FiniteBandEqual}).
Both items are acceptable by (\ref{eq:SuperCriticalDecayBubble}).
For high frequencies, we apply precisely the same argument, but with
a slightly more refined Littlewood-Paley trichotomy decomposition:
\begin{align*}
2^{-k}\left\Vert P_{k}(\widetilde{\varphi}_{\nu}\nabla_{x}\varphi_{\nu})\right\Vert _{L_{x}^{2}}\lesssim & \left\Vert P_{\leq k-7}\widetilde{\varphi}_{\nu}\right\Vert _{L_{x}^{\infty}}\left\Vert P_{k-3\leq\cdot\leq k+3}\varphi_{\nu}^{c}\right\Vert _{L_{x}^{2}}\\
 & +\left\Vert P_{k-3\leq\cdot\leq k+3}\nabla_{x}\widetilde{\varphi}_{\nu}\right\Vert _{L_{x}^{2}}\left\Vert P_{\leq k-7}\varphi_{\nu}^{c}\right\Vert _{L_{x}^{2}}\\
 & +2^{-\frac{k}{2}}\sum_{k_{1},k_{2}\geq k-6:\left|k_{1}-k_{2}\right|\leq O(1)}\left\Vert \nabla_{x}P_{k_{1}}\widetilde{\varphi}_{\nu}\right\Vert _{L_{x}^{2}}2^{\frac{k_{2}}{2}}\left\Vert P_{k_{2}}\varphi_{\nu}^{c}\right\Vert _{L_{x}^{2}},
\end{align*}
where for the first term we applied (\ref{eq:FiniteBandLess}) and
for the other two we passed first to $L_{x}^{1}$ via (\ref{eq:FiniteBandEqual}),
then used Cauchy-Schwarz, from where for the second term we used (\ref{eq:FiniteBandLess})
for $P_{\leq k-7}\varphi_{\nu}^{c}$ and (\ref{eq:FiniteBandEqual})
for $\widetilde{\varphi}_{\nu}$ transferring $\nabla_{x}$ from one
to the other, whereas for the third term this transfer of $\nabla_{x}$
happened at once via (\ref{eq:FiniteBandEqual}) since $k_{1}=k_{2}+O(1)$,
and then multiplied $P_{k_{2}}\varphi_{\nu}^{c}$ simply by $2^{-k_{2}/2}2^{k_{2}/2}$
which led to the exponential gain $2^{-k/2}$ in front of the sum
since $k_{2}\geq k+O(1)$. Square-summing the above estimate over
$k>0$, and applying discrete Cauchy-Schwarz for the third item, gives
an acceptable bound by (\ref{eq:SuperCriticalDecayBubble}), therefore
we have claim (\ref{eq:InnerConnectionDecay}).

With this understood, we can control the contribution of the inner
region to (\ref{eq:LastTerms}) for the low frequencies. Given any
positive integer $K>0$, we have regarding $B_{1}^{\nu,\ell_{\nu}}$:
\begin{align*}
 & \sum_{k\leq K}\left\Vert P_{k}\nabla_{x}\cdot(\bar{m}_{<-10}\Omega_{x,\beta}^{\nu,\ell_{\nu}}\phi_{\nu,\ell_{\nu}}^{>k+10})\right\Vert _{L_{x}^{1}}\\
 & \lesssim\sum_{k\leq K}2^{k}\sum_{k_{1},k_{2}\geq k+5:\left|k_{1}-k_{2}\right|\leq O(1)}2^{-k_{2}}\left\Vert P_{k_{1}}(\bar{m}_{<-10}\Omega_{x,\beta}^{\nu,\ell_{\nu}})\right\Vert _{L_{x}^{2}}\left\Vert \nabla_{x}\phi_{\nu}^{k_{2}}\right\Vert _{L_{x}^{2}}\\
 & \lesssim\sum_{k\leq O(1)}\left(\sum_{k+5\leq k_{1},k_{2}\leq O(1):\left|k_{1}-k_{2}\right|\leq O(1)}2^{-(k_{2}-k)}\left\Vert P_{k_{1}}(\bar{m}_{<-10}\Omega_{x,\beta}^{\nu,\ell_{\nu}})\right\Vert _{L_{x}^{2}}\left\Vert \nabla_{x}\phi_{\nu}^{k_{2}}\right\Vert _{L_{x}^{2}}\right)\\
 & +2^{K}\sum_{k_{1},k_{2}\geq O(1):\left|k_{1}-k_{2}\right|\leq O(1)}2^{-k_{1}}\left\Vert P_{k_{1}}(\bar{m}_{<-10}\Omega_{x,\beta}^{\nu,\ell_{\nu}})\right\Vert _{L_{x}^{2}}\left\Vert \nabla_{x}\phi_{\nu}^{k_{2}}\right\Vert _{L_{x}^{2}},
\end{align*}
which is $o(\mathcal{E})$ for the first term and $o_{K}(\mathcal{E})$
for the second by (\ref{eq:InnerConnectionDecay}). Analogously, looking
at $B_{2}^{\nu,\ell_{\nu}}$ we get:

\begin{align*}
 & \left\Vert \sum_{k\leq K}P_{k}(\bar{m}_{<-10}\Omega_{x,\beta}^{\nu,\ell_{\nu}}\cdot\nabla_{x}\phi_{\nu,\ell_{\nu}}^{\leq k+10})\right\Vert _{L_{x}^{1}}\\
 & \lesssim\left\Vert \sum_{k\leq K}P_{k}[P_{\leq K+O(1)}(\bar{m}_{<-10}\Omega_{x,\beta}^{\nu})\cdot\nabla_{x}\phi_{\nu,\ell_{\nu}}^{k-3\leq\cdot\leq k+3}]\right\Vert _{L_{x}^{1}}\\
 & +\left\Vert \sum_{k\leq K}P_{k}[P_{k_{1}-3\leq\cdot\leq k_{1}+3}(\bar{m}_{<-10}\Omega_{x,\beta}^{\nu})\cdot\nabla_{x}\phi_{\nu,\ell_{\nu}}^{\leq k_{2}-7}]\right\Vert _{L_{x}^{1}}\\
 & \lesssim\left\Vert P_{\leq K+O(1)}(\bar{m}_{<-10}\Omega_{x,\beta}^{\nu})\right\Vert _{L_{x}^{2}}\cdot\left\Vert (\sum_{k_{2}\leq K}|\nabla_{x}\phi_{\nu,\ell_{\nu}}^{k_{2}-3\leq\cdot\leq k_{2}+3}|)^{\frac{1}{2}}\right\Vert _{L_{x}^{2}}\\
 & +\left\Vert (\sum_{k_{1}\leq K}|P_{k_{1}-3\leq\cdot\leq k_{1}+3}(\bar{m}_{<-10}\Omega_{x,\beta}^{\nu})|^{2})^{\frac{1}{2}}\right\Vert _{L_{x}^{2}}\cdot\left\Vert \sup_{k_{2}\leq K}|\nabla_{x}\phi_{\nu,\ell_{\nu}}^{\leq k_{2}-7}|\right\Vert _{L_{x}^{2}},
\end{align*}
and this is again controlled by $o_{K}(\mathcal{E})$ via (\ref{eq:InnerConnectionDecay}).
Therefore, for both contributions, we can choose a sequence of integers
$K_{\nu}\rightarrow+\infty$, together with decaying constants $\varsigma_{\nu}\downarrow0$,
such that:
\[
\sum_{k\leq K_{\nu}}\left\Vert P_{k}\nabla_{x}\cdot(\bar{m}_{<-10}\Omega_{x,\beta}^{\nu,\ell_{\nu}}\phi_{\nu,\ell_{\nu}}^{>k+10})\right\Vert _{L_{x}^{1}}+\left\Vert \sum_{k\leq K_{\nu}}P_{k}(\bar{m}_{<-10}\Omega_{x,\beta}^{\nu,\ell_{\nu}}\cdot\nabla_{x}\phi_{\nu,\ell_{\nu}}^{\leq k+10})\right\Vert _{L_{x}^{1}}\leq\varsigma_{\nu},
\]
and this yields the decay of slowly growing frequencies for the inner
region, as desired. Note that the cut-off $\bar{m}_{0}$ has not played
any role in the above argument. However, for the high frequencies
$k>K_{\nu}$, having $\bar{m}_{0}$ will be crucial as we are going
to pass by (\ref{eq:AuxiliaryCrap}) as before, first with:
\[
B:=\sum_{k>K_{\nu}}P_{k}\nabla_{x}\cdot(\bar{m}_{<-10}\Omega_{x,\beta}^{\nu,\ell_{\nu}}\phi_{\nu,\ell_{\nu}}^{>k+10}).
\]
Considering the convolution kernel for $\nabla_{x}P_{k}P_{k'}$, with
$k=k'+O(1)$, as previously, we estimate:
\[
\left\Vert P_{k}B\right\Vert _{L_{x}^{\infty}\left(\{2^{-1}\leq\left|x\right|\leq2\}\right)}\lesssim\frac{2^{3k}}{(1+2^{k})^{3}}\left\Vert \bar{m}_{<-10}\Omega_{x,\beta}^{\nu,\ell_{\nu}}\phi_{\nu,\ell_{\nu}}^{>k'+10}\right\Vert _{L_{x}^{1}},
\]
noting the fixed positive distance of the physical support of $\bar{m}_{<-10}\Omega_{x,\beta}^{\nu,\ell_{\nu}}\phi_{\nu,\ell_{\nu}}^{>k+10}$
to the annulus $\{2^{-1}\leq\left|x\right|\leq2\}$. Using this, we
can bound (\ref{eq:AuxiliaryCrap}) in this case by:
\[
\sum_{k>K_{\nu}}2^{-k}\left\Vert \bar{m}_{<-10}\Omega_{x,\beta}^{\nu,\ell_{\nu}}\right\Vert _{L_{x}^{2}}\sum_{k_{1}>k+10}2^{-(k_{1}-k)}\left\Vert \nabla_{x}\phi_{\nu,\ell_{\nu}}^{k_{1}}\right\Vert _{L_{x}^{2}}\lesssim2^{-K_{\nu}}\mathcal{E},
\]
which is certainly acceptable, given that $K_{\nu}\rightarrow+\infty$.
Finally, the last contribution to treat is when:
\[
B:=\sum_{k>K_{\nu}}P_{k}(\bar{m}_{<-10}\Omega_{x,\beta}^{\nu,\ell_{\nu}}\cdot\nabla_{x}\phi_{\nu,\ell_{\nu}}^{\leq k+10}),
\]
in (\ref{eq:AuxiliaryCrap}), and here we proceed in complete analogy
to the above, getting the following estimate:
\[
\left\Vert P_{k}B\right\Vert _{L_{x}^{\infty}\left(\{2^{-1}\leq\left|x\right|\leq2\}\right)}\lesssim\frac{2^{2k}}{(1+2^{k})^{3}}\left\Vert \bar{m}_{<-10}\Omega_{x,\beta}^{\nu,\ell_{\nu}}\cdot\nabla_{x}\phi_{\nu,\ell_{\nu}}^{\leq k'+10}\right\Vert _{L_{x}^{1}},
\]
by looking at the convolution kernel of $P_{k}P_{k'}$, with $k=k'+O(1)$,
and the location of spatial support of $\bar{m}_{<-10}\Omega_{x,\beta}^{\nu,\ell_{\nu}}\cdot\nabla_{x}\phi_{\nu,\ell_{\nu}}^{\leq k+10}$
with respect to the annulus $\{2^{-1}\leq\left|x\right|\leq2\}$.
This in turn, yields the following control for (\ref{eq:AuxiliaryCrap}):
\[
\sum_{k>K_{\nu}}2^{-k}\left\Vert \bar{m}_{<-10}\Omega_{x,\beta}^{\nu,\ell_{\nu}}\right\Vert _{L_{x}^{2}}\left\Vert \nabla_{x}\phi_{\nu,\ell_{\nu}}\right\Vert _{L_{x}^{2}}\lesssim2^{-K_{\nu}}\mathcal{E},
\]
which, as noted above, is permissible. That concludes the treatment
of the contribution of the inner region, and therefore we have obtained
claim (\ref{eq:LastTerms}).

In the end, going back to the physical Littlewood-Paley decomposition
(\ref{eq:PhysicalLPDecomposition}) and expressing the time derivative
$\partial_{t}$ via $X$ and $\partial_{x_{1}}$ using expression
(\ref{eq:ExpressionForX-1}), we have for any $k\in\mathbb{Z}$: 
\[
\left\Vert P_{k}[(\bar{m}_{\leq N_{\nu}}-\bar{m}_{\leq n_{\nu}-1})\nabla_{t,x}\phi_{\nu}](0)\right\Vert _{L_{x}^{2}}\lesssim\sum_{\ell\in\mathbb{Z}}2^{-\left|k+\ell\right|}\varepsilon_{\nu}+\sigma_{\nu}+\iota_{\nu}+o(\mathcal{E})\longrightarrow0,
\]
where the first sum arises from the low frequencies (\ref{eq:LowFrequencies})
and the regular part involving spatial derivatives falling on the
cut-offs from (\ref{eq:Commutator1Control}) and (\ref{eq:Error1}),
the second term comes from errors having good time-like control (\ref{eq:Error2}),
the third one arise from treating the higher-order time like derivative
in (\ref{eq:Error3}), and finally the last term is due to the perturbative
$\dot{B}_{\infty}^{-1,2}$ estimate of the non-linearity for the wave
maps equation at high frequency (\ref{eq:NonLinearDoublePrim}), combined
with (\ref{eq:L1controlDual}) and (\ref{eq:LastTerms}).

Lemma \ref{lem:Besov-control.} is proved.
\end{proof}
We are now at the concluding stage of the proof of Theorem \ref{thm:Main},
for which, going back to the weak bubble tree decomposition (\ref{eq:DecomWithNecks}),
we must show that the energy of the necks $\mathcal{N}_{i,\nu}$ is
asymptotically vanishing as $\nu\rightarrow+\infty$. Recall that
those are provided with corresponding neck domains, that is the conformally
degeneration annuli from (\ref{eq:NeckDomains}), so that setting:
\[
\phi_{\nu,x_{i,\nu}^{k}}(t,x):=\phi_{i,\nu}(\lambda_{\mathrm{min},\nu}t,x_{i,\nu}^{k}+\lambda_{\mathrm{min},\nu}x),
\]
we can apply Lemma \ref{lem:Besov-control.}, by (\ref{eq:NeckProperty})
and (\ref{eq:GoodSliceProps}), to write:
\[
\nabla_{t,x}\phi_{\nu,x_{i,\nu}^{k}}=\Upsilon_{\nu,x_{i,\nu}^{k}}\,\,\,\mathrm{on}\,\,\,[-1,1]\times(B_{\lambda_{\mathrm{min},\nu}^{-1}R_{i,\nu}^{k}}\setminus B_{\lambda_{\mathrm{min},\nu}^{-1}r_{i,\nu}^{k}}),
\]
where $\Upsilon_{\nu,x_{i,\nu}^{k}}$ is supported on $[-1,1]\times(B_{2\lambda_{\mathrm{min},\nu}^{-1}R_{i,\nu}^{k}}\setminus B_{2^{-1}\lambda_{\mathrm{min},\nu}^{-1}r_{i,\nu}^{k}})$
with
\[
\left\Vert \Upsilon_{\nu,x_{i,\nu}^{k}}\right\Vert _{L_{t}^{\infty}(L_{x}^{2})[-1,1]}\lesssim1,
\]
and satisfying the decay:
\[
\sup_{k\in\mathbb{Z}}\left\Vert P_{k}\Upsilon_{\nu,x_{i,\nu}^{k}}(0)\right\Vert _{L_{x}^{2}}\longrightarrow0.
\]

Recalling (\ref{eq:GoodSliceProps}), we also have: 
\[
\left\Vert \Theta_{\nu,x_{i,\nu}^{k}}(0)\right\Vert _{L_{x}^{2}}\longrightarrow0,
\]
\[
\mathrm{where}\,\,\,\Theta_{\nu,x_{i,\nu}^{k}}(t,x):=\lambda_{\mathrm{min},\nu}\Theta_{i,\nu}(\lambda_{\mathrm{min},\nu}t,x_{i,\nu}^{k}+\lambda_{\mathrm{min},\nu}x).
\]
together with:
\[
\sum_{k\in\mathbb{Z}}\left\Vert P_{k}\Xi_{\nu,x_{i,\nu}^{k}}(0)\right\Vert _{L_{x}^{2}}\lesssim1,
\]
\[
\mathrm{where}\,\,\,\Xi_{\nu,x_{i,\nu}^{k}}(t,x):=\lambda_{\mathrm{min},\nu}\Xi_{i,\nu}(\lambda_{\mathrm{min},\nu}t,x_{i,\nu}^{k}+\lambda_{\mathrm{min},\nu}x).
\]

From there, we can estimate the energy at time $t=0$ on a neck region
by:
\begin{align*}
\left\Vert \nabla_{t,x}\phi_{\nu,x_{i,\nu}^{k}}(0)\right\Vert _{L_{x}^{2}(B_{\lambda_{\mathrm{min},\nu}^{-1}R_{i,\nu}^{k}}\setminus B_{\lambda_{\mathrm{min},\nu}^{-1}r_{i,\nu}^{k}})}^{2}\lesssim & \left|\int_{\mathbb{R}^{2}}\Upsilon_{\nu,x_{i,\nu}^{k}}(0)\Xi_{\nu,x_{i,\nu}^{k}}(0)dx\right|\\
 & +\left|\int_{\mathbb{R}^{2}}\Upsilon_{\nu,x_{i,\nu}^{k}}(0)\Theta_{\nu,x_{i,\nu}^{k}}(0)dx\right|+o(1),
\end{align*}
which we bound by:
\[
(\sup_{k\in\mathbb{Z}}\left\Vert P_{k}\Upsilon_{\nu,x_{i,\nu}^{k}}(0)\right\Vert _{L_{x}^{2}})\sum_{k\in\mathbb{Z}}\left\Vert P_{k}\Xi_{\nu,x_{i,\nu}^{k}}(0)\right\Vert _{L_{x}^{2}}+\left\Vert \Upsilon_{\nu,x_{i,\nu}^{k}}(0)\right\Vert _{L_{x}^{2}}\left\Vert \Theta_{\nu,x_{i,\nu}^{k}}(0)\right\Vert _{L_{x}^{2}}+o(1),
\]
and by the previous estimates this tends to $0$ as $\nu\rightarrow+\infty$.
Theorem \ref{thm:Main} is proved.

\hphantom{}\hphantom{}

\emph{Acknowledgements. }It is a pleasure to thank Dominic Joyce,
Andrew Lawrie, Luc Nguyen, Sung-Jin Oh, Peter Topping and Qian Wang
for very valuable and interesting discussions, as well as encouragements.
This work is part of author's DPhil thesis at the University of Oxford,
kindly supported by an EPSRC Research Studentship.

\hphantom{}\hphantom{}

\end{document}